\newcommand\Modified{Jan 21, 2015 }

\documentclass[11pt]{amsart}

\usepackage{amsfonts, amsmath, amsthm, amscd, amssymb, latexsym, mathrsfs, mathtools, slashed, stmaryrd, verbatim, wasysym }
\usepackage[all]{xy}
\usepackage{hyperref}
\usepackage{bm}
\usepackage{fixmath}

\newcommand\datver[1]{\def\datverp%
{\par\boxed{\boxed{\text{Version: #1; Run: \today}}}}}

\newtheorem{theorem}{Theorem}
\newtheorem{definition}[theorem]{Definition}
\newtheorem{corollary}[theorem]{Corollary}
\newtheorem{example}[theorem]{Example}
\newtheorem{lemma}[theorem]{Lemma}

\newtheorem{proposition}[theorem]{Proposition}
\newtheorem{remark}[theorem]{Remark}

 \usepackage{color}
\definecolor{darkgreen}{cmyk}{1,0,1,.2}
\definecolor{m}{rgb}{1,0.1,1}

\renewcommand{\hat}[1]{\widehat{#1}}

\renewcommand{\tilde}{\widetilde}

\newcommand\eps\varepsilon

\newcommand\pa{\partial}

\newcommand \R {\mathbb{R}}
\newcommand \C {\mathbb{C}}

\newcommand\Diff{\operatorname{Diff}}

\newcommand\id{\operatorname{id}}
\newcommand\Id{\operatorname{Id}}

\newcommand{\Ind}{\operatorname{Ind}}
\newcommand{\Ker}{\operatorname{Ker}}

\newcommand{\calU}{{\mathcal U}}

\newcommand{\frakS}{{\mathfrak S}}

\newcommand{\ovl}{\overline}

\DeclareMathAlphabet{\mathpzc}{OT1}{pzc}{m}{it}

\newcommand\paperintro%
        {%
         }
\newcommand\paperbody%
        {%
         }

\newcommand\bbC{\mathbb{C}}

\newcommand\bbK{\mathbb{K}}

\def\RR{\bbR}
\newcommand\bbR{\mathbb{R}}

\newcommand\bbZ{\mathbb{Z}}

\newcommand\N{\mathbb{N}}

\datver{\Modified}

\begin{document}

\title{Singular spaces, groupoids and metrics of positive scalar curvature}

\author{Paolo Piazza}
\address{Dipartimento di Matematica, Sapienza Universit\`a di Roma}
\email{piazza@mat.uniroma1.it}
\author{Vito Felice Zenobi}
\address{Georg-August-Universit\"at
G\"ottingen, Germany}
\email{zenobivf@gmail.com}




\maketitle
\begin{abstract}

We define and study, under suitable assumptions, the fundamental class, the index class and the rho class
of a spin Dirac operator on the regular part of a spin stratified pseudomanifold.
 More singular structures, such as singular foliations,
are also  treated. We employ groupoid techniques
 in a crucial way; however, an effort has been made in order to make this article
 accessible to readers with only a minimal knowledge of groupoids. 
 Finally, whenever appropriate,
 a comparison between
 classical microlocal methods and groupoids methods has been provided.
 
 \end{abstract}

\tableofcontents

\paperintro
\section{Introduction}

\subsection{Smooth closed manifolds and the Higson-Roe analytic surgery sequence}$\;$\\
Let $(X,g)$ be a spin riemannian manifold with fundamental group $\pi_1 (X)=:\Gamma$. We fix a spin structure
$s$.
We denote by $X_\Gamma$ its universal cover. We consider $\slashed{D}_g$, the Dirac operator 
associated to $g$ and $s$, and $\slashed{D}^\Gamma_g$ its $\Gamma$-equivariant lift to $X_\Gamma$.
Thanks to the work of Atiyah, Kasparov, Mishchenko and many others we know that we can attach to
these data two  
K-theory classes:
\begin{itemize}
\item the fundamental class $[\slashed{D}_g]\in K_* (X)$;
\item the index class $\Ind (\slashed{D}_g^\Gamma)\in K_* (C^*_r \Gamma)$
\end{itemize}
with $*=\dim X$.
The index class is particularly interesting in connection with the existence of metrics 
of positive scalar curvature; indeed, on the one hand 
the class $\Ind (\slashed{D}_g^\Gamma)\in K_* (C^*_r \Gamma)$
is independent of the choice of $g$ and,  on the other hand, $\Ind (\slashed{D}_g^\Gamma)=0$
if $g$ has positive scalar curvature, a result 
based on a generalization to the $C^*$-context of  the well-known Lichnerowicz 
argument. See for example \cite[Theorem 1.1]{RosI} for the details.
This means that the class $\Ind (\slashed{D}_g^\Gamma)\in K_* (C^*_r \Gamma)$
is an {\it obstruction} to the existence of {\it a} metric of positive scalar curvature on $X$ \footnote{One can also consider 
a more refined object, the class defined by the real Dirac operator in 
$KO_* (C^*_{r,\RR} (\Gamma))$; however, in this paper we shall bound ourselves to the complex case.}.

There are in fact many equivalent realizations of these classes. One that is particularly elegant
and that we would like to single out
at this point is 
due to Nigel Higson and John Roe; it stems from a long exact sequence of $K$-theory groups
for $C^*$-algebras.
 The $C^*$-algebras in question are $C^*(X_{\Gamma})^{\Gamma}$ and $D^*(X_{\Gamma})^{\Gamma}$, obtained as the closures of the $\Gamma$-equivariant operators on $X_{\Gamma}$ 
 that satisfty a  finite propagation property and, in addition, are respectively  `locally compact' or  `pseudolocal'.
 The former $C^*$-algebra is an ideal in the latter; thus we have a short exact sequence of $C^*$-algebras
%
	$0\rightarrow  C^*(X_\Gamma)^\Gamma\rightarrow 
	D^*(X_\Gamma)^\Gamma\rightarrow 
	D^*(X_\Gamma)^\Gamma/C^*(X_\Gamma)^\Gamma\rightarrow 0$
%
which gives rise to a long exact sequence in K-theory known as the {\em analytic surgery sequence} of Higson and Roe:
\begin{equation}\label{eq:HR-basic}
	\cdots\rightarrow K_{m+1}  (C^*(X_\Gamma)^\Gamma)
	\rightarrow K_{m+1}  (D^*(X_\Gamma)^\Gamma)
	\rightarrow  K_{m+1}( D^*(X_\Gamma)^\Gamma/C^*(X_\Gamma)^\Gamma)
	\rightarrow K_m (C^*(X_\Gamma)^\Gamma)
	\rightarrow \cdots
\end{equation}
Crucial to the development of the theory are the canonical isomorphisms
\begin{equation}\label{iso-K}
	K_{*+1}  (D^*(X_\Gamma)^\Gamma/C^*(X_\Gamma)^\Gamma)= K_{*} (X)
	\quad\text{and}\quad  
	K_{*}  (C^*(X_\Gamma)^\Gamma)=K_* (C^*_r \Gamma)\,;
\end{equation}
%
using these, the analytic surgery sequence 
can be rewritten as 
\begin{equation}\label{eq:HR}
	\cdots\rightarrow K_{m+1}  (C^*_r\Gamma)
	\rightarrow K_{m+1}  (D^*(X_\Gamma)^\Gamma)
	\rightarrow  K_{m}(X)
	\rightarrow K_m (C^*_r\Gamma)
	\rightarrow \cdots
\end{equation}
By employing  ellipticity and the finite propagation speed property for the wave operator associated to
$\slashed{D}^\Gamma_g$, one can define an element in $K_{*+1}  (D^*(X_\Gamma)^\Gamma/C^*(X_\Gamma)^\Gamma)$ and this element corresponds precisely to the fundamental class $[\slashed{D}_g]$ under the first isomorphism
in \eqref{iso-K}. The index class is then realized
as
\begin{equation}\label{delta-fc-index}
\Ind (\slashed{D}^\Gamma_g)=\delta [\slashed{D}_g]\in K_{*}  (C^*(X_\Gamma)^\Gamma)=K_* (C^*_r \Gamma)\,,
\end{equation}
with $\delta: K_{*+1}  (D^*(X_\Gamma)^\Gamma/C^*(X_\Gamma)^\Gamma)\rightarrow 	K_{*}  (C^*(X_\Gamma)^\Gamma)$ the connecting homomorphism in the analytic surgery sequence.

Assume now  that $g$ is of positive scalar curvature; then we know that  
$\delta [\slashed{D}_g]=0$ in  $K_{*}  (C^*(X_\Gamma)^\Gamma)$; this means that
there exists a {\it lift} of $ [\slashed{D}_g]$ to $K_{m+1}  (D^*(X_\Gamma)^\Gamma)$. In fact, 
 since the Lichnerowicz argument shows that the operator
$\slashed{D}^\Gamma_g$ is $L^2$-invertible, it is easy to see that there exists a natural lift;
we call this lift the {\it rho class of the positive scalar curvature metric $g$} and denote it by
$\rho(g)\in K_{*+1}  (D^*(X_\Gamma)^\Gamma)$. This is a (higher) {\it secondary invariant} of $\slashed{D}^\Gamma_g$.\\
Thus, if $g$ is of positive scalar curvature
$$
\xymatrix{ K_{*+1}  (D^*(X_\Gamma)^\Gamma)\ar[r] & K_{*+1}  (D^*(X_\Gamma)^\Gamma/C^*(X_\Gamma)^\Gamma)\ar[r] & K_{*}  (C^*(X_\Gamma)^\Gamma) \\
	\rho(g) & [\slashed{D}_g]\ar@/_1pc/@{.>}[l]\ar[r]^\delta& 0}
$$

\medskip
\noindent
{\bf Summarizing}, given an arbitrary metric $g$ we can define
\begin{equation}\label{delta-fc-index-bis}
 [\slashed{D}_g]\in K_{*} (X)\;\;\;\;\text{and}\;\;\;\;
 \Ind (\slashed{D}^\Gamma_g)=\delta [\slashed{D}_g]\in K_* (C^*_r \Gamma)
\end{equation}
and if $g$ is of positive scalar curvature metric  then we have
\begin{equation}\label{delta-fc-index}
\rho(g)\in K_{*+1}  (D^*(X_\Gamma)^\Gamma)
\end{equation}
as a lift of $[\slashed{D}_g]$ in \eqref{eq:HR}.
It should be noticed that if $g$ is of  positive scalar curvature then the classic APS rho invariant and the 
Cheeger-Gromov rho invariant can be reobtained from $\rho (g)$ by applying suitable traces to 
$ K_{*+1}  (D^*(X_\Gamma)^\Gamma)
$. This is work
of Higson-Roe for the former invariant, see \cite{higson-roe4}, and Benameur-Roy for the latter, see \cite{BR}.
Conjecturally, the  rho class defined above should also give Lott's higher rho invariant $\rho_{{\rm Lott}} (g)$,
see \cite{LottII},
an element in the delocalized noncommutative de Rham homology of the Connes-Moscovici algebra of $\Gamma$.

Building on fundamental work of Higson and Roe 
for the signature operator, see  \cite{higson-roeI}, \cite{higson-roeII}, \cite{higson-roeIII}, it was proved in \cite{PS-Stolz} that these three classes can be employed in order
to map Stolz' surgery sequence for positive scalar curvature metrics to the analytic surgery sequence 
\eqref{eq:HR}.
In particular, one can prove that the rho class gives well-defined maps
\begin{equation}\label{rho-map}
\rho:\pi_0 (\mathcal{R}^+ (X))\to K_{\dim X+1}  (D^*(X_\Gamma)^\Gamma)
\;\;\;\text{and}\;\; \;
\rho:\widetilde{\pi}_0 (\mathcal{R}^+ (X))\to K_{\dim X+1}  (D^*(X_\Gamma)^\Gamma)
\end{equation}
with $\mathcal{R}^+ (X)$ denoting the space of positive scalar curvature metrics on $X$ and 
$\widetilde{\pi}_0 (\mathcal{R}^+ (X))$ the associated set of concordance classes.

Crucial to the proof of the well-definedness of \eqref{rho-map} is the {\it delocalized Atiyah-Patodi-Singer
index theorem} proved in  \cite{PS-Stolz}, a result that establishes a link between the index class of a spin
riemannian manifold with boundary, with the boundary having positive scalar curvature, and the rho class of the boundary metric.

An alternative treatment of these results was later given by Xie and Yu, see \cite{Xie-Yu}, using Yu's localization
algebra. 

We cannot  end this subsection without  mentioning the very interesting geometric applications of these various rho invariants
 to the world of positive scalar curvature metrics.
Without entering into the details we list here a number of articles that use (higher) rho invariants in order to
{\it distinguish} metrics of positive scalar curvature up to bordism, concordance or isotopy: Botvinnik-Gilkey \cite{BoGi}, Leichtnam-Piazza \cite{LPPSC}, Azzali-Wahl \cite{AW}, Benameur-Mathai \cite{BenMat}, Piazza-Schick \cite{PS2}, Zeidler \cite{zeidler-jtop}, Weinberger-Yu
\cite{WY-finite}, Xie-Yu \cite{Xie-Yu-moduli}, Barcenas-Zeidler \cite{Barcenas-Zeidler}, Xie-Yu-Zeidler \cite{XYZ}, Zenobi \cite{ZenJTA}.

\subsection{Enter  groupoids}$\;$\\
There is yet another approach to the three classes 
$$[\slashed{D}_g]\,, \quad\Ind (\slashed{D}^\Gamma_g)\quad\text{and} \quad\rho (g)$$ and this alternative approach involves Lie groupoids in a fundamental way. It was developed by Zenobi in his Ph.D. thesis, see  \cite{zenobi-thesis} \cite{Zenobi:Ad}.
The cornerstone of this method is the concept of adiabatic deformation
of a Lie groupoid, a concept first introduced by Alain Connes in \cite{Co} and then developed
by many others in different contexts; this will be treated later in the paper, but we want to give  at least  the basic idea
now.
First, a couple of preliminary remarks:\\
(i)  the quantization of a symbol defines a Poincar\'e isomorphism 
\begin{equation}\label{PD-K}
PD: K^* (T^*X) \rightarrow K_* (X)\, ,
\end{equation}
see \cite{Kasp-Conspectus,CS}.
Thus we can equivalently think of the fundamental class $[\slashed{D}_g]$ as an element in 
$K^* (T^*X)$ which is in turn isomorphic to $K_* (C_0 (T^*X))$.\\
(ii) The index class $\Ind (\slashed{D}^\Gamma_g)$ is defined using a $\Gamma$-equivariant parametrix 
with remainders that are given by $\Gamma$-invariant smooth kernels on 
$X_\Gamma \times X_\Gamma$ with compact support in $
X_\Gamma \times X_\Gamma/\Gamma:= X_\Gamma \times_\Gamma X_\Gamma$. Thus we are led to consider the groupoid $G \rightrightarrows X$, 
$$G:=X_\Gamma \times_\Gamma X_\Gamma\rightrightarrows X$$
with source and range maps given by $s[x,y]=p(y)$ and $r[x,y]=p(x)$
where $p:X_\Gamma\to X$ is the projection of the Galois covering. We shall give the necessary definitions later.
Consider now the {\it adiabatic deformation} of $G\equiv X_\Gamma \times_\Gamma X_\Gamma\rightrightarrows X$.
This is the groupoid $G^{[0,1]}_{ad}$
 over $X$ given, as a set, by
$$ TX\times\{0\}\cup X_\Gamma \times_\Gamma X_\Gamma\times (0,1] \rightrightarrows X$$
with range and source maps given at $0$ by the projection $\pi: TX\to X$
 and by $r$ and $s$ introduced above on the rest of the deformation \footnote{Each vector bundle $E\xrightarrow{\pi} X$ gives in a natural way a groupoid, with both the range and the source map equal to $\pi$; the groupoid law is then the vector sum along the fibers of $E$.}.
One can define a natural smooth structure on the set $ TX\times\{0\}\cup X_\Gamma \times_\Gamma X_\Gamma\times (0,1]$. We shall in fact be interested in the groupoid 
$$G^{[0,1)}_{ad}:= TX\times\{0\}\cup X_\Gamma \times_\Gamma X_\Gamma\times (0,1) \rightrightarrows X$$
Now, as we shall explain, one can attach a $C^*$-algebra to a groupoid; the evaluation at $0$ for
$G^{[0,1)}_{ad}$  then induces 
a short exact sequence of $C^*$-algebras \\
$$0\rightarrow  C^*_r ( X_\Gamma \times_\Gamma X_\Gamma)\otimes C_0 (0,1) \rightarrow C^* ( G^{[0,1)}_{ad})\xrightarrow{{\rm ev}_0} C^*_r (TX)
\rightarrow 0$$
with associated long exact sequence in K-theory given by 
\begin{equation}\label{ad-e-s}
\begin{split}
\cdots \rightarrow K_* (C^*_r ( X_\Gamma \times_\Gamma X_\Gamma)\otimes C_0 (0,1)) \rightarrow K_* ( C^*_r ( G^{[0,1)}_{ad}) )\xrightarrow{({\rm ev}_0)_*} K_* (C^*_r (TX))\\\xrightarrow{\delta_{ad}}K_{*+1} (C^*_r ( X_\Gamma \times_\Gamma X_\Gamma)\otimes C_0 (0,1))
\rightarrow \cdots
\end{split}\end{equation}
In the above long exact sequence  $TX\xrightarrow{\pi} X$ is viewed as a groupoid, $TX\rightrightarrows X$, with $r=\pi=s$
and $C^* (TX)$ is the $C^*$-algebra of this groupoid. It is easy to see that, as $C^*$-algebras, $C_r^* (TX)\simeq C_0 (T^* X)$; we obtain in this
way  natural identifications
$K_* (C_r^* (TX))= K_* (C_0 (T^* X))=K^* (T^* X)$.
It is also true, and easy to prove,   that if the action of $\Gamma$ is cocompact then $K_* (C^*_r ( X_\Gamma \times_\Gamma X_\Gamma))$
is naturally isomorphic to  $K_* (C^*_r \Gamma)$ so that 
there exists a Bott isomorphism
$$\beta:  K_{*+1} (C^*_r ( X_\Gamma \times_\Gamma X_\Gamma)\otimes C_0 (0,1))\to K_* (C^*_r ( X_\Gamma \times_\Gamma X_\Gamma))=K_* (C^*_r \Gamma).$$
Let us now go back to our spin Dirac operator: we have already stated that its principal symbol defines a class $[\sigma_{{\rm pr}} (\slashed{D}_g)]\in 
K^* (T^* X)=K_* (C_r^* (TX))$ with the property that 
\begin{equation}\label{pd-intro}
PD ([\sigma_{{\rm pr}} (\slashed{D}_g)])= [\slashed{D}_g]\in K_* (X)\,.
\end{equation}
Consider the adiabatic index homomorphism, defined as the composition of $\delta_{ad}$ with the Bott isomorphism $\beta$:
\begin{equation}\label{adiab-index-0}
\Ind^{{\rm ad}} := \beta\circ \delta_{ad}:  K_* (C_r^* (TX))\rightarrow K_{\dim X } (C^*_r ( X_\Gamma \times_\Gamma X_\Gamma))\, .
\end{equation}
We define the adiabatic index class of $\slashed{D}^\Gamma_g$ as
\begin{equation}\label{adiab-index}
\Ind^{{\rm ad}} (\slashed{D}^\Gamma_g):= \beta\circ \delta_{ad} [\sigma_{{\rm pr}} (\slashed{D}_g)]\in K_{\dim X } (C^*_r ( X_\Gamma \times_\Gamma X_\Gamma))\, .
\end{equation}
It can be proved that the
\begin{equation}\label{adiab-index-intro}
\Ind^{{\rm ad}} (\slashed{D}^\Gamma_g)=\Ind  (\slashed{D}^\Gamma_g)\in K_{\dim X} (C^*_r \Gamma)
\end{equation}
under the identification $K_* (C^*_r ( X_\Gamma \times_\Gamma X_\Gamma))=K_* (C^*_r \Gamma)$.
In fact, more is true, in that we have a commutative diagram
\begin{equation}\label{com-dia-intro}
\xymatrix{K_* (C^* (TX))\ar[r]^(.45) { \mathrm{Ind}^{ad}  }\ar[d]^{PD}& K_{*} (C_r^* ( X_\Gamma \times_\Gamma X_\Gamma))\ar[d]^{\cong}\\
	K_* (X)\ar[r]^(.5){\mu}&  K_{*} (C^*_r \Gamma)}
\end{equation}
where $\mu$ is the assembly map and 
with the first vertical arrow an isomorphism (then \eqref{adiab-index-intro} follows from \eqref{pd-intro}).
Assume now that $g$ is of positive scalar curvature; then, from  \eqref{adiab-index-intro} or, in fact, directly, we have that
 $\mathrm{Ind}^{ad}  (\slashed{D}^\Gamma_g)=0$ in $K_{\dim X} (C^*_r( X_\Gamma \times_\Gamma X_\Gamma))$ and so we finally see that we can define the (adiabatic) rho class as a certain  natural lift of 
 $[\sigma_{{\rm pr}} (\slashed{D}_g)]\in K_* (C_0 (TX))$
 to $K_* ( C^*_r ( G^{[0,1)}_{ad}))$: 
 \begin{equation}\label{rho-intro}
 \rho^{{\rm ad}} (g)\in K_* (  C^*_r ( G^{[0,1)}_{ad}))\equiv K_* (C^*_r (TX\times\{0\}\cup X_\Gamma \times_\Gamma X_\Gamma\times (0,1) ))
\end{equation}
Thus, in this case,
$$
\xymatrix{ K_* (  C^*_r ( G^{[0,1)}_{ad}))\ar[r]^{\mathrm{ev}_0} & K_*(C^*_r(TX))\ar[r]^(.30){ \Ind^{\mathrm{ad}}} & K_{\dim X} (C^*_r ( X_\Gamma \times_\Gamma X_\Gamma))=K_{*}  (C^*_r (\Gamma))\\
	\rho^{{\rm ad}} (g) & [\sigma_{\mathrm{pr}}(\slashed{D}_g)]\ar@/_1pc/@{.>}[l]\ar[r] & 0 }
$$
\medskip
\noindent
Let us summarize the situation:\\
by considering the adiabatic deformation of the groupoid $X_\Gamma \times_\Gamma X_\Gamma\rightrightarrows X$ we have obtained the {\it adiabatic sequence} 
 \begin{equation*}
\cdots\rightarrow K_* ( C^*_r ( G^{[0,1)}_{ad}) )\xrightarrow{({\rm ev}_0)_*} K_* (C_r^* (TX))\\\xrightarrow{\mathrm{Ind}^{\mathrm{ad}}} K_{*} (C^*_r ( X_\Gamma \times_\Gamma X_\Gamma))
\rightarrow \cdots
\end{equation*}
The last two groups are the receptacle of (alternative but equivalent versions 
of) the fundamental class and of the index class respectively,
viz.  $[\sigma_{{\rm pr}} (\slashed{D}_g)]$ and $\Ind^{{\rm ad}} (\slashed{D}^\Gamma_g)$,
whereas the first group is the receptacle of the rho class $\rho^{{\rm ad}}(g)$ whenever $g$ is of positive scalar curvature.
This means that  we can consider the three K-theory classes in the Higson-Roe descripiton
\begin{equation}\label{delta-fc-index-ter}
 [\slashed{D}_g]\in K_{*} (X)\,,\quad
 \Ind (\slashed{D}^\Gamma_g)\in K_* (C^*_r \Gamma)\,,\quad
\rho(g)\in K_{*+1}  (D^*(X_\Gamma)^\Gamma)
\end{equation}
or, alternatevely, the three  classes
\begin{equation}\label{delta-fc-index-ter}
 [\sigma_{{\rm pr}}(\slashed{D}_g])\in K_{*} (C_0 (T^*X))\,,\;\;
 \Ind^{{\rm ad}} (\slashed{D}^\Gamma_g)\in K_{*} (C^*_r ( X_\Gamma \times_\Gamma X_\Gamma))
\,,\;\;
\rho^{{\rm ad}} (g)\in K_{*}  ( C^*_r ( G^{[0,1)}_{ad}) )
\end{equation}
in the adiabatic deformation picture. It should also be added at this point that a purely groupoid-proof of the delocalized APS index theorem
is possible, giving, again, the well-definedness of \eqref{rho-map} or, better, the well-definedness 
of the analogue
maps
\begin{equation}\label{rho-map-ad}
\rho^{{\rm ad}}:\pi_0 (\mathcal{R}^+ (X))\to K_{\dim X+1}   ( C^*_r ( G^{[0,1)}_{ad}) )\,,\quad 
\rho^{{\rm ad}}:\widetilde{\pi}_0 (\mathcal{R}^+ (X))\to K_{\dim X+1}  (  C^*_r ( G^{[0,1)}_{ad}))
\end{equation}
It is important to notice that the
$C^*$-algebras entering into the adiabatic treatment are much better behaved than their coarse counterparts;
for example $C^*_r (G^{[0,1)}_{ad})$ is a separable $C^*$-algebra whereas $D^* (X_\Gamma)^\Gamma$ is not.
This is important whenever one wants to pair the rho class with suitable cyclic cocycles.

\smallskip
\noindent
We now come to a {\bf fundamental question}:

\medskip
\noindent
{\it did we gain anything by bringing groupoids into the picture ?}

\medskip
\noindent
The answer is {\it yes}, as we shall argue in the next subsection.

\subsection{Why groupoids ?} $\;$\\
The positive answer to our fundamental question 
 is based
on the following very {\bf general principle}:

\medskip
\noindent
 {\it groupoid techniques usually apply to large classes of  groupoids;  they are not specific 
to a particular example}.

\medskip
\noindent
And indeed, in \cite{zenobi-thesis} \cite{Zenobi:Ad}, the above construction  is applied to the adiabatic deformation
of {\it any} Lie groupoid, not necessarily the groupoid $X_\Gamma \times_\Gamma X_\Gamma\rightrightarrows X$.
In particular, a very general delocalized APS index theorem is proved using groupoid techniques.
This gives, for example, the rho class of a metric on a foliation $(X,\mathcal{F})$ which has positive scalar curvature
along the leaves. The whole present article can be seen as another manifestation of this principle: we shall define
the rho class of a stratified space and of a singular foliation (under a positive scalar curvature 
assumption) following the above procedure for suitable groupoids; needless to say, these groupoids 
depend on the specific geometries that  will be considered.\\
One might rightly wonder  why this general principle is true. We single out the following reasons:

\begin{itemize}
\item groupoids have very strong functoriality properties; for example, they can be restricted and  pulled-back;
\item to any Lie groupoid we can associate a $C^*$-algebra; we can use the functoriality properties of groupoids
in order to
obtain natural $C^*$-algebras homomorphisms;
\item  to any Lie groupoid we can associate a psedodifferential algebra; 
\item the combination of these last two items often produce interesting K-theory classes that are defined
in great generality.
\end{itemize}

\subsection{Singular structures}
In this work we shall see these ideas  applied to the following three geometric situations. First, we are given a
pseudomanifold ${}^{\textrm{S}} X$ of depth $k$ with a Thom-Mather stratification. Associated to 
${}^{\textrm{S}} X$ we have its resolution, $X$, a manifold with fibered corners. See \cite{ALMP:Witt}
and also \cite{DLR}. 
For this introduction we fix our attention on Thom-Mather pseudomanifolds of depth 1. This means, in particular,
that we are given  a locally compact metrizable space ${}^{\textrm{S}} X$ such that
\begin{itemize}
\item ${}^{\textrm{S}} X$ is the union of two smooth manifolds ${}^{\textrm{S}} X^{{\rm reg}}$ and $Y$, the two strata;
\item  ${}^{\textrm{S}} X^{{\rm reg}}$ is open and dense in ${}^{\textrm{S}} X$;
\item $Y$, the singular stratum, is a smooth compact manifold;
\item there is an open neighbourhood $T_Y$ of $Y$ in ${}^{\textrm{S}} X$, 
        with a continuous retraction $\pi\colon T_Y\to Y$ and a continuous map $\rho\colon T_Y\to [0,+\infty[$ such that 
        $\rho^{-1}(0)= Y$. 
       \item there exists a smooth compact manifold $L_Y$, called the link associated to $Y$,
       such that $T_Y$ is a fiber bundle over $Y$ with fiber $C(L_Y)$, the cone over $L_Y$.
\end{itemize}
The resolution $X$ associated to   ${}^{\textrm{S}} X$ is defined
as $X:={}^{\textrm{S}} X\setminus \rho^{-1}([0,1))$. This is a manifold with boundary
 $H:=\rho^{-1}(1)$ and  $H$ is the total space of a fibration  $L_Y \to H \xrightarrow{\phi} Y$ with base $Y$ and with 
 typical fiber the link $L_Y$.  
Thus the resolution of a depth-1 Thom-Mather pseudomanifold ${}^{\textrm{S}} X$
is a manifold with fibered boundary $X$; there is a natural identification between the interior of $X$, $\mathring{X}$, and 
${}^{\textrm{S}} X^{{\rm reg}}$.

We shall introduce a metric structure on our singular space by endowing ${}^{\textrm{S}} X^{{\rm reg}}$, or, equivalently
$\mathring{X}$, with a 
riemannian metric $g$. There are many different types of metrics that can be considered; we shall consider 
(rigid) fibered boundary metrics, firstly studied by Mazzeo and Melrose in \cite{MMPHI}. These are metrics that in a tubular neighbourhood of the singularity $Y$, or equivalently, in a collar neighborhood of  the boundary 
of $X$, can be written  in the following special form:
\begin{equation}\label{fb-metric}
\frac{dx^2}{x^4} + \frac{\phi^* g_Y}{x^2} + h_{H/Y}
\end{equation}
with $x$ a boundary defining function for $\pa X$ and $h_{H/Y}$ a fiber metric on the boundary fibration  $L_Y \to H\equiv\pa X \xrightarrow{\phi} Y$. The vector fields dual to the 
metric \eqref{fb-metric} are given by
$$x^2 \pa_x\,,\quad x\pa_{y_1},\dots,x\pa_{y_k}\,,\quad \pa_{z_1},\dots,\pa_{z_\ell}$$
with $(y_1,\dots,y_k)$ coordinates on $Y$ and $(z_1,\dots,z_\ell)$ coordinates along the link $L_Y$; these vector fields
generate locally the Lie algebra of vector fields on $X$
\begin{equation}\label{fibered-corners-vf-intro}
\mathcal{V}_{{\rm \Phi}} (X)=\{\xi\in \mathcal{V}_b (X)\;\;|\;\;\xi |_{\pa X}\;\;\text{is tangent to the fibers of }\;
\phi :\pa X\to Y \text{ and } \xi x\in x^2 C^\infty (X)\}
\end{equation}
where $ \mathcal{V}_b (X)$ is the Lie algebra of vector fields that are tangent to the boundary.
The Lie algebra \eqref{fibered-corners-vf-intro} is finitely generated and projective as a $C^\infty (X)$-module
and thus, according to Serre-Swan, there exists a smooth vector bundle ${}^{{\rm \Phi}}  TX\to X$,
the ${\rm \Phi}$-tangent bundle, whose sections are precisely the vector fields in $ \mathcal{V}_{{\rm \Phi}}  (X)$.
A fibered boundary metric  of product type as in \eqref{fb-metric} extends as a smooth metric to ${}^{{\rm \Phi}} TX\to X$.

Keeping with the {\it general principle} stated in the previous subsection, we shall consider related but more general
singular structures.  As a first generalization of manifolds with fibered boundary we shall consider  manifolds with a {\it foliated}
 boundary, already studied by Rochon \cite{rochon-foliated}; thus $X$ is a manifold
with boundary and there exists a foliation $\mathcal{F}$ on $\pa X$. If $T\mathcal F\subset T(\pa X)$ is the integrable
vector bundle that defines $(\pa X,\mathcal{F})$, then we shall consider
the Lie algebra of vector fields
\begin{equation*}
\mathcal{V}_{\mathcal{F}} (X)=\{\xi\in \mathcal{V}_b (X)\;\;\xi |_{\partial X} \in C^\infty (\partial X,T\mathcal{F})\;\; \text{ and } \xi x\in x^2 C^\infty (X)\}
\end{equation*}
This defines, as before, the $\mathcal{F}$-tangent bundle, ${}^\mathcal{F} TX\to X$; a $\mathcal F$-metric
$g_{\mathcal{F}}$ 
is a metric on $\mathring{X}$ that extends as a smooth metric on ${}^\mathcal{F} TX\to X$. 

There is a further generalization of a depth-1 Thom-Mather space, or, equivalently, of a manifold with fibered boundary;
we consider a foliated manifold $(X,\mathcal{H})$ with non-empty boundary and with
$\mathcal{H}$ transverse to the boundary $\pa X$. We then make the additional assumption that $\partial X$ is foliated by a second foliation
$\mathcal{F}$ such that $\mathcal{F}\subset \mathcal{H}_{| \pa X}$. We refer to this geometric situation as a {\it
foliation degenerating on the boundary} or, quite plainly, as a {\it singular foliation}.\\
The relevant Lie algebra of vector fields in now given by
\begin{equation*}
\mathcal{V}_{\mathcal{F}} (X,\mathcal{H})=\{\xi\in C^\infty(X,T\mathcal{H})\cap \mathcal{V}_b (X),
\;\;\xi |_{\partial X} \in C^\infty (\partial X,T\mathcal{F})\;\; \text{ and } \xi x\in x^2 C^\infty (X)\}
\end{equation*}
This is a finitely generated projective $C^\infty (X)$-module and thus, according to Serre-Swan, there
exists a vector bundle ${}^\mathcal{F} T\mathcal{H}$ whose sections are precisely given by 
$\mathcal{V}_{\mathcal{F}} (X,\mathcal{H})$. An admissible metric $g_{\mathcal{F}}^{\mathcal{H}}$ in this situation is a foliated metric on $(\mathring{X}, \mathcal{G} |_{\mathring{X}})$ that extends to a smooth metric on ${}^\mathcal{F} T\mathcal{H}$. 

Notice that the three Lie algebras of vector fields in the three singular structures
$$\mathcal{V}_{{\rm \Phi}} (X)\,,\quad \mathcal{V}_{\mathcal{F}} (X)\,,\quad \mathcal{V}_{\mathcal{F}} (X,\mathcal{H})$$
give rise to three algebras of differential operators, denoted respectively,
$$\Diff^*_{{\rm \Phi}} (X)\,,\quad \Diff^*_{\mathcal{F}} (X)\,,\quad \Diff^*_{\mathcal{F}} (X,\mathcal{H})\,.$$

\subsection{Main goal of this article}
Consider the three singular structures introduced above and consider in particular the
resulting tangent bundles endowed with the corresponding admissible metrics 
$$({}^{{\rm \Phi}} TX,g_{{\rm \Phi}}) \,,\quad ({}^\mathcal{F} TX, g_{\mathcal{F}})\,,\quad 
({}^\mathcal{F} T\mathcal{H},g_{\mathcal{F}}^{\mathcal{H}}).$$ 
Make a spin assumption on these bundle and fix in each case a spin structure. 
Then we have, correspondingly,
a spin-Dirac operator in each situation, denoted generically as $\slashed{D}$. We have, respectively : 
$$\slashed{D}\in \Diff^*_{{\rm \Phi}} (X)\,,\quad \slashed{D}\in \Diff^*_{\mathcal{F}} (X)\,,\quad \slashed{D}\in \Diff^*_{\mathcal{F}} (X,\mathcal{H})$$
where for simplicity we have omitted the spinor bundle form the notation.

\medskip
\noindent
We ask ourselves the following questions:

\medskip
\noindent
(i) can we define a fundamental class $[\slashed{D}]$ ?\\
(ii) can we define an index class ?\\
(iii) if the metric satisfies a positive scalar curvature assumption, is the index class equal to zero ?\\
(iv) in the latter case, can we define a rho class ?\\
(v) does the rho class descend to the set of path-connected 
components or to the set of concordance classes of the corresponding space of metrics of positive scalar curvature,
as in \eqref{rho-map} ?

\smallskip
\noindent
Needless to say, part of the problem is to understand in which K-theory groups these classes should live.

\medskip
\noindent
{\it The main goal of this article is to provide answers to the above questions.}

\subsection{Results and techniques.}
We first concentrate on stratified pseudomanifolds, for simplicity now  in the depth-1 case.
We let  ${}^{\textrm{S}} X$ be the singular pseudomanifold and $X$ its resolution; we denote by $\Gamma$ the fundamental group of ${}^{\textrm{S}} X$. We fix a 
${\rm \Phi}$-metric $g_{{\rm \Phi}}$ with local form near the boundary equal to
\begin{equation*}
\frac{dx^2}{x^4} + \frac{\phi^* g_Y}{x^2} + h_{\pa X/Y}\,.
\end{equation*}
Assume additionally that the vertical
tangent bundle $T(\pa X/Y)$ is spin and fix a spin structure for  $(T(\pa X/Y),h_{\pa X/Y})$.
Building on the ${\rm \Phi}$-calculus of Mazzeo and Melrose \cite{MMPHI} and on its generalization to pseudomanifolds 
of arbitrary depth by Debord, Lescure and Rochon \cite{DLR}, we show that if the metric $h_{\pa X/Y}$
has positive scalar curvature along the fibers, then $\slashed{D}$ and $\slashed{D}_\Gamma$ are {\it fully elliptic} and  there is 
therefore a well defined fundamental class
$[\slashed{D}]\in K_* ( {}^{\textrm{S}} X)$ and an index class $\Ind(\slashed{D}_\Gamma)\in K_* (C^*_r \Gamma)$.
Moreover, if $g_{{\rm \Phi}}$ is of positive scalar curvature, then $\Ind(\slashed{D}_\Gamma)=0$ in $ K_* (C^*_r \Gamma)$.
We then bring groupoids into the picture and show that these two classes can also be described through
a suitable  adiabatic groupoid (this groupoid already appears in \cite{DLR}). For the fundamental class this 
equivalent description of $[\slashed{D}]\in K_* ( {}^{\textrm{S}} X)$ was already known; see  \cite{DLR}.
For the index class
we give a very detailed proof of this compatibility, building on ideas  of Monthubert-Pierrot, Debord and Skandalis, see \cite{MP,DSk}.
Having given an adiabatic groupoid description of  $[\slashed{D}]\in K_* ( {}^{\textrm{S}} X)$ and  of $\Ind(\slashed{D}_\Gamma)\in K_* (C^*_r \Gamma)$ we now {\bf define} the rho class of a positive scalar curvature metric $g_{{\rm \Phi}}$
as the adiabatic rho class. Using the delocalized APS index theorem for general Lie groupoids, due to Zenobi,
we then prove that this rho class enjoys the usual stability properties on the space
of positive scalar curvature ${\rm \Phi}$-metrics; for example, we establish the analogue of  \eqref{rho-map-ad}.
The advantage of this approach through groupoids is that it can be extended with a relatively small effort to the other two more
singular situations; this is of course in accordance with the general principle put forward in this introduction
and it is in fact one of the main reasons to give a groupoid treatment of the three K-theoretic invariants in the stratified setting.
In generalizing to the foliated case we shall take advantage of the  results  contained in the recent paper of Debord and Skandalis \cite{DSk}; it is also because of their results that we are  able to give a unified treatment of the  three 
singular structures.
Notice that in the foliated  case all the three classes are defined {\it directly} in the adiabatic context. More information
about the content of this paper can be gathered from the description of the single sections given  below.

%

\bigskip\bigskip\bigskip
\noindent
{\bf The paper is organized as follows.} In Section \ref{sect:sss} we recall the basics about stratified pseudomanifolds
and their resolutions into manifolds with corners with an iterated fibration structure (briefly,  manifolds with fibered corners).
Section \ref{sect:vf&m} is devoted to ${\rm \Phi}$-geometry; thus we introduce the
${\rm \Phi}$-vector fields on a manifold with fibered corners, the associated  ${\rm \Phi}$-tangent bundle and  the metrics that we shall consider,
the ${\rm \Phi}$-metrics. Section \ref{sect:crash} is a brief introduction to groupoids and algebroids. In Section \ref{phi-grpd} we present the 
groupoid associated to a stratified pseudomanifold. In Section \ref{sect:pseudo} we introduce the ${\rm \Phi}$-pseudodifferential algebra,
both using  classical microlocal techniques  and  as the pseudodifferential algebra associated to the groupoid  defined by our stratified pseudomanifold; it is in this section that fully elliptic ${\rm \Phi}$-pseudodifferential operators
are introduced. In Section \ref{sect:k-classic} we employ ${\rm \Phi}$-pseudodifferential operators in order to define the fundamental class
and the index class of a $\Gamma$-equivariant fully elliptic ${\rm \Phi}$-pseudodifferential operator on a stratified pseudomanifold 
 ${}^{\textrm{S}} X_\Gamma$ endowed
with a free stratified cocompact $\Gamma$-action \footnote{for example, the universal cover of a stratified pseudomanifold}; these are elements
in $K^\Gamma_* ({}^{\textrm{S}} X_\Gamma)$ and $K_* (C^*_r \Gamma)$ respectively. In Section \ref{sect:k-adiabatic} we introduce the adiabatic
deformation of our groupoid and define the non-commutative symbol and the adiabatic index class associated to   a $\Gamma$-equivariant fully elliptic ${\rm \Phi}$-pseudodifferential operato. We also introduce the rho-class of an invertible operator and study its fundamental properties.
In Section \ref{sect:compatibility-facts}
we state the main theorems relating the fundamental class and the index class defined through classic microlocal methods with
the non-commutative symbol and the adiabatic index class defined through the adiabatic deformation of our groupoid.
Section \ref{sect:compatibility-proofs} is devoted to a detailed proof of the equality between the classical index class and the adiabatic index class
in our stratified context. In Section \ref{sect:dirac-fully} and Section \ref{sect:dirac-k} we treat carefully the case of spin Dirac operators
and its connections with the world of positive scalar curvature metrics. Finally, in Section \ref{sing-fol} we extend all of our results
to a class of singular foliations.

\bigskip
\noindent
{\bf Acknowledgements.} We thank Pierre Albin, Claire Debord, Jean-Marie Lescure, Fr\'ed\'eric Rochon and Georges Skandalis for many illuminating
and useful discussions. Part of this work was done while the first author was visiting {\it Universit\'e de Montpellier}, {\it Universit\'e Paris Diderot (Paris 7)}
and {\it G\"ottingen University} and while the second author was visiting {\it Sapienza Universit\`a di Roma}.
We thank these institutions and the ANR Project {\it SingStar} (Analysis on Singular and Noncompact Spaces: a $C^*$-algebra approach) for the financial support that made these visits possible. 
 \bigskip


\section{Smoothly stratified spaces and their resolutions}\label{sect:sss}

\subsection{Smoothly stratified spaces}
We start by recalling the definition of a
smoothly stratified pseudomanifold ${}^{\textrm{S}} X$  with Thom-Mather control data (briefly, a 
smoothly stratified space). The definition is given inductively. 

\begin{definition}
A smoothly stratified space of depth $0$ is a closed manifold. Let $k \in \N$, assume that 
the concept of smoothly stratified space of depth $\leq k$ has been defined.
A smoothly stratified space ${}^{\textsf{S}} X$ of depth $k+1$ is a  locally compact, second countable Hausdorff space which admits a locally finite 
decomposition into a union of locally closed {\rm strata} $\frakS = \{S^j\}$, 
where each $S^j$ is a smooth (usually open) manifold, with dimension depending on the index $j$. 
We assume the following:

\begin{itemize}
\item[i)] If $S^i, S^j \in \frakS$ and $S^i \cap \overline{S^j} \neq \emptyset$,
then $S^i \subset \overline{S^j}$. 
\item[ii)] Each stratum $S$ is endowed with a set of `control data' $T_S$, $\pi_S$ and $\rho_S$;
here $T_S$ is a neighbourhood of $S$ in ${}^{\textsf{S}} X$ which retracts onto $Y$, $\pi_S: T_S \longrightarrow S$ is 
a fixed continuous retraction and $\rho_S: T_S \to [0,2)$ is a  `radial function'  on the 
tubular neighbourhood such that $\rho_S^{-1}(0) = S$. Furthermore, we require that if $Z \in \frakS$ and 
$Z \cap T_S \neq \emptyset$, then 
\[
(\pi_S,\rho_S): T_S\cap Z \longrightarrow S \times [0,2)
\]
is a proper differentiable submersion. 
\item[iii)] If $W,Y,Z \in \frakS$, and if $p \in T_Y \cap T_Z \cap W$ and $\pi_Z(p) \in T_Y \cap Z$,
then $\pi_Y(\pi_Z(p)) = \pi_Y(p)$ and $\rho_Y(\pi_Z(p)) = \rho_Y(p)$.
\item[iv)] If $Y,Z \in \frakS$, then
\begin{eqnarray*}
Y \cap \ovl{Z} \neq \emptyset & \Leftrightarrow & T_Y \cap Z \neq \emptyset, \\
T_Y \cap T_Z \neq \emptyset & \Leftrightarrow & Y \subset \ovl{Z}, \ Y=Z\ \  \mbox{or}\ Z \subset \ovl{Y}.
\end{eqnarray*}
\item[v)] There exist a family of smoothly stratified spaces (with Thom-Mather control data) of depth less than or
equal to $k$, indexed by $\frakS$, $\{L_Y, Y\in \frakS\}$, with the property that the restriction $\pi_Y: T_Y \to Y$ is a 
locally trivial fibration with 
fibre the cone $C(L_Y)$ over  $L_Y$ (called the link over $Y$), with atlas $\calU_Y = 
\{(\phi,\calU)\}$ where each $\phi$ is a trivialization $\pi_Y^{-1}(\calU) \to \calU \times C(L_Y)$, and the 
transition functions are stratified isomorphisms  of $C(L_Y)$ which preserve the rays of 
each conic fibre as well as the radial variable $\rho_Y$ itself, hence are suspensions of isomorphisms of 
each link $L_Y$ which vary smoothly with the variable $y \in \calU$. 
\end{itemize}

If in addition we let ${}^{\textsf{S}}X_j$ be the union of all strata of dimensions less than or equal to $j$, and
require that 
\begin{itemize}
\item[vi)] ${}^{\textsf{S}}X =  {}^{\textsf{S}}X_n\supseteq {}^{\textsf{S}}X_{n-1} = {}^{\textsf{S}}X_{n-2} \supseteq {}^{\textsf{S}}X_{n-3} \supseteq \ldots \supseteq {}^{\textsf{S}}X_0$ and
${}^{\textsf{S}}X \setminus {}^{\textsf{S}}X_{n-2}$ is dense in ${}^{\textsf{S}}X$
\end{itemize}
then we say that $\hat X$ is a stratified pseudomanifold. 
\end{definition}
The depth of a stratum $S$ is the largest integer $k$ such that there is a chain of strata
$S= S^k, \ldots, S^0$ with $S^{j-1} \subset \overline{S^{j}}$ for $1 \leq j \leq k$. A stratum of minimal
depth is always a closed manifold. The maximal depth of any stratum in ${}^{\textsf{S}}X$ is called the depth of ${}^{\textsf{S}}X$ as 
a stratified space. (We follow here the convention
of \cite{DLR}, given that we shall use heavily this paper.)

We refer to the dense open stratum of a stratified pseudomanifold ${}^{\textsf{S}}X$ as its regular set,
and the union of all other strata as the singular set,
\[
\mathrm{reg}({}^{\textsf{S}}X) := {}^{\textsf{S}}X\setminus \mathrm{sing}({}^{\textsf{S}}X), \qquad \mathrm{where}\qquad
\mathrm{sing}({}^{\textsf{S}}X) = \bigcup_{{S\in \mathfrak S}\atop{\mathrm{depth}\,S > 0}} S.
\]
In this paper, we shall often for brevity refer to a smoothly stratified pseudomanifold with Thom-Mather
control data as a {\em smoothly stratified space}.

\subsection{Iterated fibration structures}
Let $X$ be a manifold with corners. We assume that each boundary hypersurface $H\subset X$ is
embedded in $X$. This means that there exists a boundary defining function
$x_H\in C^\infty(X)$ such that $x^{-1}_H (0) = H$, $x_H$ is positive on $X \setminus H$ and the
differential $dx_H$ is nowhere zero on $H$. In such a situation, one can choose
a smooth retraction $r_H \colon \mathcal{N}_H \to H$ , where $\mathcal{N}_H$ is a (tubular) neighborhood
of $H$ in $X$ such that $(r_H, x_H) \colon \mathcal{N}_H \to H \times [0,1)$ is a diffeomorphism
on its image. We call $(\mathcal{N}_H, r_H, x_H)$ a tube system for $H$. A smooth map
$f\colon X\to Y$
between manifolds with corners is said to be a fibration if it
is a locally trivial surjective submersion.

\begin{definition} Let $X$ be a compact manifold with corners and
$H_1, . . . ,H_k$ an exhaustive list of its set of boundary hypersurfaces $M_1X$.
Suppose that each boundary hypersurface $H_i$ is the total space of a smooth
fibration $\phi_i \colon H_i\to S_i$ where the base $S_i$ is also a compact manifold with corners. The collection of fibrations
 $\phi=(\phi_1,\dots,\phi_k)$ 
 is said to be an iterated
fibration structure if there is a partial order on the set of hypersurfaces such
that
\begin{enumerate}
\item for any subset $I\subset \{1, \dots, k\}$ with 
$\bigcap_{i\in I}
H_i \neq\emptyset$, the set $\{H_i | i \in I\}$
is totally ordered.
\item If $H_i < H_j$ , then $H_i \cap H_j  \neq\emptyset$, $\phi_i \colon H_i \cap H_j \to S_i$ is a surjective
submersion and $S_{ji}:=\phi_j (H_i\cap H_j) \subset S_j$ is a boundary hypersurface
of the manifold with corners $S_j$ . Moreover, there is a surjective
submersion $\phi_{ji}\colon S_{ji} \to S_i$ such that on $H_i \cap H_j$ we have
$\phi_{ji}\circ \phi_j=\phi_i$.
\item The boundary hypersurfaces of $S_j$ are exactly the $S_ji$ with $H_i < H_j$ .
In particular if $H_i$ is minimal, then $S_i$ is a closed manifold.
\end{enumerate}
\end{definition}

We shall refer to a manifold with corners endowed with an iterated fibration structure as
a {\em manifold with fibered corners}. For more on the notion of manifold with fibered corners, a notion due
to Richard Melrose,  the reader is referred to 
\cite{ALMP:Witt}, \cite{DLR}.

\subsection{The resolution of a stratified space}
If ${}^{\textrm{S}} X$  is a smoothly stratified pseudomanifold then, as explained in detail in \cite{ALMP:Witt},
it is possible to resolve ${}^{\textrm{S}} X$ into a manifold with fibered corners $X$. More precisely, there exists
a manifold with fibered corners $X$ and a continuous surjective map $\beta: X\to {}^{\textrm{S}} X$ which restrict
to a diffeomorphism from the interior of $X$ onto the regular part of ${}^{\textrm{S}} X$. The construction of $X$ is iterative,
by means of radial blow-ups.
If ${}^{\textrm{S}} X$ is normal (i.e. the links are connected), then the boundary hypersurfaces
of $X$ correspond bijectively to the strata of ${}^{\textrm{S}} X$: to each stratum $Y$  there is an associated
boundary hypersurface $H_Y$ which is a fibration with base the resolution of the closure of $Y$, a stratified
space itself, and fibers the resolution of the links of $Y$. If the links are not connected then each stratum
contributes to a {\em collective boundary hypersurface}, see \cite{ALMP:Hodge} and references therein. 
For simplicity we bound ourselves to the normal case; we thus have a bijection between the strata of
${}^{\textrm{S}} X$ and the boundary hypersurfaces of $X$.\\ Notice that the process  of replacing ${}^{\textrm{S}} X$ by $X$ is already 
present in Thom's seminal work \cite{Thom:Ensembles}, and versions of it have 
appeared in Verona's `resolution of singularities' \cite{Verona} and the `d\'eplissage' of Brasselet-Hector-Saralegi \cite{BHS}. These 
constructions show that any smoothly stratified space can be resolved to a smooth manifold, 
possibly with corners; the additional information given in \cite{ALMP:Witt} is  the existence of an 
iterated boundary
fibration structure on the resolved manifold with corners $X$.

\subsection{Galois coverings.} Let $\Gamma$ be a finitely generated discrete group.
Assume now that $\Gamma - {}^{\textrm{S}} X_\Gamma \to {}^{\textrm{S}} X$ is a Galois $\Gamma$-covering of a smoothly stratified space
${}^{\textrm{S}} X$. Since  ${}^{\textrm{S}} X_\Gamma$ and  ${}^{\textrm{S}} X$ are locally homeomorphic we can endow ${}^{\textrm{S}} X_\Gamma$ with  the  structure of a smoothly stratified space in such a way that the projection map  ${}^{\textrm{S}} X_\Gamma\to  {}^{\textrm{S}} X $ be  a stratified map. Notice that the strata of  ${}^{\textrm{S}} X_\Gamma$ are the lifts of the strata of ${}^{\textrm{S}} X$ and that the link of a stratum
$Y\subset {}^{\textrm{S}} X$ is the same as the link of the lifted stratum, $Y_\Gamma$, in ${}^{\textrm{S}} X_\Gamma$.
Let $X_\Gamma$ be the resolution of ${}^{\textrm{S}} X_\Gamma$; following the inductive procedure that defines the resolution
of a stratified space it is not difficult to see that it is possible to lift the action of $\Gamma$ from 
${}^{\textrm{S}} X_\Gamma$ to $X_\Gamma$. Moreover, with this induced action, $X_\Gamma$ is a Galois covering
of $X$ and the interior of $X_\Gamma$, which is the regular part of   ${}^{\textrm{S}} X_\Gamma$, is a Galois $\Gamma$-cover of
the interior of $X$, which is in turn the regular part of ${}^{\textrm{S}} X$. In fact, the action of $\Gamma$ respects the boundary fibrations structures of $X_\Gamma$
and $X$: if $H\xrightarrow{\phi} S$ is a boundary hypersurface of $X$ corresponding to a  stratum $Y$
of  ${}^{\textrm{S}} X$ and if $p$ denotes the quotient map induced by the action, then there is a commutative diagram 

\begin{equation}\label{compatible-bfs}
\xymatrix{H_{\Gamma}\ar[r]^(.5){p}\ar[d]^{\phi_\Gamma}&H\ar[d]^{\phi}\\
	S_\Gamma\ar[r]^(.5){p}& S }
\end{equation}
where $H_\Gamma\xrightarrow{p} H$ induces a diffeomorphism on the fibers of the two fibrations,
$H_{\Gamma}\xrightarrow{\phi_\Gamma} S_\Gamma$ and $H\xrightarrow{\phi} S$, 
and where $S_\Gamma$ is the resolution of the closure of $Y_\Gamma$.

\section{Vector fields and metrics}\label{sect:vf&m} 

\subsection{Fibered-corners vector fields}

Let $X$ be a manifold with fibered corners as above, 
with fibered hypersurfaces $\phi_1\colon H_1\to S_1, \dots, \phi_k\colon H_k\to S_k$. 
We assume that $H_1,\dots,H_k$ is an exhaustive list of the boundary hypersurfaces of $X$.
For each $i$, let $x_i$ be a boundary defining function of the hypersurface $H_i$. 
The $b$-vector fields are the vector fields on $X$ that are tangent to the boundary:
\begin{equation}\label{b-vf}
\mathcal{V}_b (X) =\{\xi \in C^\infty(X,TX)\,\;\; ; \;\;\xi x_i\in x_i C^\infty (X) \,\forall i\}
\end{equation}
They form a Lie subalgebra of the Lie algebra of all vector fields on $X$.\\
We introduce the
space  of fibered-corners vector  fields, or ${\rm \Phi}$-vector fields
as
\begin{equation}\label{fibered-corners-vf}
\mathcal{V}_{{\rm \Phi}} (X)=\{\xi\in \mathcal{V}_b (X)\;\;\xi |_{H_i}\,,\;\;\text{is tangent to the fibers of }\;
\phi_i:H_i\to S_i \text{ and } \xi x_i\in x_i^2 C^\infty (X) \,\;\forall i\}
\end{equation}
and we point out that they also form a Lie subalgebra.

\subsection{The vector bundle  ${}^{{\rm \Phi}} TX$ associated to $\mathcal{V}_{{\rm \Phi}} (X)$}
The algebra $\mathcal{V}_{{\rm \Phi}} (X)$ is a finitely generated projective $C^\infty (X)$-module
and thus, according to the  Serre-Swan theorem,  there exists a smooth vector bundle   ${}^{{\rm \Phi}} TX$
on $X$ and a natural map $i_{{\rm \Phi}}: {}^{{\rm \Phi}} TX\to TX$ with the property that
$$i_{{\rm \Phi}} (C^\infty (X;{}^{{\rm \Phi}} TX))=\mathcal{V}_{{\rm \Phi}}(X)$$
where, with a small abuse of notation, we also use $i_{{\rm \Phi}}$ for the resulting map  on sections.
The following properties are crucial:

\begin{itemize}
\item $C^\infty (X;{}^{{\rm \Phi}} TX)$ has a Lie algebra structure;
\item the map $i_{{\rm \Phi}}$ on sections satisfies $i_{{\rm \Phi}} [X,Y]=[i_{{\rm \Phi}} X, i_{{\rm \Phi}} Y]$ for all $X,Y\in C^\infty (X;{}^{{\rm \Phi}} TX)$;
\item if $f\in C^\infty (X)$ then $[X,fY]=f[X,Y]+ (i_{{\rm \Phi}} (X)f) Y$ for all $X,Y\in C^\infty (X;{}^{{\rm \Phi}} TX)$\,.
\end{itemize}
As we shall see in a moment, it is possible to encode all of the above by saying that  ${}^{{\rm \Phi}} TX$ is a Lie algebroid over $X$
with anchor map $i_{{\rm \Phi}}: {}^{{\rm \Phi}} TX\to TX$.

\subsection{Metrics.}
Consider a smoothly stratified space ${}^{\textrm{S}} X$ and its regular part $\mathrm{reg}({}^{\textsf{S}}X) $.
Since $\mathrm{reg}({}^{\textsf{S}}X) $ is diffeomorphic to the interior of $X$, $\mathring{X}$, we directly
use the symbol $\mathring{X}$ for this smooth non-compact manifold.
We can endow $\mathring{X}$ with a variety of riemannian metrics.

 In \cite{ALMP:Witt} \cite{ALMP:Hodge}
the main focus was on iterated {\em incomplete edge} metrics, which are now called
 {\em iterated wedge metrics}. Iterated  {\em complete} edge metrics, simply called now {\it iterated edge metrics} can also
be considered: given an iterated wedge metric $g_{{\rm w}}$ we can define an iterated edge metric by setting
$g_{{\rm e}}:= \rho^{-2} g_{{\rm w}}$
where $\rho$ is the product of all the boundary defining functions $x_i$, $i\in\{1,\dots,k\}$.

As explained for example in \cite{ALMP:Witt},
an iterated edge metric $g_{{\rm e}}$ extends as a smooth metric on the edge tangent bundle ${}^{\rm e} TX\to X$ which is, by definition, the vector bundle obtained by applying the Serre-Swan theorem 
to the Lie algebra of edge vector fields
\begin{equation}\label{edge-vf}
\mathcal{V}_{{\rm e}} (X)=\{\xi\in \mathcal{V}_b (X)\;\;\xi |_{H_i}\,,\;\;\text{is tangent to the fibers of }\;
\phi_i:H_i\to S_i  \,\;\forall i\}
\end{equation}
(this is indeed a finitely generated projective $C^\infty (X)$-module).

\medskip
In this article we shall be interested in  {\em fibered corner metrics} \footnote{also called ${{\rm \Phi}}$-metrics in the literature}
and these are defined analogously but  in terms
of $\mathcal{V}_{{\rm \Phi}} (X)$ and the associated vector bundle ${}^{{\rm \Phi}}TX$:

\begin{definition}\label{def:metric}
A riemannian  metric on $\mathring{X}$ is
a fibered corner  metric
if it  extends as a 
 smooth bundle metric to ${}^{{\rm \Phi}} TX\to X$.\\
 We shall not distinguish between the metric on $\mathring{X}$ and its extension to ${}^{{\rm \Phi}} TX$.
 \end{definition}
 
 Notice that a  fibered corner metric $g_{{\rm \Phi}}$ is complete. $\mathring{X}$, endowed with a fibered corner metric 
 $g_{{\rm \Phi}}$, is an example of a manifold with a Lie structure at infinity \cite{ALN}; $\mathring{X}$ endowed with  an  iterated edge
 metric  $g_{{\rm e}}$  is another example of  a manifold with a Lie structure at infinity.
 
Given a fibered corner metric $g_{{\rm \Phi}}$ we can also consider the metric $g_{{\rm fcusp}}:= \rho^{-2} g_{{\rm \Phi}}$,
which is, by definition, an {\rm iterated fibered cusp metric}. This is also a complete metric.

Notice that iterated fibered cusp metrics and iterated wedge metrics are also smooth metrics on suitable 
vector bundles over $X$; however, the corresponding vector fields, i.e., the sections of these vector bundles,
are not closed under Lie bracket.
\bigskip

 \noindent
 {\bf Example. (Rigid metrics on depth-1 spaces.)} \\
{\it  If, for example, ${}^{\textrm{S}} X$ is a depth-1 stratified space, so that $X$ is a manifold with fibered boundary
 $\partial X=:H$, $Z - H\xrightarrow{\phi} S$,
 then an example of  fibered corner metric (in this case,  a fibered {\em boundary} metric) is a metric that on $\mathring{X}$,
  that near $\partial X\equiv H\equiv\{x=0\}$, can be written as
\begin{equation}\label{fibered-bdry-metric}
g_{{\rm \Phi}}=\frac{dx^2}{x^4} + \frac{\phi^* g_S}{x^2} + g_{H/S}
\end{equation}
with $g_S$ a metric on $S$, the singular locus of ${}^{\textrm{S}} X$, and $g_{H/S}$ a metric on the vertical tangent bundle of $H$.
A fibered boundary metric with this product structure near the boundary is called {\bf rigid}. Wedge metrics, edge metrics and fibered cusp metrics with the additional property of being rigid
would be respectively defined by the following forms near $\pa X$:
\begin{equation*}
g_{\rm w}=dx^2 +\phi^* g_S+ x^2 g_{H/S}\,,\;\;\;\; g_{\rm e}=\frac{dx^2}{x^2} + \frac{\phi^* g_S}{x^2} + g_{H/S}
\,,\;\;\; \;g_{\rm fcusp}=\frac{dx^2}{x^2} + \phi^* g_S + x^2 g_{H/S}.
\end{equation*}
}
 
 \medskip
 A general rigid fibered corner metric has an iterative description similar to the one in \eqref{fibered-bdry-metric}:
 in a collar neighbourhood of $H_i$ it can be written as
 \begin{equation}\label{fibered-corner-metric}
g_{{\rm \Phi}}=\frac{dx^2_i}{x_i^4} + \frac{\phi_i^* g_{S_i}}{x_i^2} + g_{H_i /S_i}
\end{equation}
with $g_{S_i}$ a fibered corner metric on $S_i$ and $g_{H_i /S_i}$  a family of fibered corners metrics on the fibers of 
$H_i\xrightarrow{\phi_i} S_i$. 

\bigskip
\noindent
{\bf In this article we shall work exclusively with rigid fibered corner metrics.}

 \subsection{Densities}
The ${{\rm \Phi}}$-density bundle ${}^{{\rm \Phi}}\Omega$ is the bundle on $X$ with fiber at $p$ equal to
$\{u:\Lambda^{\dim X} ({}^{{\rm \Phi}} T_pX)\to \bbR\;;\; u(t\omega)=|t|u(\omega) \; \forall \omega \in \Lambda^{\dim X} ({}^{{\rm \Phi}} T_pX)\,,\;
\forall t\not=0\}$. The volume form associated to a fibered corner metric is a section of the $\Phi$-density bundle.

\section{A crash course on  Lie groupoids and  Lie algebroids}\label{sect:crash}
We refer the reader to the survey  \cite{DL} and the bibliography therein  for a detailed overview about groupoids and 
their role in index theory. Still, we shall now give the fundamental notions that are necessary in order to understand the content of this article.
 
\subsection{Basics}
\begin{definition} Let $G$ and $G^{(0)}$ be two sets.  A groupoid structure on $G$ over $G^{(0)}$ is given by the following morphisms:
	\begin{itemize}
	
		\item Two maps: $r,s: G\rightarrow G^{(0)}$,
		which are respectively the range and  source map.
			\item A map $u:G^{(0)}\rightarrow G$ called the unit map that is a section for both $s$ and $r$. We can identify $G^{(0)}$ with its image
			in $G$. 
		\item An involution: $ i: G\rightarrow G
		$, $  \gamma  \mapsto \gamma^{-1} $ called the inverse
		map. It satisfies: $s\circ i=r$.
		\item A map $ m: G^{(2)}  \rightarrow  G
		$, $ (\gamma_1,\gamma_2)  \mapsto  \gamma_1\cdot \gamma_2 $
		called the product, where the set 
		$$G^{(2)}:=\{(\gamma_1,\gamma_2)\in G\times G \ \vert \
		s(\gamma_1)=r(\gamma_2)\}$$ is the set of composable pair. Moreover for $(\gamma_1,\gamma_2)\in
		G^{(2)}$ we have $r(\gamma_1\cdot \gamma_2)=r(\gamma_1)$ and $s(\gamma_1\cdot \gamma_2)=s(\gamma_2)$.
	\end{itemize}
	
	The following properties must be fulfilled:
	\begin{itemize}
		\item The product is associative: for any $\gamma_1,\
		\gamma_2,\ \gamma_3$ in $G$ such that $s(\gamma_1)=r(\gamma_2)$ and
		$s(\gamma_2)=r(\gamma_3)$ the following equality
		holds $$(\gamma_1\cdot \gamma_2)\cdot \gamma_3= \gamma_1\cdot
		(\gamma_2\cdot \gamma_3)\ .$$
		\item For any $\gamma$ in $G$: $r(\gamma)\cdot
		\gamma=\gamma\cdot s(\gamma)=\gamma$ and $\gamma\cdot
		\gamma^{-1}=r(\gamma)$.
	\end{itemize}
	
	We denote a groupoid structure on $G$ over $G^{(0)}$ by
	$G\rightrightarrows G^{(0)}$,  where the arrows stand for the source
	and target maps. 
\end{definition}

We will adopt the following notations: $$G_A:=
s^{-1}(A)\ ,\ G^B:=r^{-1}(B)\, \ G_A^B:=G_A\cap G^B \,\ \mbox{ and } G_{| A}:=G^A_A$$
in particular if $x\in G^{(0)}$, the  $s$-fiber (resp. 
$r$-fiber) of $G$ over $x$ is $G_x=s^{-1}(x)$ (resp. $G^x=r^{-1}(x)$).

\begin{definition}
	
	We call $G$ a Lie groupoid when $G$ and $G^{(0)}$ are second-countable smooth manifolds
	with $G^{(0)}$ Hausdorff, and the structural homomorphisms are smooth.
\end{definition}

Let us see some examples. In the following ones $X$ is a smooth manifold, $\phi\colon X\to B$ is a smooth fibration, $p\colon X_\Gamma\to X$ is a Galois $\Gamma$-covering, $E\to X$ is a vector bundle:\\

\begin{tabular}{|l|l|l|l|l|}
 	\hline
 	$G\rightrightarrows G^{(0)}$ & $r$ & $s$ & $i$ & $m$\\
 	\hline
 $X\times X\rightrightarrows X$ & $(x,y)\mapsto x$ & $(x,y)\mapsto y$& $(x,y)\mapsto (y,x)$ & $(x,y)\cdot(y,z)= (x,z)$\\
 	\hline
  $X\times_B X\rightrightarrows X$ & $(x,y)\mapsto x$ & $(x,y)\mapsto y$& $(x,y)\mapsto (y,x)$ & $(x,y)\cdot(y,z)= (x,z)$ \\
 	\hline
  $X_\Gamma\times_\Gamma X_\Gamma\rightrightarrows X$ & $(\tilde{x},\tilde{y})\mapsto p(\tilde{x})$ & $(\tilde{x},\tilde{y})\mapsto p(\tilde{y})$& $(\tilde{x},\tilde{y})\mapsto (\tilde{y},\tilde{x})$ & $(\tilde{x},\tilde{y})\cdot(g\tilde{y},\tilde{z})=(\tilde{x},g^{-1}\tilde{z})$ \\
 	\hline
 	$E\rightrightarrows X$ & $(x,\xi)\mapsto x$ & $(x,\xi)\mapsto x$ & $(x,\xi)\mapsto (x,-\xi)$ & $(x,\xi)\cdot(x,\eta)=(x,\xi+\eta)$\\
 	\hline
 \end{tabular}

\subsection{Groupoid C*-algebras}
We can associate to a Lie groupoid $G$ the *-algebra $$C^\infty_c(G,\Omega^{\frac{1}{2}}(\ker ds\oplus\ker dr))$$ of the compactly supported sections of the half densities bundle associated to $\ker ds\oplus\ker dr$, with:

\begin{itemize}
	\item the involution given by $f^*(\gamma)=\overline{f(\gamma^{-1})}$;
	\item and the convolution product given by $f*g(\gamma)=\int_{G_{s(\gamma)}} f(\gamma\eta^{-1})g(\eta)$.
\end{itemize}  

For all $x\in G^{(0)}$ the algebra $C^\infty_c(G,\Omega^{\frac{1}{2}}(\ker ds\oplus\ker dr))$ can be represented on 
$L^2(G_x,\Omega^{\frac{1}{2}}(G_x))$ by 
\[\lambda_x(f)\xi(\gamma)=\int_{G_{x}} f(\gamma\eta^{-1})\xi(\eta), \]
where $f\in C^\infty_c(G,\Omega^{\frac{1}{2}}(\ker ds\oplus\ker dr))$ and $\xi\in L^2(G_x,\Omega^{\frac{1}{2}}(G_x))$.

\begin{definition}
	The reduced C*-algebra of a Lie groupoid G, denoted by $C^*_r(G)$, is the completion of $C^\infty_c(G,\Omega^{\frac{1}{2}}(\ker ds\oplus\ker dr))$ with respect to the norm
	\[
	||f||_r=\sup_{x\in G^{(0)}}||\lambda_x(f)||_{x}
	\]
	where $||\cdot||_{x}$is the operator norm on $L^2(G_x,\Omega^{\frac{1}{2}}(G_x))$.
	
	The full C*-algebra of $G$ is the completion of 
	$C^\infty_c(G,\Omega^{\frac{1}{2}}(\ker ds\oplus\ker dr))$ with respect to all continuous representations.
\end{definition}

\smallskip
\noindent
{\bf Notation}; We shall often omit from the notation the half-densities bundle and consider it as understood.

\begin{example}\label{examples-c*}
We give a few examples:\begin{itemize} \item$C^*_r(X\times X)\cong C^*(X\times X)\cong \mathbb{K}(L^2(X))$; 
	\item $C^*_r(X\times_B X)\cong C^*(X\times_B X)$ is a field over $B$ of compact operators C*-algebras;
	\item $C^*_r(X_\Gamma\times_\Gamma X_\Gamma)$ is Morita equivalent 
\footnote{in particular they have the same K-theory groups} to $C^*_r(\Gamma)$;
\item $C^*_r(E)\cong C^*(E)\cong C_0(E^*)$. 
\end{itemize}
\end{example}

\begin{remark}\label{fullvsred}
From now on, if $X$ is a $G$-invariant closed subset of $G^{(0)}$ (this is also called a {\it saturated} closed subset of 
 $G^{(0)}$)
we will call
$e_X\colon C^\infty_c(G)\to C^\infty_c(G_{|X})$ the restriction map to $X$.
This gives an exact sequence of full groupoid C*-algebras
\begin{equation}\label{short-saturated}
\xymatrix{0\ar[r]& C^*(G_{|G^{(0)}\setminus X})\ar[r]&C^*(G)\ar[r]& C^*(G_{|X})\ar[r]&0},
\end{equation}
see \cite{renault}.
Notice that in general this is not true for the reduced $C^*$-algebras:  
the reader can find examples of this phenomenon in \cite{HLS}. However, in some examples that we shall consider (not all) the
short exact sequence \eqref{short-saturated} will also hold for the reduced $C^*$-algebras. For instance this is the case when the groupoid $G_{|X}$ is amenable.\\

\end{remark}

\subsection{Lie algebroids}

\begin{definition} A  Lie
	algebroid $\mathfrak{A} =(p:\mathfrak{A}\rightarrow TM,[\ ,\ ]_{\mathfrak{A}})$ on a smooth
	manifold $M$ is a vector bundle $\mathfrak{A} \rightarrow M$
	equipped with a bracket $[\ ,\ ]_{\mathfrak{A}}:C^\infty (M;\mathfrak{A})\times C^\infty (M;\mathfrak{A})
	\rightarrow C^\infty(M,\mathfrak{A})$ on the module of sections of $\mathfrak{A}$, together
	with a homomorphism of fiber bundle $p:\mathfrak{A} \rightarrow TM$ from $\mathfrak{A}$ to the
	tangent bundle $TM$ of $M$, called the  anchor map, fulfilling the following conditions:
	\begin{itemize}
		\item the bracket $[\ ,\ ]_{\mathfrak{A}}$ is $\RR$-bilinear, antisymmetric
		and satisfies the Jacobi identity,
		\item $[X,fY]_{\mathfrak{A}}=f[X,Y]_{\mathfrak{A}}+p(X)(f)Y$ for all $X,\ Y \in
	C^\infty (M;\mathfrak{A})$ and $f$ a smooth function of $M$, 
		\item $p([X,Y]_{\mathfrak{A}})=[p(X),p(Y)]$ for all
		$X,\ Y \in C^\infty (M;\mathfrak{A})$.
	\end{itemize}
	Here, with a small abuse, we are using the same notation for the bundle map $p$ 
	and for the map induced by $p$ on the sections of the two bundles.
	\end{definition}

Let $G$ be a Lie groupoid.
The  tangent space to $s$-fibers, that is $T_sG := \ker ds$ restricted to the objects of $G$ is
$\bigcup_{x\in G^{(0)}} TG_x$  and it has the structure of  Lie algebroid
on $G^{(0)}$, with the anchor map given by $dr$. 
It is denoted by 
$\mathfrak{A}G$ and we call it the Lie algebroid of $G$.
It is easy to prove that $\mathfrak{A}G$ is isomorphic to the normal bundle of the inclusion $G^{(0)}\hookrightarrow G$. 

Given a Lie algebroid  $\mathfrak{A} =(p:\mathfrak{A}\rightarrow TM,[\ ,\ ]_{\mathfrak{A}})$
on manifold $M$ we can ask whether it can be integrated, i.e. whether it is the Lie algebroid 
of a Lie groupoid. As clarified in \cite{almeida-molino} \cite{Cr-Fe} this is not always possible; however, as we shall see in 
moment, there are sufficient conditions ensuring the existence of a Lie groupoid integrating a given Lie algebroid.

Notice that if a Lie algebroid is integrable then it can be the Lie algebroid associated to different Lie groupoids;
for example if $M$ is a smooth compact manifold with universal cover $M_\Gamma$ and $\mathfrak{A}=(\Id : TM\to TM)$ is the  Lie algebroid over $M$
given by the identity map, then $\mathfrak{A}$ is the Lie algebroid associated to the pair groupoid 
$M\times M\rightrightarrows M$ but also to the groupoid $M_\Gamma \times_\Gamma M_\Gamma \rightrightarrows M$.

\section{The groupoid associated to stratified spaces} 
\label{phi-grpd}


We now go back to the  smoothly stratified space ${}^{\textrm{S}} X$. Let $X$ be its  resolution, a manifold with fibered corners. Consider the Lie algebra of vector fields $\mathcal{V}_{{\rm \Phi}} (X)$ and the associated
Lie algebroid $({}^{{\rm \Phi}} TX, i_{{\rm \Phi}})$. We wish to integrate this  algebroid to a groupoid. We follow \cite{DLR}.

 Let $H_i \xrightarrow{\phi_i} S_i$, $i\in\{1,\dots,k\}$ be an exhaustive list
of the boundary hypersurfaces of $X$ and let $x_i$ be a boundary defining function for $H_i$. We assume that
$$i<j\,\Rightarrow H_i < H_j\;\;\text{or}\;\;H_i\cap H_j=\emptyset\,.$$  By means of the Lie algebra of ${{\rm \Phi}}$ vector fields $\mathcal{V}_{{\rm \Phi}} (X)$ we have defined the  Lie algebroid ${}^{{\rm \Phi}} TX$
with anchor map $$i_{{\rm \Phi}}: {}^{{\rm \Phi}} TX\rightarrow TX\,.$$ Notice that the anchor map is injective (in fact, an isomorphism)
when restricted to $\mathring{X}$, a dense open subset of $X$. According to a theorem of Debord, see
\cite[Theorem 2]{debord-jdg} we thus obtain that the Lie algebroid ${}^{{\rm \Phi}} TX$ is integrable. Following
 \cite{DLR} we can explicitly exhibit a groupoid $G_{{\rm \Phi}}\rightrightarrows X$ integrating the Lie algebroid ${}^{{\rm \Phi}} TX$.
 This is described as follows: over $\mathring{X}$ the groupoid $G_{{\rm \Phi}}$ is simply the pair groupoid
 $\mathring{X}\times \mathring{X}$; over the boundary of $X$ the groupoid $G_{{\rm \Phi}}$ is given by
 $$\bigsqcup_{i=1}^k (H_i\underset{\phi_i}{\times} {}^{{\rm \Phi}} T S_i \underset{\phi_i}{\times}  H_i)_{|G_i} \times\RR$$
 with $G_i= H_i\setminus \cup_{j>i} H_j$. The range an source map are given as follows:
 if $h$ and $h^\prime$ are points in $G_i$, with $\phi_i (h)=s=\phi_i (h^\prime)$, and if $v\in T_s S_i$ then 
 $$s(h,v,h^\prime,\lambda)= h\,,\qquad r(h,v,h^\prime,\lambda)= h^\prime\,.$$
 The composition of two composable elements $(h,v,h^\prime,\lambda)$ and $(h^\prime,w,h^{\prime\prime},\mu)$ is given, by definition, by
 $(h,v+w,h^{\prime\prime},\lambda+\mu)$.\\
 Although only a set for the time being, $G_{{\rm \Phi}}$ can be given a smooth structure by observing,
 as in
 \cite{DLR}, that there is a natural bijection  \begin{equation}\label{inclusion}
 G_{{\rm \Phi}}\longleftrightarrow\mathring{X^2_{{\rm \Phi}}}\cup \mathring{{\rm ff}}_{{\rm \Phi}}\end{equation}
 with $X^2_{{\rm \Phi}}$ the ${{\rm \Phi}}$-double space (introduced in the next section) and 
 ${\rm ff}_{{\rm \Phi}}$ its front face. Using the properties of the ${{\rm \Phi}}$-double space it is proved
 in \cite{DLR} , see page 26 there, and \cite[Section 4.3, Proposition 6]{lescure-hdr}, that $G_{{\rm \Phi}}$ integrates the Lie algebroid $({}^{{\rm \Phi}} TX, i_{{\rm \Phi}})$.
 For a slightly different approach to $G_{{\rm \Phi}}\rightrightarrows X$ see also the last section of this paper.

We shall be also concerned with another groupoid integrating the Lie algebroid ${}^{{\rm \Phi}} TX$. This is the Lie groupoid
$G^\Gamma_{{\rm \Phi}}\rightrightarrows X$ which is described as follows.
Recall the Galois covering 
$X_\Gamma \xrightarrow{p} X$ and the structure of the total space $X_\Gamma$
as a manifolds with fibered corners $H_{i, \Gamma}$ where, as
discussed in \eqref{compatible-bfs}, we have 
\begin{equation}\label{compatible-bfs-2}
\xymatrix{H_{i,\Gamma}\ar[r]^(.5){p}\ar[d]^{\phi_{i,\Gamma}}&H\ar[d]^{\phi_i}\\
	S_{i,\Gamma}\ar[r]^(.5){p}& S }
\end{equation} 
$p$ denoting the quotient map with respect to the induced $\Gamma$-action.
Consider now $G_{{{\rm \Phi}},\Gamma}\rightrightarrows X_\Gamma$ defined precisely as before. This groupoid is freely acted upon by
$\Gamma$: if $g$ is an element of $\Gamma$ and $(x,y)$ is in  $\mathring{X}_\Gamma \times\mathring{X}_\Gamma$ , then $g\cdot(x,y)=(g\cdot x,g\cdot y)$; if instead $(x,\xi,y,t)$ is an element over the boundary, then $g\cdot(x,\xi,y,t)= (g\cdot x,dg\cdot\xi,g\cdot y,t)$. The quotient is our groupoid $G^\Gamma_{{\rm \Phi}} \rightrightarrows X$.
Then, by construction and by  \eqref{compatible-bfs-2}, $G^\Gamma_{{\rm \Phi}}\rightrightarrows X$ is equal  to $\mathring{X}_\Gamma \times_\Gamma \mathring{X}_\Gamma$
over $\mathring{X}$ and is equal to  $$\bigsqcup_{i=1}^k (H_i\underset{\phi_i}{\times} {}^{{\rm \Phi}} T S_i \underset{\phi_i}{\times}  H_i)_{|G_i} \times\RR$$
on $\pa X$.

\subsection{Notation.}
In the sequel we shall use exclusively the groupoid  $G^\Gamma_{{\rm \Phi}}\rightrightarrows X$. We shall often denote denote the groupoid $G^\Gamma_{{\rm \Phi}}\rightrightarrows X$  simply by  $G\rightrightarrows X$.


\section{Pseudodifferential operators}\label{sect:pseudo}
Let $\Diff^*_{{\rm \Phi}} (X)$ be the  algebra generated by products of elements of  the Lie algebra $\mathcal{V}_{{\rm \Phi}} (X)$
and of functions in $C^\infty (X)$. The algebra
$\Diff^*_{{\rm \Phi}} (X)$ is contained in an algebra of  pseudodifferential operator, which can be described in two compatible ways. 

The first way  is purely microlocal and it involves a blown-up double space on which the Schwartz kernel of the pseudodifferential operators can be easily described. The general philosophy  underlying this approach is due to Merlose; for this particular Lie algebra of vector fields the corresponding pseudodifferential calculus is due to Mazzeo and Melrose for manifolds with fibered boundary and to Debord, Lescure and Rochon for manifolds with fibered corners.

The second way to enlarge $\Diff^*_{{\rm \Phi}}(X)$  to a pseudodifferential algebra  is to use the Lie groupoid $G_{{\rm \Phi}}$ and the pseudodifferential algebra  canonically associated to it. 

We shall now briefly recall these two points of view. 

\subsection{The ${{\rm \Phi}}$-calculus} \label{phi-calculus}
In this subsection, following the seminal paper of Mazzeo and Melrose \cite{MMPHI} and its extension to manifolds with fibered corners given in 
\cite{DLR}, we recall briefly how the ${{\rm \Phi}}$-calculus on a manifold with fibered corners $X$
is defined. According to the general philosophy put forward by Richard Melrose, we define the ${{\rm \Phi}}$-pseudodifferential
operators by specifying properties of their Schwartz kernels on $X\times X$ and we do this by lifting
these kernels to a resolved product space, the ${{\rm \Phi}}$-double space. Let $H_i \xrightarrow{\phi_i} S_i$, $i\in\{1,\dots,k\}$, an exhaustive list
of the boundary hypersurfaces of $X$ and let $x_i$ be a boundary defining function for $H_i$. We assume as before that
$$i<j\,\Rightarrow H_i < H_j\;\;\text{or}\;\;H_i\cap H_j=\emptyset\,.$$
In order to define the
${{\rm \Phi}}$-double space we first blow-up the $p$-submanifolds
 \footnote{We recall that a p-submanifold $D$ of a manifold with corners $M$ is a smooth submanifold  which meets all boundary faces of $M$ transversally, and which is covered by coordinate neighborhoods $\{U,(x,y)\}$ in $M$ such that
 	$D\cap U =\{y_j=0\,;\, j=1,\dots,\mathrm{codim}(D)\}$} $H_i\times H_i$ in $X\times X$, thus obtaining
the $b$-double space
$$X^2_b:=[X\times X; H_1\times H_1; H_2\times H_2,\dots ;H_k\times H_k]$$
also denoted $X\times_b X$. Blowing up the corners $H_i\times H_i$
in a different order will result in a  diffeomorphic manifold. The submanifolds $$D_{\phi_i}=\{(h,h')\in H_i \times H_i \,\,\mbox{s.t.}\,\,\phi_i(h)=\phi_i(h') \}$$ of $H_i \times H_i$ can be lifted to 
$X^2_b$ where they become $p$-submanifolds, denoted $\Delta_{\phi_i}$. The ${{\rm \Phi}}$-double space is, by definition,
the space obtained by blowing-up in $X^2_b$ the $p$-submanifolds $\Delta_{\phi_1},\dots\Delta_{\phi_k}$, in this order:
$$X^2_{{\rm \Phi}}:=[X^2_b;\Delta_{\phi_1};\dots;\Delta_{\phi_k}]$$
We shall also denote it by $X\times_{{\rm \Phi}} X$.
There are well defined blow-down   maps $$\beta_b: X^2_b\to X^2\,,\;\;
\beta_{{\rm \Phi}-b}: X^2_{{\rm \Phi}}\to X^2_b\,\;\; \beta_{{\rm \Phi}}:=\beta_b\circ \beta_{{\rm \Phi}-b} : X^2_{{\rm \Phi}}\to X^2;$$ the front faces associated
to the boundary hypersurfaces are, by definition, the $p$-submanifolds
$$\operatorname{ff}_{\phi_i}:= \beta_{{\rm \Phi}-b}^{-1} (D_{\phi_i})\,.$$
They are boundary hypersurfaces of $X^2_{{\rm \Phi}}$. We use the notation $\operatorname{ff}_{{\rm \Phi}}:=\cup_{i=1}^{i=k} \operatorname{ff}_{\phi_i}$.  Of particular importance is also the lifted diagonal
$$\Delta_{{\rm \Phi}}:=\overline{\beta_{{\rm \Phi}}^{-1}(\mathring{\Delta}_X)}$$
If $E$ and $F$ are vector bundles over $X$, then the space of ${{\rm \Phi}}$-pseudodifferential operator of order $m$ 
is defined in terms of the conormal distrubutions at $\Delta_{{\rm \Phi}}$:
\begin{equation*}
\Psi_{{\rm \Phi}}^m(X;E,F) :=\{K\in I^m (X^2_{{\rm \Phi}};\Delta_{{\rm \Phi}};\beta^*_{{\rm \Phi}} {\rm HOM}(E,F)\otimes \pi_R^* (\Omega^{{\rm \Phi}}));
K\equiv 0 \;\text{at}\; \partial X^2_{{\rm \Phi}} \setminus \operatorname{ff}_{{\rm \Phi}})\}
\end{equation*}
where $\equiv 0$ means vanishing of the Taylor series and $\pi_R$ is the map induced by the projection on the
right factor. Classical ${\rm \Phi}$-pseudodifferential operators are defined in terms of polyhomogeneuos conormal distributions; they give the space $\Psi^m_{{\rm \Phi}-{\rm ph}}(X;E,F)$.
It is not difficult to show that the Schwartz kernels of the ${\rm \Phi}$-differential operators of order $m$ lift 
to the ${\rm \Phi}$-double space where they define elements in $I^m_{{\rm ph}} (X^2_{{\rm \Phi}};\Delta_{{\rm \Phi}};\beta^*_{{\rm \Phi}} {\rm HOM}(E,F)\otimes \pi_R^* (\Omega^{{\rm \Phi}}))$; thus 
$$\Diff^m_{{\rm \Phi}}(X;E,F)\subset \Psi_{{\rm \Phi},{\rm ph}}^m(X;E,F)\subset \Psi_{{\rm \Phi}}^m(X;E,F)\,.$$ There is a short exact sequence, the principal symbol sequence, $$0\to \Psi^{m-1}_{{\rm \Phi}} (X;E,F)\to \Psi^{m}_{{\rm \Phi}} (X;E,F)\xrightarrow{\sigma_m}
S^{[m]} ({}^{{\rm \Phi}} TX,p^* {\rm HOM} (E,F))\to 0\,,$$
with $p: {}^{{\rm \Phi}} TX\to X$ the natural projection. In addition to the principal symbol of $P\in 
 \Psi_{{\rm \Phi}}^m(X;E,F)$ one can consider its boundary symbols 
$\sigma_{\partial_i} (P)$, $i\in\{1,\dots,k\}$, also called {\em normal operators}. Defined initially by the restriction of the Schwartz kernel of $P$ to the 
 front face $\operatorname{ff}_{\phi_i}$, the operator $\sigma_{\partial_i} (P)$ is in fact a ${}^{{\rm \Phi}} N S_i$-suspended  family
 of ${\rm \Phi}$-operators, with ${}^{{\rm \Phi}} N S_i=TS_i\times\bbR$; in formulae
 $$\sigma_{\partial_i} (P)\in \Psi_{{\rm \Phi}-{\rm sus}({}^{{\rm \Phi}} N S_i)} (H_i/S_i;E,F)$$
  This means that $\sigma_{\partial_i} (P)$ is a smoothly varying
 family of operators parametrized by
 $S_i$, the base of $H_i$, and such that
 $$ \sigma_{\partial_i} (P)(s)\in \Psi_{{\rm \Phi}-{\rm sus}(T_s S_i\times \bbR)} (\phi_i^{-1} (s);E,F)\,.$$
 In practise  this means that  $\sigma_{\partial_i} (P)(s)$ is an operator on $T_s S_i \times\bbR \times \phi_i^{-1} (s)$
 which is translation invariant in the $T_s S_i \times\bbR$ direction and thus with Schwartz kernel
 in $T_s S_i \times\bbR\times \phi_i^{-1} (s)\times \phi_i^{-1} (s)$. \\
 We shall say that  $P\in \Psi_{{\rm \Phi}}^m (X;E,F)$  is elliptic if 
$\sigma_m (P)$ is invertible off the zero section of ${}^{{\rm \Phi}} TX$; we shall say that $P$ is fully elliptic if it is elliptic and if, in addition, $\sigma_{\partial_i}(P)$ is invertible for each $i\in\{1,\dots,k\}$ \footnote{where invertibility
is meant for each parameter $s\in S_i$ and for each $i\in\{1,\dots,k\}$ as a map from the space of rapidly decreasing sections of $T_s S_i \times\bbR \times \phi_i^{-1} (s)$ into itself}.\\ 
One can prove a composition formula for these operators: if $P\in \Psi_{{\rm \Phi}}^m(X;E,F)$, $Q\in \Psi_{{\rm \Phi}}^\ell (X;F,G)$
then $Q\circ P\in \Psi_{{\rm \Phi}}^{m+\ell}(X;E,G)$; moreover the principal symbol and the boundary symbols are multiplicative.\\
The main result in the ${\rm \Phi}$-calculus is the following parametrix construction:\\ {\it if $P\in \Psi_{{\rm \Phi}}^m(X;E,F)$
is fully elliptic then there exists $Q\in  \Psi_{{\rm \Phi}}^{-m} (X;F,E)$ such that 
\begin{equation}\label{parametrix}
Q\circ P-\Id \in \dot{\Psi}_{{\rm \Phi}}^{-\infty} (X;E)\,,\quad P\circ Q-\Id \in \dot{\Psi}_{{\rm \Phi}}^{-\infty} (X;F)
\end{equation}
where $ \dot{\Psi}_{{\rm \Phi}}^{-\infty}(X;E) $ denotes the  ${\rm \Phi}$-operators of order $(-\infty)$ that vanish of infinite
order at {\em all} boundary hypersurfaces of $X^2_{{\rm \Phi}}$. } \\Notice that elements in  $ \dot{\Psi}_{{\rm \Phi}}^{-\infty}(X;E) $
are in fact defined by smooth kernels in $\dot{C}^\infty (X\times X, {\rm END}(E))$ where the dot means vanishing of infinite order at the boundary.
We point out  that in contrast with the $b$-calculus or the edge calculus, the parametrix of a fully elliptic ${\rm \Phi}$-operator
is an element
of $\Psi^*_{{\rm \Phi}}$; there is no need to enlarge the calculus. Using this crucial information
it is not difficult to show that \begin{equation}\label{made-bded}
P\in \Psi_{{\rm \Phi}}^m(X;E)\Rightarrow P\circ (P^* P + \Id)^{-\frac{1}{2}}\in  \Psi_{{\rm \Phi}}^0 (X;E)
\end{equation}
One can introduce ${\rm \Phi}$-Sobolev spaces and prove the boundedness of  $P\in \Psi_{{\rm \Phi}}^m(X;E,F)$
from $H^\ell_{{\rm \Phi}} (X;E)$ to $H^{\ell-m}_{{\rm \Phi}} (X;E)$; in particular 0-th order ${\rm \Phi}$-pseudodifferential operators
are $L^2$-bounded.\\ 

We close this subsection with a few comments on the equivariant case: if $X$ is the resolution of 
${}^{\textrm{S}} X$ and $X_\Gamma$ is the resolution of a Galois cover $ {}^{\textrm{S}} X_\Gamma$,  $\Gamma - {}^{\textrm{S}} X_\Gamma \to
{}^{\textrm{S}} X$, then we know that $X_\Gamma$ is a Galois cover of $X$,
\begin{equation}\label{cover}
\Gamma  - X_\Gamma\xrightarrow{p} X
\end{equation}
 with an additional property encoded by 
\eqref{compatible-bfs}. Consider in this context the diagonal action of $\Gamma$ on $X_\Gamma\times X_\Gamma$;
there is a lift of this action on $X_{\Gamma}\times_b X_{\Gamma}$ (this is already considered in \cite{LPMEMOIRS}) and on
$$X_{\Gamma}\times_{{\rm \Phi}} X_{\Gamma}$$ and it is therefore possible to define the space of $\Gamma$-equivariant 
${\rm \Phi}$-pseudodifferential operators of order $m$, $\Psi^m_{{\rm \Phi},\Gamma}(X_\Gamma)$, and similarly for 
$\Psi^m_{{\rm \Phi},\Gamma}(X_\Gamma,E_\Gamma)$ if $E_\Gamma$ is a $\Gamma$-equivariant 
complex vector bundle on $X_\Gamma$.\\ An operator is of $\Gamma$-compact support if the support 
of its Schwartz kernel is compact in $X_{\Gamma}\times_{{\rm \Phi}} X_{\Gamma}/\Gamma$. 
We denote by $\Psi^m_{{\rm \Phi},\Gamma,c}(X_\Gamma,E_\Gamma)$ the resulting vector space.
Remark that a  $\Gamma$-equivariant  ${\rm \Phi}$-differential
operator
is certainly of $\Gamma$-compact support.
The composition of two ${\rm \Phi}$-operators of $\Gamma$-compact support is well defined and still a ${\rm \Phi}$-operator
of $\Gamma$-compact support. We thus obtain an algebra  $\Psi^*_{{\rm \Phi},\Gamma,c}(X_\Gamma,E_\Gamma)$.\\ Finally, the boundary symbols of a $\Gamma$-equivariant ${\rm \Phi}$-differential operator 
$D_\Gamma$ can be identified  with the boundary symbols of the  operator $D$ induced  on the quotient $X$. 
Indeed, if $H_\Gamma\xrightarrow{\phi_\Gamma} S_\Gamma$
is a boundary hypersurface of $X_\Gamma$ then we know that there is no action
of $\Gamma$ in the fibers of $H_\Gamma$; $\sigma_{\partial} (D_\Gamma)$ is a $S_\Gamma$-family of operators, which is $\Gamma$-equivariant. Thus   $\sigma_{\partial} (D_\Gamma)(s)$ is an operator on $T_s S_\Gamma \times\bbR \times \phi_\Gamma ^{-1} (s)\equiv  T_{p (s)} S\times \bbR \times \phi^{-1} (\phi(s))$ 
which is, in addition,
$\Gamma$-equivariant with respect to the variable $s$; this means that 
$$\sigma_{\partial} (D_\Gamma)(s)=\sigma_{\partial} (D_\Gamma)(\gamma s);$$
thus, up to the above identifications,  we can obtain the boundary operator of $D_\Gamma$ at the boundary hypersurface $H_\Gamma\to S_\Gamma$ of $X_\Gamma$ from the boundary operator
of $D$
at the  boundary hypersurface $H\to S$ of $X$.\\
This reasoning establishes the following important fact:
\begin{proposition}\label{prop:parametrix-Gamma}
If $D_\Gamma$ is a fully elliptic element in $\Diff^m_{{\Phi},\Gamma}(X_\Gamma,E_\Gamma)$ then
there exists a parametrix $Q_\Gamma$ of $\Gamma$-compact support, $Q_\Gamma\in  
\Psi^{-m}_{{\rm \Phi},\Gamma,c}(X_\Gamma;F_\Gamma,E_\Gamma)$, such that  
\begin{equation}\label{parametrix-Gamma}
Q_\Gamma\circ D_\Gamma-\Id \in \dot{\Psi}_{{\rm \Phi},\Gamma,c}^{-\infty} (X_\Gamma;E_\Gamma)\,,\quad D_\Gamma\circ Q_\Gamma-\Id \in \dot{\Psi}_{{\rm \Phi},\Gamma,c}^{-\infty} (X_\Gamma;F_\Gamma)
\end{equation}
\end{proposition}

\noindent
The same proposition can be stated replacing $D_\Gamma$ by  a fully elliptic element  $P\in \Psi^m_{{\rm \Phi},\Gamma,c}(X_\Gamma,E_\Gamma)$.

\subsection{Pseudodifferential operators on a groupoid $G$}\label{subsect:pseudo-G}
Given a Lie groupoid $G$ it is possible to define an algebra of $G$-pseudodifferential operator $\Psi^*_c (G)$, see \cite{MP, NWX}.

%
%
%
%
%
Let us recall briefly the general definition.
\begin{definition}
	A linear $G$-operator is a continuous linear map
	$$ P\colon C^\infty_c(G,\Omega^{\frac{1}{2}})\to C^\infty(G,\Omega^{\frac{1}{2}})$$
	such that:
	\begin{itemize}
		\item $P$ restricts to a continuous family $(P_x)_{x\in G^{(0)}}$ of linear operators $P_x\colon C^\infty_c(G_x,\Omega^{\frac{1}{2}})\to C^\infty(G_x,\Omega^{\frac{1}{2}})$
		such that
		$$Pf(\gamma)= P_{s(\gamma)}f_{s(\gamma)}(\gamma) \quad \forall f\in C^\infty_c(G,\Omega^{\frac{1}{2}})$$
		where $f_x$ denotes the restriction of $f$ to $G_x:=s^{-1}(x)$.
		\item The following equivariance property holds:
		$$ U_{\gamma}P_{s(\gamma)} =P_{r(\gamma)}U_{\gamma},  $$
		where $U_\gamma$ is the map induced on functions by the right multiplication by $\gamma$.
	\end{itemize}
	
	A linear G-operator $P$ is pseudodifferential of order $m$ if 
	\begin{itemize}
		\item  its  Schwartz kernel $k_P$ is a distribution on $G$ that is smooth outside $G^{(0)}$;
		\item for every distinguished chart $\psi: U\subset G\to \Omega\times s(U)\subset\RR^{n-p}\times\RR^{p}$ of $G$:
		$$  \xymatrix{U\ar[rr]^{\psi}\ar[dr]_{s} & &\Omega\times s(U)\ar[dl]^{p_{2}} \\
			&s(U)&  }$$ 
		the operator $(\psi^{-1})^{*}P\psi^{*}:
		C^{\infty}_{c}(\Omega\times s(U))\to C^{\infty}_{c}(\Omega\times
		s(U))$ is a smooth family parametrized by $s(U)$ of  
		pseudodifferential operators of order $m$ on $\Omega$.     
	\end{itemize} 	
	We say that $P$ is smoothing if $k_P$ is smooth and that $P$ is compactly supported if $k_P$ is compactly supported on $G$.
\end{definition}

\noindent
One can we show that the space $\Psi^*_c(G)$ of compactly supported pseudodifferential $G$-operators is an involutive algebra. 

\medskip
\noindent
{\bf Symbol map.} 
Observe that a pseudodifferential $G$-operator induces a family of pseudodifferential operators on the $s$-fibers.
So we can define the principal symbol of a pseudodifferential $G$-operator $P$ as a function $\sigma(P)$ on 
$\mathfrak{S}^*G$, the cosphere bundle associated to the Lie algebroid $\mathfrak{A}G$ by
$$\sigma(P)(x,\xi)=\sigma_{pr} (P_x)(x,\xi)$$
where $\sigma_{pr}(P_{x})$ is the principal symbol of the pseudodifferential
operator $P_{x}$ on the manifold $G_{x}$.  

\medskip
\noindent
{\bf Quantization.} 
Conversely,  given a symbol $f$ of order $m$ on $\mathfrak{A}^{*}G$ we can quantize it
to a pseudodifferential $G$-operator once we have, in addition, the
following data:
\begin{enumerate}
	\item A smooth embedding  $\theta:  U\to \mathfrak{A}G$,  where $ U$ is a open
	set in $G$ containing $G^{(0)}$, such that
	$\theta(G^{(0)})=G^{(0)}$, $(d\theta)|_{G^{0}}=\hbox{Id}$ and $\theta(\gamma)\in\mathfrak{A}_{s(\gamma)}G$
	for all $\gamma\in U$;
	\item A smooth compactly supported map $\phi:G\to \RR_{+}$ such that
	$\phi^{-1}(1)=G^{(0)}$.
\end{enumerate}
Then a $G$-pseudodifferential operator $P_{f, \theta, \phi}$ is obtained by the
formula:
$$ P_{f, \theta, \phi}u(\gamma)=
\int_{\gamma'\in G_{s(\gamma)}\,,\, \xi\in \mathfrak{A}^{*}_{r(\gamma)}(G)}
e^{-i\theta(\gamma'\gamma^{-1})\cdot \xi}f(r(\gamma),
\xi)\phi(\gamma'\gamma^{-1})u(\gamma')$$
with $u\in C^{\infty}_{c}(G,\Omega^{\frac{1}{2}})$. 
The principal symbol of $P_{f, \theta, \phi}$ is just the leading
part of $f$.  

By definition, operators of zero order $\Psi_c^0(G)$ are a subalgebra of the multiplier algebra
$M(C^{*}_r(G))$ (that is nothing but the algebra of bounded adjointable operators on the $C^*_r(G)$-module given by $C^*_r(G)$ itself), whereas  operators of negative order are actually in
$C^{*}_r(G)$. In what follows  we will denote $\overline{\Psi_c^0(G)}$ the C*-algebra obtained as the closure of $0$-order $G$-pseudodifferential operator in the multiplier algebra $M(C^{*}_r(G))$.
All these definitions and properties immediately extend to the case of operators acting
between sections of bundles on $G^{(0)}$ pulled back to $G$ with the
range map $r$. The space of compactly supported pseudodifferential
operators on  $G$ acting on sections of $r^*E\otimes \Omega^{\frac{1}{2}}$ and taking values in sections of $r^*F\otimes \Omega^{\frac{1}{2}}$
will be denoted $\Psi_c^*(G;E, F)$. 
If $F=E$ we get an algebra
denoted by $\Psi_c^*(G;E)$. 
Notice that $\Psi_c^0(G;E)$ is a subalgebra of $\mathbb{B} (\mathcal{E})$ where $\mathcal{E}$ is the Hilber $C^*_r(G)$-module obtained as the closure 
of $C^\infty_c (G; r^* E\otimes \Omega^{\frac{1}{2}})$  with respect to the obvious  $C^*_r(G)$-norm. Similarly we can define
$\mathcal{F}$  from  $C^\infty_c (G; r^* F\otimes \Omega^{\frac{1}{2}})$ and one can easily see that  $\Psi_c^*(G;E, F)$ is a subalgebra of $ \mathbb{B}(\mathcal{E},\mathcal{F})$. \\

\subsection{Simple ellipticity and K-theory classes}\label{subsect:k-classes-general}
An operator is {\it
	elliptic} when its principal symbol is invertible off the zero section of the dual of the algebroid. If $P$ is elliptic   then,   as in the classical situation, 
it has a parametrix inverting it modulo
$\Psi^{-\infty}_{c}(G)=C^{\infty}_{c}(G)$. See \cite{Vassout-jfa}.\\
 
If $P\in \Psi_c^0 (G;E,F)$ is elliptic and $Q$ is a parametrix
with remainders $R$ and $S$ in $C^{\infty}_{c}(G;r^*E\otimes s^*E^*\otimes \Omega^{\frac{1}{2}})\subset \mathbb{K} (\mathcal{E})$ and 
$C^{\infty}_{c}(G;r^*F\otimes s^*F^*\otimes \Omega^{\frac{1}{2}})\subset \mathbb{K} (\mathcal{F})$ respectively, then we can define
a class in $KK(\bbC, C^*_r (G))$ by considering the Kasparov bimodule defined by the operator
$$T:=\begin{pmatrix} 0&Q\\P&0 \end{pmatrix} $$
and  the Hilbert $C^*_r (G)$-module $\mathcal{H}=\mathcal{E}\oplus \mathcal{F}$.
\\
Let $Y\subset G^{(0)}$ be a closed saturated subset, where we recall that  saturated means that $Y$ is  a
union of $G$-orbits. Then we can consider $G_{|Y}$; we assume that 
\begin{equation}\label{short-saturated-reduced}
\xymatrix{0\ar[r]& C^*_r(G_{|X\setminus Y})\ar[r]&C^*_r(G)\ar[r]& C^*_r(G_{|Y})\ar[r]&0}
\end{equation} 
is exact.
By restriction we obtain a pair $(T_{|Y},H_{|Y})$
which  is a Kasparov $(\bbC,C^*_r (G_{|Y}))$-bimodule. Let us assume that this bimodule is degenerate;
then the Kasparov bimodule  $(T,H)$ actually defines an element  $[T,H]$ in $ KK(\bbC, C^*_r (G_{|X\setminus Y})).$ 
See \cite[Lemma 3.2]{skandalis-exact} for a detailed proof of this fact.

\subsection{The algebra $\Psi^*_c (G^\Gamma_{{\rm \Phi}})$}\label{Pseudo-groupoid}
Let us go back to stratified pseudomanifolds and let us investigate what a pseudodifferential $G^\Gamma_{{\rm \Phi}}$-operator $P$ is.\\
By definition, $P$ is really a family of pseudodifferential operators on the $s$-fibers of 
$G^\Gamma_{{\rm \Phi}}$ 
with an additional
equivariant property. Bearing in mind the objects of the groupoid 
$G^\Gamma_{{\rm \Phi}}$ 
(they are given in terms of  the manifold with fibered corners $X$) we have,
correspondingly, two kinds of families:

\begin{itemize}
	\item let $x$ be a point in the interior of $X$, then the $s$-fiber  over $x$ is just $\mathring{X}_\Gamma$ and $P_x$ is a pseudodifferential ${\rm \Phi}$-operator. Moreover,
	since the action of $G^\Gamma_{{\rm \Phi}}$ 
	is transitive when restricted to $\mathring{X}_\Gamma$, the equivariance property implies that $P_x=P_y$ for all $x,y\in X\setminus \partial X $;
	moreover, since the isotropy group $(G^\Gamma_{{\rm \Phi}})_x^x$ (which we recall is defined as $s^{-1}(x)\cap r^{-1}(x)$) is isomorphic to $\Gamma$, it follows that $P_x$ is $\Gamma$-equivariant for all $x\in X\setminus \partial X $;
	\item let $h$ be a point in the boundary $\partial X$ of $X$ and let $Z$ be the fiber of $\phi$ over $h$; $s^{-1}(h)$ is isomorphic to $Z\times\R^n\times\R$.
	So $P_h$ is a pseudodifferential operator on  $Z\times\R^n\times\R$ and, by the equivariance property
	one has that $P_h$ is translation invariant on the euclidean part and that $P_h=P_k $ for all $h,k\in \partial X$ such that $\phi(h)= \phi(k)$. Notice that on the boundary there is no $\Gamma$-equivariancy condition. 
	\end{itemize}

\subsection{Simple ellipticity versus full ellipticity}	
If a pseudodifferential $G^\Gamma_{{\rm \Phi}}$-operator 
is elliptic, then, as observed before, it is invertible modulo $C^\infty_c (G^\Gamma_{{\rm \Phi}})\subset C^*_r (G^\Gamma_{{\rm \Phi}})$. 
This $C^*$-algebra is too big; for example if  $\Gamma=\{1\}$ then 
$C^*_r (G_{{\rm \Phi}})$ (or even $C^\infty_c (G_{{\rm \Phi}})$)
  is not contained in the algebra of compact operators on $L^2$. This means 
 that, in contrast with the closed compact case,  an element that is elliptic, it is not in general  Fredholm.
 More generally, an elliptic pseudodifferential $G^\Gamma_{{\rm \Phi}}$-operator  will define an index class in
 $K_* ( C^*_r (G^\Gamma_{{\rm \Phi}}))$ but this K-theory group is not very  interesting for the questions we address in this article.
If we want to work with Fredholm operators, or more generally, if we want to define interesting K-theory classes, 
we must  consider   operators that are not only elliptic but also 
have  some additional property that ensures, to the very least, their invertibility modulo compacts. 
On the basis of the microlocal approach explained in the previous section, we know that a sufficient condition
is given by the invertibility of the normal families, but at this stage we want to get to this condition in an autonomous way,
i.e. within the theory of groupoids.  To this end we begin by observing that
when $\Gamma=\{1\}$ the algebra of compact operators on $L^2$ is realized by the $C^*$-algebra 
$C^*_r(\mathring{X}\times\mathring{X})$, the C*-algebra of the restriction of $G_{{\rm \Phi}}$ to $\mathring{X}=X\setminus \partial X$.
So, we are looking for a condition  that would ensure invertibility 
modulo $C^*_r(\mathring{X}\times\mathring{X})$. This is when $\Gamma=\{1\}$.
If we are in the true equivariant case, i.e. $\Gamma\not= \{1\}$, this condition translates to being invertible modulo $C^*_r(\mathring{X}_\Gamma\times_\Gamma\mathring{X}_\Gamma)$, the C*-algebra of the restriction of $G^\Gamma_{{\rm \Phi}}$ to $\mathring{X}$; indeed, it is easy to see that $K_* (C^*_r(\mathring{X}_\Gamma\times_\Gamma\mathring{X}_\Gamma))=K_* (C^*_r \Gamma)$ and the latter is certainly an interesting $K$-theory group.

\smallskip
\noindent
The additional condition we are looking for is precisely  {\it full ellipticity} and our immediate goal now
it to get to this condition within  the groupoid approach.

\smallskip
\noindent
Let us denote by $\partial G^\Gamma_{{\rm \Phi}}$ the restriction of $G^\Gamma_{{\rm \Phi}}$ to $\partial X$ 
and let 
$$
\sigma_\partial\colon \overline{\Psi^0_c(G^\Gamma_{{\rm \Phi}})}\to \overline{\Psi^0_c(\partial G^\Gamma_{{\rm \Phi}})}$$
be the morphism of C*-algebras induced by the restriction to the boundary. Recall that here the closure of the groupoid pseudodifferential *-algebra is taken inside the multiplier algebra of the groupoid C*-algebra. See the end of Section \ref{subsect:pseudo-G}.

\noindent
Then $C^*_r(\mathring{X}_\Gamma\times_\Gamma\mathring{X}_\Gamma)=\ker(\sigma)\cap\ker(\sigma_\partial)=\ker(\sigma\oplus\sigma_\partial)$,
where $\sigma\colon \overline{\Psi^0_c(G^\Gamma_{{\rm \Phi}})}\to C(\mathfrak{S}^*G^\Gamma_{{\rm \Phi}})$ is given by the principal symbol.
If we denote by $\sigma_{f.e.}$ the morphism induced by the universal property associated to the following pull-back diagram
\[
\xymatrix{\overline{\Psi^0_c(G^\Gamma_{{\rm \Phi}})}\ar@/^2pc/[drr]^{\sigma}\ar@/_2pc/[ddr]_{\sigma_\partial}\ar@{-->}[dr]^{\sigma_{f.e.}} & & \\
	& \overline{\Psi^0_c(\partial G^\Gamma_{{\rm \Phi}})}\underset{\partial X}{\times} C(\mathfrak{S}^*G^\Gamma_{{\rm \Phi}})\ar[r]\ar[d]& C(\mathfrak{S}^*G^\Gamma_{{\rm \Phi}})\ar[d]\\
	& \overline{\Psi^0_c(\partial G^\Gamma_{{\rm \Phi}})}\ar[r] &   C(\mathfrak{S}^*\partial G^\Gamma_{{\rm \Phi}})}
\] 
then we obtain the following exact sequence of C*-algebras
\begin{equation}\label{f.e.extention}
\xymatrix{
	0\ar[r]& C^*_r(\mathring{X}_\Gamma\times_\Gamma\mathring{X}_\Gamma)\ar[r] & \overline{\Psi^0_c(G^\Gamma_{{\rm \Phi}})}\ar[r]^(.3){\sigma_{f.e.}}&  \overline{\Psi^0_c(\partial G^\Gamma_{{\rm \Phi}})}\underset{\partial X}{\times} C(\mathfrak{S}^*G^\Gamma_{{\rm \Phi}})\ar[r]& 0
}.
\end{equation}

\begin{definition}\label{fully-groupoid}
	A pseudodifferential $G^\Gamma_{{\rm \Phi}}$-operator $P\in \Psi_c^0(G_{{\rm \Phi}})$ is fully elliptic if 
	the element $\sigma_{f.e.}(P)$ is invertible in $\overline{\Psi^0_c(\partial G^\Gamma_{{\rm \Phi}})}\underset{\partial X}{\times} C(\mathfrak{S}^* G^\Gamma_{{\rm \Phi}})$. We call $\sigma_{f.e.}(P)$ the full symbol of $P$.
\end{definition}

\noindent
From now on we will briefly denote the C*-algebra $\overline{\Psi^0_c(\partial G^\Gamma_{{\rm \Phi}})}\underset{\partial X}{\times} C(\mathfrak{S}^* G^\Gamma_{{\rm \Phi}})$ of  full symbols by $\Sigma$.

\smallskip
\noindent
We postpone the treatment of the K-theory classes associated to a fully elliptic ${\rm \Phi}$-operator to a later section.

\subsection{Comparing $\Psi^*_{{\rm \Phi}}(X)$ and $\Psi^*_c (G_{{\rm \Phi}} )$ and a fundamental remark}
The inclusion of $G_{{\rm \Phi}}$ into $X\times_{{\rm \Phi}} X$ given by \eqref{inclusion}, together
with the remarks made in subsection \ref{Pseudo-groupoid},
lead to the following inclusion of *-algebras
$$\Psi^*_c (G_{{\rm \Phi}})\to\Psi^*_{{\rm \Phi}} (X)$$
with $\Psi^*_c (G_{{\rm \Phi}})$  identified with the subalgebra of  $\Psi^*_{{\rm \Phi}}(X)$
made of operators with Schwartz kernel of compact support in  $\mathring{X^2_{{\rm \Phi}}}\cup \mathring{{\rm ff}}_{{\rm \Phi}}$.
Notice, in particular, that this induces a bijection at the level of differential operators.
Up to this identification, one immediately deduces that if $P$ is an operator in 
$\Psi^*_c (G_{{\rm \Phi}})$, then $\sigma_\partial (P)$ corresponds to the totality of normal operators for $P$; indeed, 
both of them are given by restriction to the front faces. See also the last section of this paper for more  on this point.
Furthermore, it also follows in an obvious way that the notions of fully ellipticity in the two contexts are compatible.
 
 \begin{remark}
As we have just observed, $\Psi^*_c (G_{{\rm \Phi}})$ is smaller than $\Psi^*_{{\rm \Phi}} (X)$ because of the support condition.
In particular  the parametrix of a fully elliptic differential operator is in $\Psi^*_{{\rm \Phi}} (X)$ but not  in 
$\Psi^*_c (G_{{\rm \Phi}})$, given that the parametrix involves the inverses of the normal operators and these are
not compactly supported.
This is usually seen as a drawback of the groupoid approach: the algebra 
$\Psi^*_c (G_{{\rm \Phi}})$ seems to be too small to be of any interest in conjunction with index theory (indeed, index
theory is based on the construction of a parametrix).
 We explain why, in the present context, this is not 
so: 

\smallskip
\noindent
we are following a  K-theoretic approach to index theory; K-theory works well for $C^*$-algebras, given that  for $C^*$-algebras
we have additional results such as Bott periodicity. Because of this, in the groupoid approach to index theory
we are really interested in  
$\overline{\Psi^0_c (G_{{\rm \Phi}})}$, a $C^*$-algebra, rather  than in  $\Psi^0_c (G_{{\rm \Phi}})$, a *-algebra.
 Now, and this is the crucial remark,
$$\overline{\Psi^0_c (G_{{\rm \Phi}})}=\overline{\Psi^*_{{\rm \Phi}} (X)}$$
which means that the difference between $\Psi^*_c (G_{{\rm \Phi}})$ and  $\Psi^*_{{\rm \Phi}} (X)$ will disappear once
we take closures.\\
Similar considerations can be given for the relationship between $\Psi^*_c (G^\Gamma_{{\rm \Phi}})$ and
 $\Psi^*_{{\rm \Phi},\Gamma,c} (X_\Gamma)$
\end{remark}

\noindent
Needless to say,  there are situations, typically involving spectral theory, where the difference
between the two calculi are indeed relevant.

\section{K-theory classes: the microlocal approach}\label{sect:k-classic}
Let  ${}^{\textrm{S}}X$ be a smoothly stratified space of dimension $n$ and let  ${}^{\textrm{S}}X_\Gamma$ be a Galois covering of it. Let $X$ and $X_\Gamma$
be the respective resolutions. 
Let $P$ be a $\Gamma$-equivariant fully elliptic $\Phi$-operator on $X_\Gamma$;
we assume that $P$ has $\Gamma$-compact support. 
 Our next goal  is to define 
\begin{itemize}
	\item a K-homology class $[P]$ in  $K^\Gamma_{n}({}^{\textrm{S}}X_\Gamma)$;
	\item an index class $\Ind (P)$ in $K_n (C^*_r(\Gamma))$.
\end{itemize}
There are two way to do it: the first one is classic and employs a parametrix, the second one employs groupoid techniques. In this section we shall briefly explain the classic, microlocal approach, in the next one we shall explain the groupoid approach, with a final subsection 
 devoted  to the compatibility between the two points of view.  The case of  $\Gamma$-equivariant Dirac operators
 will be treated in a separate section.
 
\subsection{The fundamental class of a fully elliptic operator in $K^\Gamma_* ({}^{\mathrm{S}} X_\Gamma)$}\label{subsect:fclass}
We follow \cite[Section 11]{DLR} but we perform constructions directly in the equivariant case.
Let $P\in\Psi^0_{{\rm \Phi},\Gamma,c}(X_\Gamma)$ be a $\Gamma$-equivariant fully elliptic  ${\rm \Phi}$-pseudodifferential 
operator of $\Gamma$-compact support. Let $Q\in \Psi^0_{{\rm \Phi},\Gamma,c}(X_\Gamma)$ be a parametrix for $P$.
Consider  ${\rm \bf{H}} = L^2_{g_{{\rm \Phi}}}(X_\Gamma)\oplus L^2_{g_{{\rm \Phi}}}(X_\Gamma)$ where $g_{{\rm \Phi}}$ is a $\Gamma$-equivariant $\Phi$-metric
on $X_\Gamma$. Consider the subalgebra 
 $C^\infty_{{\rm \Phi}} (X_\Gamma)\subset
 C^\infty(X_\Gamma)$ of smooth functions that are constant along the fibers of $H_{i,\Gamma}\xrightarrow{\phi_{i,\Gamma}} S_{i,\Gamma}$ 
 for each $i\in\{1,\dots,k\}$. There is a dense inclusion $C^\infty_{{\rm \Phi}} (X_\Gamma)\subset C({}^{\textrm{S}}X_\Gamma)$.
We consider the bounded operator on ${\rm \bf{H}} $ defined by multiplication by $f\in C^\infty_{{\rm \Phi}} (X_\Gamma)$, denoted ${\rm \bf{m}} (f)$. 
This is a ${\rm \Phi}$-operator of order 0 and it is obviously of $\Gamma$-compact support. We obtain in this way a 
representation  ${\rm \bf{m}}  :C({}^{\textrm{S}}X_\Gamma) \to \mathbb{B} ({\rm \bf{H}})$. Let 
$${\rm \bf{P}} = \begin{pmatrix} 0& Q\\P&0 \end{pmatrix}\,,$$
a $\Gamma$-equivariant bounded operator on ${\rm \bf{H}}$. Then, following closely the proof given in \cite[Section 11]{DLR},
one obtains  the following $\Gamma$-equivariant K-homology class
\begin{equation}\label{K-homology}
[{\rm \bf{H}} ,{\rm \bf{m}} ,{\rm \bf{P}} ] \;\;\;\in\;\;\; K^\Gamma_* ({}^{\textrm{S}}X_\Gamma)\,.
\end{equation} 
We shall denote this K-homology class simply as $[P]\in K^\Gamma_* ({}^{\textrm{S}}X_\Gamma)$.

\subsection{The index class in $K_* (C^*_r \pi_1 ({}^{\mathrm{S}} X))$}
Let $P$ and $Q$ be as in the previous subsection. 
We first assume that ${}^{\textrm{S}}X$ is even dimensional. We can define in the usual way 
a $C^*_r \Gamma$-Hilbert module $\mathcal{E}$, see for example \cite{BCH-proper}.
Let $R, S\in \dot{\Psi}_{\Phi,\Gamma,c}^{-\infty} (X_\Gamma)$ be the remainders given by 
Proposition \ref{prop:parametrix-Gamma}. They certainly define elements in $\mathbb{K}(\mathcal{E})$.
We consider the following idempotents in the unitalization of $\mathbb{K}(\mathcal{E})$:
\begin{equation}\label{CS-projector}
p:= \begin{pmatrix} R^2 & R (I+R)Q \\ SP &
I-S^2 
\end{pmatrix}\,,\qquad p_0:= \begin{pmatrix} 0 & 0
\\ 0 & I 
\end{pmatrix} 
\end{equation}
and we set, by definition,
\begin{equation}\label{CS}
\Ind_\Gamma (P):= [p] - [p_0]
\end{equation}
The class in \eqref{CS} is an element  in $K_0 (\mathbb{K}(\mathcal{E}))\simeq K_0 (C^*_r \Gamma)$.
See for example \cite{CM} for the motivation behind this definition.

Classic arguments, see for example  \cite{BCH-proper}, \cite{blackadar}, \cite{WO},
 show that our index class in  $K_0 (C^*_\Gamma)$ 
 is the image of the K-homology class $[P]$, defined by $P$,  through the composition of the assembly
map $\mu_\Gamma: K^\Gamma_0 ({}^{\textrm{S}}X_\Gamma)\to KK (\bbC,C^*_r\Gamma)$
with the  isomorphism $KK (\bbC,C^*_r\Gamma)\simeq K_0 (C^*_r\Gamma)$. With a small abuse of notation
we write
\begin{equation}\label{assembly-index}
\mu_\Gamma [P]=\Ind_\Gamma (P)\,.
\end{equation}
 The odd case can be treated by similar methods.

\section{K-theory classes: the groupoid approach}\label{sect:k-adiabatic}

\subsection{The adiabatic groupoid: an introduction} 
Let $X$ be a closed smooth manifold.
Consider the pair groupoid $X\times X\rightrightarrows X$. Its smooth convolution algebra $C_c^\infty(X\times X, \Omega^{\frac{1}{2}}(\ker dr\oplus\ker ds))$ of  smooth compactly supported half-densities on $X\times X$ is nothing but the *-algebra of the smoothing operators on $L^2(X,\Omega^{\frac{1}{2}})$ 
and $C^*_r(X\times X)$, its reduced C*-algebra, is thus isomorphic to the algebra of compact operators $\bbK(L^2(X,\Omega^{\frac{1}{2}}))$.

The Lie algebroid of $X\times X$ is given by the tangent bundle $TX$: it is a Lie groupoid and, by means of the Fourier transform,  its 
C*-algebra $C^*_r(TX)$ is isomorphic to $C_0(T^*X)$ (notice that  0th order symbols on $X$ are bounded multipliers of this algebra).
By Poincar\'e duality, see \cite{CS}, we know that  $K_*(C_0(T^*X))$ is isomorphic to $KK^*(C(X),\mathbb{C})$, the K-homology of $X$.

Following  Alain Connes we now consider a new groupoid,  
the tangent groupoid of the smooth manifold $X$. As a set it is given by 
\[
(X\times X)^{[0,1]}_{ad}:=TX\times\{0\}\sqcup X\times X\times(0,1]\rightrightarrows X\times[0,1],
\]
This set can be equipped with a suitable smooth structure that we shall now describe. 
On $TX\times\{0\}$ and $X\times X\times(0,1]$ we have the usual topology; the gluing of this two pieces is described in terms of convergence in the following way:
we say that $(x_n, y_n, t_n)\in X\times X\times(0,1]$ converges to $(x,\xi, 0)\in TX\times\{0\}$ if $t_n$ goes to $0$, $x_n, y_n$ tends to $x$ and the tangent vector
$(x_n,(y_n-x_n)/t_n)$ converges to $(x,\xi)$, see \cite{Co}. In Section \ref{sing-fol} we will describe this smooth structure in an alternative way for more general situations, see Example \ref{DNC}.

The evaluation at $0$ defines a $C^*$-algebra homomorphism  $\mathrm{ev}_0\colon C^*_r((X\times X)^{[0,1]}_{ad})\to C^*_r(TX)$. 
which induces a short exact sequence of $C^*$-algebras: 
\begin{equation}
\label{eq4intro-0}\xymatrix{0\ar[r]& C^*_r\left(X\times X\right)\otimes C_0(0,1)\ar[r]& C^*_r\left((X \times X)_{ad}^{[0,1)} \right)\ar[r]& C^*_r(TX)\ar[r]& 0}.
\end{equation}
Consider the long exact sequence in K-theory:
$$\xymatrix{\cdots\ar[r]& K_*\left( C^*_r\left(X\times X\right)\otimes C_0(0,1)\right)\ar[r]& K_*(C^*_r((X \times X)_{ad}^{[0,1)} ))\ar[r]& K_*(C^*_r(TX))\ar[r]^(.7){\delta_{ad}}& \cdots}$$
As already explained in the introduction, we define the adiabatic index homomorphism $\Ind^{\mathrm{ad}}$ as the composition of $\delta_{ad}$ and the Bott isomorphism
$\beta$:
$$ \Ind^{ad}=\beta \circ \delta_{ad}: K_*(C^*_r(TX)) \to K_*\left( C^*_r\left(X\times X\right)\right)=K_* (\bbK) $$
 (needless to say, $K_1 (\bbK)=0$).\\
We have the following important result, see  \cite{Co} \cite{MP} for a proof:
\begin{proposition}\label{compatibility-numeric}
Under the identification of  $K_0(C^*_r(TX)) $ with $K^0 (T^* X)$ and of $K_0 (\bbK)$ with $\bbZ$, the adiabatic index homomorphism
is equal to the Atiyah-Singer analytic index homomorphism.
\end{proposition}

\subsection{The adiabatic groupoid: beyond an introduction}
 Both for explaining the compatibility result and for later use in our treatment of rho classes, we are now going 
 to treat more in detail the adiabatic connecting homomorphism.

\noindent 
One can prove the following:
\begin{itemize}
	\item  let us consider the $C^*$-algebra homomorphism  $\mathrm{ev}_0\colon C^*_r((X\times X)^{[0,1]}_{ad})\to C^*_r(TX)$
	induced by evaluation at 0. Since the kernel of this homomorphism is a cone, which is  K-contractible, the element $[\mathrm{ev}_0]$ of $ KK(C^*_r((X\times X)^{[0,1]}_{ad}), C^*_r(TX))$ is a KK-equivalence, namely there exists an inverse element $[\mathrm{ev}_0]^{-1}\in KK(C^*_r(TX), C^*_r((X\times X)^{[0,1]}_{ad}))$. 
	\item 
	If $\sigma$ is the symbol of an elliptic $0$-order pseudodifferential operator $P$ on $X$, then we can describe the image of its class through $[\mathrm{ev}_0]^{-1}$ in the following way:\\ take the pull-back of $\sigma$   to $T^*X\times[0,1]$, the dual Lie algebroid of $(X\times X)_{ad}^{[0,1]}$; this is the symbol $\sigma\times id_{[0,1]}$; it produces an elliptic pseudodifferential operator $P_{ad}$ on $(X\times X)_{ad}^{[0,1]}$, whose restriction at $1$ is the pseudodifferential operator $P$ and whose restriction at $0$ is the Fourier transform of $\sigma$; thus we 
	see that  $[\sigma]\otimes_{C^*_r(TX)}[\mathrm{ev}_0]^{-1}=[P_{ad}]$, where on the right hand side we have the class associated to the elliptic operator $P_{ad}$, as explained in subsection \ref{subsect:k-classes-general}; 
	\item if we denote by $\mathrm{ev}_1\colon C^*_r((X\times X)_{ad}^{[0,1]})\to C^*_r(X\times X)$  the evaluation at 1, then the  KK-element \begin{equation}\label{indexmap}[\mathrm{ev}_0]^{-1}\otimes_{C^*_r((X\times X)^{[0,1]}_{ad})}[\mathrm{ev}_1]\in KK(C^*_r(TX),C^*_r(X\times X))\end{equation}
	is, up to the Bott isomorphism, equal to the boundary morphism $\delta_{ad}$ of \eqref{ad-e-s}; it follows that 
	$\Ind^{ad}: K_* (C^*_r(TX))\to K_*(C^*_r (X\times X))=K_* (\mathbb{K})$ is given by 
	\begin{equation}\label{factorization-ind-ad}
	\Ind^{ad}=[\mathrm{ev}_0]^{-1}\otimes[\mathrm{ev}_1].
	\end{equation}
\end{itemize}


\begin{remark}
	
Similarly if we have a Galois $\Gamma$-covering $X_\Gamma \to X$, we can consider the  Lie groupoid $X_\Gamma\times_\Gamma X_\Gamma\rightrightarrows X$ and its adiabatic deformation 
\[
TX\times\{0\}\sqcup X_\Gamma \times_\Gamma X_\Gamma\times(0,1]\rightrightarrows X\times[0,1].
\]
Using the same arguments as above, taking the $\Gamma$-equivariant pseudodifferential operator associated to an equivariant symbol $\sigma$, we can define the adiabatic index class in  $K_*(C^*_r(X_\Gamma\times_\Gamma X_\Gamma))\cong K_*(C^*_r(\Gamma))$, using the boundary morphism associated to the following exact sequence
\begin{equation}
\label{eq4intro}\xymatrix{0\ar[r]& C^*_r\left(X_\Gamma\times_\Gamma X_\Gamma\right)\otimes C_0(0,1)\ar[r]& C^*_r((X_\Gamma\times_\Gamma X_\Gamma)_{ad}^{[0,1)})\ar[r]& C^*_r(TX)\ar[r]& 0}.
\end{equation}
Moreover this class is proved to be equal to the index class we defined using an equivariant parametrix. 
In the following we will adapt this construction to the context of fully elliptic operators on singular manifolds.
\end{remark}
\begin{remark}

This construction extends to any Lie groupoid $G$. We refer the reader to Section \ref{sing-fol} for 
the details.
\end{remark}

\subsection{The noncommutative tangent bundle and the noncommutative symbol}\label{subsect:nctb-ncs}
Let us now move to the case of a singular manifold ${}^{\textrm{S}}X$. Let $X$ be its resolution.
Following the general philosophy we have explained in the closed case, we are led to consider the adiabatic deformation of the  Lie groupoid $G^\Gamma_{{\rm \Phi}}\rightrightarrows X$. 

\medskip
\noindent
{\it From now on we denote the Lie groupoid $G^\Gamma_{{\rm \Phi}}\rightrightarrows X$ simply by 
$G\rightrightarrows X$.}

\medskip
\noindent
We recall that the  Lie algebroid  of $G\rightrightarrows X$ is ${}^{{\rm \Phi}} TX$ and that the Fourier transform of 
the symbol of a $\Gamma$-equivariant 0th-order ${\rm \Phi}$-pseudodifferential operator of $\Gamma$-compact
support defines a class in $K_*(C^*_r({}^{{\rm \Phi}} TX))\simeq K_* (C_0 ({}^{{\rm \Phi}} T^*X))$.
In contrast with the closed case 
this K-group is not isomorphic to the K-homology of ${}^{\textrm{S}}X$. Indeed Debord, Lescure and Rochon \cite{DLR}
building on work of Debord and Lescure \cite{DL-poincare},  proved that there is another Lie groupoid $T^{NC}X\rightrightarrows X\times\{0\}\cup \partial X\times (0,1)$, built up from the adiabatic deformation of $G$, such that the K-theory of its C*-algebra is isomorphic to the K-homology group $K_*({}^{\textrm{S}}X)$.

The  idea is that the right objects to consider are not  elliptic symbols but elliptic symbols together with the  non-commutative symbols, i.e. the normal operators.

\begin{definition}
We define $T^{NC}X$, the non-commutative tangent bundle of $X$, in the following way: 
first  we take the adiabatic deformation of $G$
\[
G^{[0,1]}_{ad}:={}^{{\rm \Phi}} TX\times\{0\}\cup G\times(0,1]\rightrightarrows X\times[0,1], 
\]
then we restrict to $\mathring{X}\times \{0\}\cup \partial X\times [0,1)$.
\end{definition}

We obtain the disjoint union of $T(X\setminus\partial X)$, the tangent bundle of the interior of $X$,  and of the adiabatic deformation of 
the restriction of $G$ to $\partial X$, $\pa G$, open at 1:
\begin{equation}\label{NC}
T\mathring{X}\cup (\partial G)^{[0,1)}_{ad}\rightrightarrows\mathring{X}\times \{0\}\cup \partial X\times [0,1)\,.
\end{equation}
We shall see that this last piece of information, i.e. the fact that the deformation is open at 1,  is directly linked to the full ellipticity of the operator.
 The C*-algebra of the groupoid $T^{NC}X$ fits into the following 
exact sequence, obtained by restriction:
\begin{equation}
\label{eq4intro}\xymatrix{0\ar[r]& C^*_r\left(\mathring{X}_\Gamma \times_\Gamma \mathring{X}_\Gamma \right)\otimes C_0(0,1)\ar[r]& C^*_r\left(G_{ad}^{[0,1)}\right)\ar[r]& C^*_r(T^{NC}X)\ar[r]& 0}.
\end{equation}

Let $\sigma$ be the symbol of a fully elliptic pseudodifferential  $G$-operator $P$ 
of order $0$. As we saw in the previous subsection for the closed case, it defines a pseudodifferential $G^{[0,1]}_{ad}$-operator $P_{ad}$ that 
restricts to the Fourier transform of $\sigma$ at 0 and to $P$ at 1.

Recall from Section \ref{Pseudo-groupoid} that a $G$-operator, with $G=G_{{\rm \Phi}}^\Gamma$, is a family of operators parametrized by $X$: the restriction of $P$ to the interior $X\setminus\partial X$ corresponds to a $\Gamma$-equivariant pseudodifferential $\Phi$-operator $P$ 
of $\Gamma$-compact support whereas  the restriction of $P$ to $\partial X$ corresponds to the collection
of all normal families of $P$. 

It is crucial to observe that $P_{ad}$, the adiabatic operator associated to a fully elliptic operator,
 not only defines a class in $KK^*(\bbC,C^*_r(G^{[0,1]}_{ad}))$ 
but also a class in $KK^*(\bbC,C^*_r(G^{[0,1],F}_{ad}))$, where  $ C^*_r(G^{[0,1],F}_{ad})$ is the restriction of 
$G^{[0,1]}_{ad}$ to $X\times[0,1]\setminus \partial X\times\{1\}$ \footnote{here the superscript $F$ stands for {\it Fredholm}}.
Indeed, using the discussion given at the end of Subsection \ref{subsect:k-classes-general}, this is true precisely because $P$ is fully elliptic: its restriction to $\partial X\times \{1\}$,
which is given by the normal families of $P$, is invertible and we can find a parametrix $\mathcal{Q}$ of $P$
whose normal operators are the inverses of the normal operators of $P$.
Summarizing, from a fully elliptic operator $P\in \Psi^0_c (G)$ we have defined a class $[P_{ad}^F]$
in  $KK^*(\bbC, C^*_r(G^{[0,1],F}_{ad}))$.
We now observe that there is KK-equivalence between $C^*_r(G^{[0,1],F}_{ad})$ and $C^*_r (T^{NC}X)$
induced by  the restriction of $G^{[0,1],F}_{ad}$ to $T^{NC}X$ (indeed, the kernel of the restriction
homomorphism is $C^*_r (\mathring{X}_\Gamma\times_\Gamma \mathring{X}_\Gamma\times (0,1])\cong C^*_r (\mathring{X}_\Gamma\times_\Gamma \mathring{X}_\Gamma)\otimes C_0((0,1])$ which
is K-contractible).\\
Summarizing, the following definition is well-posed:
\begin{definition}\label{sigmanc}
Let $P$ be a fully elliptic operator with symbol $\sigma$.
The restriction of the operator $P^F_{ad}$ to $\mathring{X}\times \{0\}\cup \partial X\times [0,1)$ defines a class
\begin{equation*}
[\sigma_{nc} (P)]\in K_*(C^*_r(T^{NC}X)).
\end{equation*}
We call this class the noncommutative symbol of  $P$.\\
\end{definition}

Generalizing the discussion in the previous subsection, see in particular \eqref{factorization-ind-ad}, we 
now define the adiabatic index  of $P$ as the image of $[\sigma_{nc} (P)]$ through the boundary morphism associated to \eqref{eq4intro}:

\begin{definition}
	The adiabatic index of a non-commutative symbol $[\sigma_{nc} (P)]\in K_*(C^*_r(T^{NC}X))$ is defined as its image through the following composition
	\begin{equation}\label{index-homo-nc}
	\xymatrix{K_*(C^*_r(T^{NC}X))\ar[r]^(.5){[\mathrm{ev}_0]^{-1}}& K_*(  C^*_r(G^{[0,1],F}_{ad}))  )\ar[r]^(.45){[\mathrm{ev}_1]}& K_*(C^*_r(\mathring{X}_\Gamma\times_\Gamma\mathring{X}_\Gamma))}.
	\end{equation}
Thus
\begin{equation}\label{index-homo-nc-bis}
\Ind^{{\rm ad}}:= [\mathrm{ev}_1] \circ [\mathrm{ev}_0]^{-1}
\end{equation}
We set
\begin{equation}\label{index-homo-nc-ter}
 \Ind^{{\rm ad}} (P):= [\mathrm{ev}_1] \circ [\mathrm{ev}_0]^{-1} ([\sigma_{nc} (P)])\in K_*(C^*_r(\mathring{X}_\Gamma\times_\Gamma\mathring{X}_\Gamma))\simeq K_* (C^*_r \Gamma).
 \end{equation}

\end{definition}

\subsection{Rho classes and delocalized APS index theorem}

Let us now assume that  there is a homotopy $\{P_t\}$
from $P_0=P$ to an invertible operator $P_1$, so that the adiabatic index class of $P$ in $K_* (C^*_r \Gamma)$  is zero. Thus the concatenation of $P_{ad}$ and $P_t$, after a suitable reparametrization, defines a 
$G^{[0,1]}_{ad}$-operator such that its restriction to $X\times\{1\}$ defines a degenerate Kasparov cycle.
Again, using the discussion given at the end of Subsection \ref{subsect:k-classes-general}, it follows that this operator actually defines a class in $KK(\bbC,C^*_r(G^{[0,1)}_{ad}))$. 
\begin{definition}\label{rho-definition}
We define $\rho(P,\{P_t\})$ as the class in $KK(\bbC, C^*_r(G^{[0,1)}_{ad}) )$ that we have just constructed. It depends of course on $P$ but also  on the homotopy $\{P_t\}$.
If $P$ itself is invertible then we omit the constant path in the notation and  write simply $\rho(P)$. 
\end{definition}
Notice that this definition is a particular case of a general definition for Lie groupoids given in \cite{Zenobi:Ad}. It is proved in that paper
that these rho-classes are linked to Atiyah-Patodi-Singer index classes on groupoids
with boundary, a result that generalizes to the groupoid context the delocalized APS index theorem for Galois coverings
proved in \cite{PS-Stolz}.
In this article we shall only consider the following  special case of this very general  version of the delocalized APS index theorem.\\

Consider a stratified pseudomanifold ${}^{\textrm{S}}X$ and the associated cylinder ${}^{\textrm{S}}X\times [0,1]$.
Let ${}^{\textrm{S}}X_\Gamma$ be its universal covering.
Let $X_\Gamma$ denote, as usual, the resolved manifold associated to
${}^{\textrm{S}}X_\Gamma$. Let $G$ be the groupoid $G^\Gamma_{{\rm \Phi}}$.
We consider the Lie groupoid $\mathcal{G}$ given, by definition, by $G\times b[0,1]\rightrightarrows X\times [0,1]$, that is, the product of 
$G$ with the $b$-groupoid of $[0,1]$, see \cite{monthubert-pams}.  There is a natural notion of full ellipticity for $\mathcal{G}$.
Now, let us consider $\mathcal{G}_{ad}^F$, defined as the restriction of $\mathcal{G}_{ad}$ to $X\times[0,1]\times[0,1]\setminus\partial(X\times[0,1])\times\{1\}$. 
A fully elliptic  operator $A$ on $X\times[0,1]$ defines, as before, a class $[A_{ad}^F]\in KK^*(\bbC,C^*_r(\mathcal{G}_{ad}^F))$.
If we restrict the groupoid to the boundary of the cylinder, we obtain the groupoid
$G^{[0,1)}_{ad}\times\{0,1\}\times\bbR\rightrightarrows X\times\{0,1\}$.

We assume that the normal operators associated to the boundary hypersufaces $X\times\{i\}$ of the cylinder
are of the form $A_i+\partial$ where $A_i\in \Psi^0_c(G)$ is invertible for $i=1,2$ and $\partial$ is a translation invariant operator on $\bbR$ that generates $KK^1(\bbC,C_0(\bbR))$.

 Since the restriction of the groupoid $X\times \{0,1\}$ is equal to the product of the two groupoids $G^{[0,1)}_{ad}\times \{0,1\}$ and $\RR$, it is easy to see that the image of $[A_{ad}^F]$ through the restriction $ev_{\{0,1\}}$ to $X\times \{0,1\}$ is the exterior Kasparov product of $\rho(A_0)\oplus -\rho(A_1)\in KK^{*-1}(\bbC, C^*_r( G^{[0,1)}_{ad})\oplus C^*_r(G^{[0,1)}_{ad}))$ with $[\partial]$. 
Instead, if we evaluate the class  $[A_{ad}^F]$ at the adiabatic deformation parameter $t=1$, we obtain the class 
$\operatorname{Ind}^{{\rm ad}}(A)$ in $ KK^*(\bbC,C^*_r(\mathring{X}_\Gamma\times_\Gamma \mathring{X}_\Gamma\times (0,1)\times(0,1))$, 
which we identify with to $KK^*(\bbC,C^*_r(\mathring{X}_\Gamma\times_\Gamma \mathring{X}_\Gamma))$
through Bott isomorphism. \\
Consider the two following homomorphisms:
\begin{itemize}
\item $\alpha\colon KK^*(\bbC, C^*_r( G^{[0,1)}_{ad})\oplus C^*_r( G^{[0,1)}_{ad}))\otimes C_0(\bbR))\to KK^{*+1}(\bbC,  C^*_r( G^{[0,1)}_{ad})) $, which is the composition of the sum $a\oplus b\mapsto a+b$ and the Bott periodicity isomorphism;
\item $\beta\colon KK^*(\bbC,C^*_r(\mathring{X}_\Gamma\times_\Gamma \mathring{X}_\Gamma))\to KK^{*+1}(\bbC,  C^*_r( G^{[0,1)}_{ad}) ) $, which is the composition of  Bott periodicity with the obvious injection $\iota\colon C^*_r(\mathring{X}_\Gamma\times_\Gamma \mathring{X}_\Gamma\times(0,1))\to  
C^*_r( G^{[0,1)}_{ad})$.
\end{itemize}
The following result is a special case of \cite[Theorem 3.6]{Zenobi:Ad}

\begin{theorem}\label{deloc-aps}
If $A$ is as above, then $\alpha (\mathrm{ev}_{\{0,1\}}[A_{ad}^F])=\iota_*\circ\beta(ev_1[A_{ad}^F])$. 
Put it differently,
the following delocalized APS index formula holds:
\[
\rho(A_0)-\rho(A_1)= \iota_*\circ\beta(\operatorname{Ind}^{ad}(A))\in KK^{*+1}(\bbC,  C^*_r( G^{[0,1)}_{ad}) ). 
\]
In particular if $\operatorname{Ind}^{ad}(A)$ vanishes,  then $\rho(A_0)=\rho(A_1)$.
\end{theorem}

\section{Compatibility of K-theory classes: results}\label{sect:compatibility-facts}

The following result is discussed in the work of Debord, Lescure and Rochon  \cite{DLR}:

\begin{theorem}[Poincar\'{e} duality]
	The K-theory of $C^*_r(T^{NC}X)$ is isomorphic to the equivariant K-homology of the stratified manifold ${}^{\textrm{S}}X_\Gamma$.
	Moreover, under this isomorphism, the equivariant K-homology class  $[P]\in K^\Gamma_* ({}^{\textrm{S}}X_\Gamma)$ of a fully elliptic operator $P$ in 
	$\Psi^0_c (G^\Gamma_{{\rm \Phi}})$ 
	corresponds to	the class 
	$[\sigma_{nc}(P)]\in K_*(C^*_r(T^{NC}X))$ defined above.
	\end{theorem}
%

Now we want to compare the index class defined through the adiabatic deformation and the index class defined in the classical way:

\begin{theorem}\label{equality-index-classes}
Let $P\in	
	\Psi^0_c (G^\Gamma_{{\rm \Phi}})$  be a fully elliptic $\Phi$-operator. Then the adiabatic  index class 
	and the index class
	defined in \eqref{CS} through the Connes-Skandalis projector, are equal. In formulae:
	\begin{equation}\label{=}
	 \Ind^{{\rm ad}} (P)= \Ind_\Gamma (P) \quad\text{in}\quad K_* (C^*_r \Gamma)
	\end{equation}

	In fact, more is true, in that there is a commutative diagram with vertical maps isomorphisms of abelian groups:
	\begin{equation}\label{commutative-equality}
\xymatrix{K_*(C^*_r(T^{NC}X))\ar[r]^(.5){\Ind^{{\rm ad}}}\ar[d]_{{\rm PD}}&K_*(C^*_r(\mathring{X}_\Gamma\underset{\Gamma}{\times}\mathring{X}_\Gamma))\ar[d]^{\simeq}\\
	K^\Gamma_* ({}^{\textrm{S}}X_\Gamma) \ar[r]^(.5){\mu_\Gamma}& K_* (C^*_r \Gamma)	}
	\end{equation}
	\end{theorem}
	Consequently, if $[\sigma_{nc} (P)] \in K_*(C^*_r(T^{NC}X))$ is the non-commutative symbol
	of a fully elliptic $\Gamma$-equivariant $\Phi$-operator $P$ and $[P]\in K^\Gamma_* ({}^{\textrm{S}}X_\Gamma)$ is its K-homology class then, identifying the groups on the right hand side through the right vertical isomorphism, we 
	have 
	\begin{equation}\label{=bis}
	 \Ind^{{\rm ad}} (P)\equiv  [ev_1] \circ [ev_0]^{-1} ([\sigma_{nc} (P)])= \mu_\Gamma ({\rm PD}([\sigma_{nc} (P)])) \equiv
	  \mu_\Gamma ([P]) 
	  	  \quad\text{in}\quad K_* (C^*_r \Gamma)
	\end{equation}
	

\section{Compatibility of K-theory classes: proofs}\label{sect:compatibility-proofs}

The Poincar\'e duality isomorphism is proved in great detail in \cite{DLR}, building on \cite{DL-poincare}. We therefore concentrate on the 
proof of Theorem \ref{equality-index-classes}.
\subsection{Prelimaries to the proof of Theorem \ref{equality-index-classes}: relative K-Theory.}\label{subsect:preliminaries}
We begin by giving some basic definitions and results concerning relative K-theory and excision.
Let $A$ and $B$ be $C^*$-algebras and let $\phi:A\to B$ a homomorphism of $C^*$-algebras.
We define $K_* (\phi)$ as the homotopy classes of triples $(p,q,z)$, where $p,q$ are projections in $\mathbb{M}_\infty(A)$ and $z$ is an invertible element in $\mathbb{M}_\infty(B)$ such that $z\phi(p)z^{-1}=\phi(q)$.\\
This is a special case of the notion, treated in \cite{SkExt}, of the K-theory group associated
to an element in $KK(A,B)$, that in this case is the element associated to $\phi$. It is proved in this article that there is an isomorphism
of abelian groups
\begin{equation}\label{mapping-cone-iso}
K_* (\phi)\simeq K_* (C_\phi)
\end{equation}
where $C_\phi$ denotes  the mapping cone of $\phi:A\to B$: $$C_{{\rm \Phi}}:=\{(a,f)\in A\oplus C_0 ([0,1),B) \;\;;\;\;
\phi(a)=f(0)\}.$$
The relative K-theory group of a morphism fits into the following long exact sequence:
\begin{equation}\label{long-relative}
\dots\rightarrow K_* (B\otimes C_0 (0,1)) \rightarrow K_* (\phi)\rightarrow K_* (A)\xrightarrow{\phi} K_* (B)\rightarrow\dots
\end{equation}
It turns out that if $\phi$ is surjective  then there is an excision isomorphism
\begin{equation*}
K_* (\phi)  \xrightarrow{{\rm ex}} K_* (\Ker (\phi))\,.
\end{equation*}
Finally if we have a commutative diagram 
\begin{equation}\label{commutative-relative}
\xymatrix{A\ar[r]^(.5){\phi}\ar[d]_{\alpha}&B\ar[d]^{\beta}\\
	A^\prime\ar[r]^(.5){\phi^\prime}& B^\prime	}
	\end{equation}
	then there is an associated group homomorphism 
	\begin{equation}\label{commutative-relative-bis}
	K_* (\phi)\xrightarrow{(\alpha,\beta)_*}  K_* (\phi^\prime)\,.
	\end{equation}
The diagram  \eqref{commutative-relative} induces a mapping of the long exact sequence \eqref{long-relative} for $A\xrightarrow{\phi}
B$
to the one for $A^\prime \xrightarrow{\phi^\prime} B^\prime$ and it is clear from the five-lemma that if $\alpha_*$ and $\beta_*$
are isomorphism of K-theory groups, then so is $K_* (\phi)\xrightarrow{(\alpha,\beta)_*}  K_* (\phi^\prime)$.

\subsection{
More preliminaries: the class $\Ind_\Gamma (P)\in K_* (C^*_r \Gamma)$ through relative K-Theory.}\label{relative}
%
We set
\begin{equation}\label{algebras-sigma}
\Sigma:= \overline{\Psi^0 (G^\Gamma_{{\rm \Phi}})}/ C^*_r (\mathring{X}_\Gamma \times_\Gamma \mathring{X}_\Gamma)\,.
\end{equation}

We define a morphism $\mathfrak{m}: C(X)\to \Sigma$ as the composition
$$C(X)\xrightarrow{m} \overline{\Psi^0 (G^\Gamma_{{\rm \Phi}})} \xrightarrow{\sigma_{f.e.}} \Sigma$$
with the first arrow given by the multiplication operator and the second arrow denoting the morphism
already considered in \eqref{f.e.extention}.

Now let $E,F$ be two vector bundles on $X$, let 
$(\sigma, N(P_\sigma))$ be a fully elliptic non-commutative symbol between $E$ and $F$. These data define 
a class 
\begin{equation}\label{relative-class}
	[P]_{rel}=
\left[\begin{pmatrix}1_E &0 \\0&0\end{pmatrix}, \begin{pmatrix}0 &0 \\0& 1_F\end{pmatrix}; \begin{pmatrix}0 &(\sigma, N(P_\sigma))^{-1}) \\(\sigma, N(P_\sigma))&0\end{pmatrix}\right]
\end{equation}
 in $K_*(\mathfrak{m})$, the K-theory of the morphism $\mathfrak{m}$, that is a pair of 
$C(X)$-modules that are isomorphic by means of $(\sigma, N(P_\sigma))$ when we see them as $\Sigma$-modules   
through $\mathfrak{m}$.

Since the following diagram
\[
\xymatrix{C(X)\ar[r]^(.5){\mathfrak{m}}\ar[d]_{M}&\Sigma \ar[d]^{\id}\\
	\overline{\Psi^0_c(G_{{\rm \Phi}})}\ar[r]^(.6){\sigma_{f.e.}}& \Sigma	}
\]
is commutative, we have a map 
\[
M_*\colon K_*(\mathfrak{m})\to K_*(\sigma_{f.e.}).
\]

Since $\sigma_{f.e.}$ is a quotient map we see that  $K_*(\sigma_{f.e.})$ is 
isomorphic, through the inverse of the excision map, to
$K_* (\Ker(\sigma_{f.e.}))= K_*(C^*_r(\mathring{X}_\Gamma \times_{\Gamma} \mathring{X}_\Gamma))$. \\
 The inverse of the excision map is given by the following two steps: first we represent any class in $K_*(\sigma_{f.e.})$ by a relative cycle $[p,q,z]$ with respect to 
 $\mathbf{q}\colon C^*_r(\mathring{X}_\Gamma \times_{\Gamma} \mathring{X}_\Gamma)^+\to \bbC$, the quotient map from the unitization of $C^*_r(\mathring{X}_\Gamma \times_{\Gamma} \mathring{X}_\Gamma)$ to $\bbC$; then $[p,q,z]$ is sent to $[p]-[q]$, which, by definition of $K_*(\mathbf{q})$, is a class in $K_*(C^*_r(\mathring{X}_\Gamma \times_{\Gamma} \mathring{X}_\Gamma))$.\\
 More explicitely, let us consider the class \eqref{relative-class}. If we take the following invertible lift 
 \[
 T=\begin{pmatrix} 1_E-QP& -(1_E-QP)Q+Q\\ P& 1_F-PQ\end{pmatrix}\quad\text{of}\quad \begin{pmatrix}0 &(\sigma, N(P_\sigma))^{-1}) \\(\sigma, N(P_\sigma))&0\end{pmatrix},
 \]
 where $Q$ is a full parametrix of $P$,
 and if $T_t$ is a path through invertible elements from $1_E\oplus 1_F$ to $T$ (that always exists by the Theorem of Kuiper), then
 \[
 \left[T_t\begin{pmatrix}1_E &0 \\0&0\end{pmatrix}T_t^{-1}, \begin{pmatrix}0 &0 \\0& 1_F\end{pmatrix}; \begin{pmatrix}0 &(\sigma, N(P_\sigma))^{-1}) \\(\sigma, N(P_\sigma))&0\end{pmatrix}\sigma_{f.e.}(T_t^{-1})\right]
 \]
 is a homotopy from the cycle in \eqref{relative-class} to a relative cycle with respect to $\mathbf{q}$, whose image in $K_*(C^*_r(\mathring{X}_\Gamma \times_{\Gamma} \mathring{X}_\Gamma))$
 is given by:
 \begin{equation}\label{cm-projection}
 T\begin{pmatrix}1_E  & 0\\ 0& 0\end{pmatrix}T^{-1}- \begin{pmatrix}0  & 0\\ 0& 1_F\end{pmatrix}\in \mathbb{M}_2(C^*_r(\mathring{X}_\Gamma\times_\Gamma\mathring{X}_\Gamma)).
 \end{equation}

\begin{definition}\label{analytical index}
	The relative analytical index $\Ind^{{\rm rel}}$ of a fully elliptic non-commutative symbol
	$(\sigma, N(P_\sigma))$ with coefficients in two vector bundle $E,F$ is given by
	the following composition 
	\[
	\xymatrix{K_*(\mathfrak{m})\ar[r]^{M_*}& K_*(\sigma_{f.e.})\ar[r]^(.35){\operatorname{ex}}&K_*(C^*_r(\mathring{X}_\Gamma 
	\times_\Gamma \mathring{X}_\Gamma))}.
	\]
	\end{definition}

\begin{remark}
It is a classic fact that the boundary map $\partial_{full}$ in the K-theory sequence
associated to 
\[\xymatrix{
	0\ar[r]& C^*_r(\mathring{X}_\Gamma\times_\Gamma\mathring{X}_\Gamma)\ar[r] & \overline{\Psi^0_c(G^\Gamma_{{\rm \Phi}})}\ar[r]^(.6){\sigma_{f.e.}}&  \Sigma\ar[r]& 0
}
\]
sends a fully elliptic symbol $(\sigma,N(P_\sigma))$ to the class \eqref{cm-projection}.\\
Consider now the K-theory homomorphism $\psi: K_{1} (\Sigma)\to K_0 (\mathfrak{m})$, which 
 is the natural morphism that associates to an element $z\in \operatorname{GL}_k(\Sigma)$ the relative cycle $(1_k,1_k, z)$; then the composition of $\psi$ with 
the relative index homomorphism $\Ind^{{\rm rel}}$ of Definition \ref{analytical index} is easily seen to be equal to the above boundary map .
Put it differently, the following diagram 
\[
\xymatrix{K_{*+1}(\Sigma)\ar[d]\ar[rrd]^{\partial_{full}}& & \\
	K_*(\mathfrak{m})\ar[r]^{M_*}& K_*(\sigma_{f.e.})\ar[r]^(.35){\operatorname{ex}}&K_*(C^*_r(\mathring{X}_\Gamma 
	\times_\Gamma \mathring{X}_\Gamma)) }
\]
is commutative.
\end{remark}
%


\subsection{Proof of Theorem \ref{equality-index-classes}}
We begin by introducing some useful notation:
\begin{itemize}
\item $\Sigma^F := \overline{\Psi^0_c (G^{[0,1]}_{ad})}/ C^*_r (G_{ad}^{[0,1],F})$ and $\sigma^F: \overline{\Psi^0_c (G)}  \to \Sigma^F $ is the associated quotient map;
\item $T^{{\rm NC}}_{[0,1]}X$ is defined as the restriction of $G^{[0,1]}_{ad}$ to $\mathring{X}\times \{0\}
\cup \partial X\times [0,1]$;
\item $\Sigma^{{\rm NC}} :=  \overline{\Psi^0_c (T^{{\rm NC}}_{[0,1]}X)}/C^*_r (T^{{\rm NC}} X)$ and  $\sigma^{{\rm NC}}:  \overline{\Psi^0_c (T^{{\rm NC}}_{[0,1]}X)}\to \Sigma^{{\rm NC}} $ is the associated quotient map;
\item $M^{ad} : C(X\times [0,1])\to \overline{\Psi^0_c ((G^\Gamma_{{\rm \Phi}})_{ad})}$ 
and  $M^{{\rm NC}} : C(\mathring{X}\times \{0\}
\cup \partial X\times [0,1])\to  \overline{\Psi^0_c (T^{{\rm NC}}_{[0,1]}X)}$ 
are given by  multiplication operators.
\item $\mathfrak{m}^{F}:  C(X\times [0,1])\to \Sigma^F$
is the composition of $M^{ad}$ with the quotient map $\sigma^F$;
\item $\mathfrak{m}^{{\rm NC}}: C(\mathring{X}\times \{0\}
\cup \partial X\times [0,1]) \to \Sigma^{{\rm NC}}$ is the composition of $M^{{\rm NC}}$ with the quotient map
 $\sigma^{{\rm NC}}$\,.
\end{itemize}
All these algebras and maps, together with the inverse of the excision isomorphisms, denoted ${\rm ex}$, fit into the following diagram:

\begin{equation}\label{9pieces-diagram}
\xymatrix{K_* (\mathfrak{m}^{{\rm NC}})\ar[d]_{(M^{{\rm NC}},\id)_*}&K_* (\mathfrak{m}^{F})\ar[l]_{({\rm ev}_0)_*}\ar[r]^{({\rm ev}_1)_*}
\ar[d]^{(M^{ad},\id)_*}&K_* (\mathfrak{m})\ar[d]^{(M,\id)_*}\\
K_* (\sigma^{{\rm NC}})\ar[d]_{{\rm ex}}&K_* (\sigma^{F})\ar[l]_{({\rm ev}_0)_*}\ar[r]^{({\rm ev}_1)_*}
\ar[d]^{{\rm ex}}&K_* (\sigma_{f.e.})\ar[d]^{{\rm ex}}\\
K_*(C^*_r(T^{NC}X))& K_*(C^*_r((G_{{\rm \Phi}}^\Gamma)_F))\ar[l]_{({\rm ev}_0)_*}\ar[r]^(.45){\tiny{({\rm ev}_1)_*}}& K_*(C^*_r(\mathring{X}_\Gamma\underset{\Gamma}{\times}\mathring{X}_\Gamma))
}
\end{equation}
Here, with an abuse of notation that we have already employed we denote by ${\rm ev}_0$ the homomorphism induced by restriction from $X\times [0,1]$ to  $\mathring{X}\times \{0\}
\cup \partial X\times [0,1]$ and by ${\rm ev}_1$ the homomorphism induced by restriction from $X\times [0,1]$ to  
$\mathring{X}\times \{1\}$. Notice that  we have used \eqref{commutative-relative} and \eqref{commutative-relative-bis}repeatedly and that with an abuse of notation we write $({\rm ev}_0)_*$ instead of $({\rm ev}_0,{\rm ev}_0)_*$.
By functoriality of all these homomorphisms it is easy to see that this diagram commutes.\\ 
Observe  that the composition
$({\rm ev}_1)_* \circ ({\rm ev}_0)_*^{-1}$ in the bottom row  is the adiabatic index homomorphism $\Ind^{{\rm ad}}$, whereas the 
composition ${\rm ex}\circ M_*$ in the right column is the relative index homomorphism $\Ind^{{\rm rel}}$.\\
We shall now prove that the two homomorphisms appearing in the upper row are isomorphisms and the same is true
for the homomorphism  $(M^{{\rm NC}},\id)_*$, in the vertical left column.\\
Consider first 
\begin{equation}\label{2-isos}
K_* (\mathfrak{m}^{F})\xrightarrow{({\rm ev}_1,{\rm ev}_1)_*} K_* (\mathfrak{m})\quad \text{ and  }\quad
K_* (\mathfrak{m}^{F})\xrightarrow{({\rm ev}_0,{\rm ev}_0)_*} K_* (\mathfrak{m}^{{\rm NC}}).
\end{equation}
They are induced, respectively, by
the following squares, see  \eqref{commutative-relative}:

\begin{equation*}
\xymatrix{C(X\times [0,1])\ar[r]^(.7){\mathfrak{m}^F}\ar[d]_{{\rm ev}_1}&\Sigma^F \ar[d]^{{\rm ev}_1} & & C(X\times [0,1])\ar[r]^(.7){\mathfrak{m}^F}\ar[d]_{{\rm ev}_0}&\Sigma^{F} \ar[d]^{{\rm ev}_0}\\
	C(X)\ar[r]^(.6){\mathfrak{m}}& \Sigma & & C(\mathring{X}\times \{0\}
\cup \partial X\times [0,1])\ar[r]^(.75){\mathfrak{m}^{{\rm NC}}}& \Sigma^{\rm NC} 	}
\end{equation*}
Since the vertical morphisms are obviously surjective and the kernels of the are isomorphic to cones over suitable $C^*$-algebras,
we deduce that the vertical maps induce isomorphisms in K-theory. Consequently, by the argument given at the end
of Subsection \ref{subsect:preliminaries}, we see that the morphisms appearing in \eqref{2-isos} are indeed isomorphisms.\\
Consider next $K_* (\mathfrak{m}^{{\rm NC}})\xrightarrow{(M^{{\rm NC}},\id)_*} K_* (\sigma^{{\rm NC}})$, which is induced
by the following square:
$$ 
\xymatrix{C(\mathring{X}\times \{0\}
\cup \partial X\times [0,1])\ar[r]^(0.75){\mathfrak{m}^{{\rm NC}}}\ar[d]_{M^{{\rm NC}}}&\Sigma^{{\rm NC}}\ar[d]^{\id}\\
\overline{\Psi^0_c (T^{{\rm NC}}_{[0,1]}X)}\ar[r]^{\sigma^{{\rm NC}}} &\Sigma^{{\rm NC}} 
}
$$
Our goal is to show that $(M^{{\rm NC}},\id)_*$ is an isomorphism. To this end we first observe that, since $\id$ is an 
isomorphism, it suffices to prove that $M^{{\rm NC}}$ induces a  K-theory isomorphism (see again the remark at the  end
of Subsection \ref{subsect:preliminaries}). Next we remark that we have a mapping-cone long exact sequence associated to
$M^{{\rm NC}}$, viz.
\begin{equation*}
\dots\rightarrow K_* (\overline{\Psi^0_c (T^{{\rm NC}}_{[0,1]}X)}\otimes C_0 (0,1)) \rightarrow K_* (C_{M^{{\rm NC}}})\rightarrow 
K_* (C(\mathring{X}\times \{0\}
\cup \partial X\times [0,1]))\xrightarrow{M^{{\rm NC}}_*} K_* (\overline{\Psi^0_c (T^{{\rm NC}}_{[0,1]}X)})\rightarrow\dots
\end{equation*}
Hence, it suffices to prove that $K_* (C_{M^{{\rm NC}}})$ is trivial. To show this 
we consider the commutative  diagram  with exact rows 

\begin{equation*}
\xymatrix{0\ar[r]&C_0(\partial X\times (0,1])\ar[r] \ar[d]^{M^{\partial}\times\id_{(0,1]}}& C(\mathring{X}\times \{0\}
\cup \partial X\times [0,1])\ar[r]\ar[d]^{M^{{\rm NC}}}&C(X\times\{0\}) \ar[r]\ar[d]^{M^T}&0 \\
0\ar[r]&\overline{\Psi^0_c ( \partial G^\Gamma_{{\rm \Phi}} \times (0,1]  )}\ar[r] & \overline{\Psi^0_c (T^{{\rm NC}}_{[0,1]}X)}
\ar[r] & \overline{\Psi^0_c ( {}^{{\rm \Phi}} TX  )}\ar[r] & 0
}
\end{equation*}
where the two vertical arrows 
\begin{equation*}
C_0(\partial X\times (0,1])\xrightarrow{M^{\partial}\times\id_{(0,1]}} \overline{\Psi^0_c ( \partial G^\Gamma_{{\rm \Phi}} \times (0,1]  )}\,\qquad\quad
C(X\times\{0\}) \xrightarrow{M^T} \overline{\Psi^0_c ( {}^{{\rm \Phi}} TX  )}
\end{equation*}
are given by obvious multiplication operators.
The diagram induces the following short exact sequence of
mapping-cone $C^*$-algebras :
\begin{equation*}
0\rightarrow C_{M^{\partial}\times\id_{(0,1]}} \rightarrow C_{M^{{\rm NC}}} \rightarrow C_{M^{T} }\rightarrow 0
\end{equation*}
Now we observe that $C_{M^{\partial}\times\id_{(0,1]}}$ is isomorphic to $C_{M^{\partial}}\otimes C_0 (0,1]$
which is K-contractible; next, using the fact that  $\overline{\Psi^0_c ( {}^{{\rm \Phi}} TX  )}$ is isomorphic to the 
algebra of continuous functions over the ball bundle of  ${}^{{\rm \Phi}} T^*X$, we see that $M^{T}$  is an homotopy
equivalence of $C^*$-algebras and thus its mapping cone is K-contractible. Summarizing,
$K_* (C_{M^{\partial}\times\id_{(0,1]}})$ and $K_* (C_{M^T})$ are both equal to the trivial group and therefore so
is $K_* (C_{M^{{\rm NC}}})$, which is what we needed to prove in order to establish that 
$(M^{{\rm NC}},\id)_*\colon K_* (\mathfrak{m}^{{\rm NC}})\to K_* (\sigma^{{\rm NC}})$ is an isomorphism.

Notice in particular that we can connect $K_*(C^*_r(T^{NC}X))$ to $K_* (\mathfrak{m})$
by means of an isomorphism, denoted here $\Theta$, obtained by travelling
up on the 
 the left column and then right on the top row of diagram \eqref{9pieces-diagram}.
 We end the proof of Theorem  \ref{equality-index-classes} by showing that if $P$ is a fully elliptic
 element in $\Psi^0_c (G_{{\rm \Phi}}^\Gamma)$ then $\Theta (\sigma_{{\rm nc}} (P))=[P]_{rel}$.
 To see this we define an element in $K_* (\mathfrak{m}^{F})$ which is mapped to $[P]_{rel}$
 through ${\rm ev}_1$ and it is mapped to $\sigma_{{\rm nc}} (P)$ by the homomorphism that goes left
 and then down on the external part of the diagram starting with $K_* (\mathfrak{m}^{F})$.
 This element is nothing but the class of $(\sigma (P)\otimes \id_{[0,1]}, N(P))\in K_* (\mathfrak{m}^{F})$.
 The image of this element through the homomorphism induced by ${\rm ev}_0$ is given
 by simple restriction, whereas its image through $M^{{\rm NC}}_*$ amounts to a change of coefficients
 in the modules associated to $E$ and $F$ (from continuous functions to pseudodifferential operators).
 Finally, that the image of this element under excision is given by $\sigma_{{\rm nc}} (P)$ is proved by
 proceeding as in Subsection \ref{relative}, once we identify $KK(\bbC, C^*_r(T^{NC}X))$ with
 $K_*(C^*_r(T^{NC}X))$.

\begin{remark}\label{remark:previous}
	The proof of Theorem \ref{equality-index-classes} follows closely the proof of the analogous Theorem in \cite{MP} for the analytical index of elliptic operators on Lie groupoids, in the case where the objects of the Lie groupoid form a closed manifold. 
	Moreover our result gives a generalization of \cite[Theorem 4.1]{CLM} from the case of a manifold
	boundary to the case of a manifold with fibered corners.
	Yet in another direction our proof generalizes  Propostion 3.8 in \cite{DSk}; notice  that our arguments here build directly
	on those given there. More precisely our proof generalizes \cite[Propostion 3.8]{DSk} to the case in which (in their notations) $\operatorname{Ind}_{G_F}$, where $G_F$ is here the restriction of the Lie groupoid $Blup^+_{r,s}$ to the boundary, is not invertible and it identifies, by means of the isomorphism $\Theta$ in the proof of Theorem \ref{equality-index-classes}, the relative index $\operatorname{ind}_{rel}\colon K_*(\mu)\to K_*(C^*(G_W))$ and the morphism $\operatorname{ind}_G^W\colon K_*(C^*(\mathcal{A}_WG))\to K_*(C^*(G_W))$, defined in \cite[Section 3.3.3]{DSk}.
	\\
\end{remark}
\section{Fully elliptic  spin Dirac operators}\label{sect:dirac-fully}

\subsection{Spin Dirac operators}
Let $ {}^{\textrm{S}} X$ be a stratified pseudomanifold of dimension $n$, $X$ its resolution, $^{{\rm \Phi}}T X$ the associated ${\rm \Phi}$-tangent
bundle  and $g_{{\rm \Phi}}$ a rigid ${\rm \Phi}$-metric. We shall say that $ {}^{\textrm{S}} X$ is spin if its regular part, 
${\rm reg} ({}^{\textrm{S}} X)=\mathring{X}$, is a spin manifold, meaning that $T\mathring{X}$ is a spin bundle.  We fix a spin structure associated to $g_{{\rm \Phi}}$ and get a (complex)
spinor bundle $\slashed{S}$ on $\mathring{X}$.
We can equivalently  assume that 
 the bundle $^{{\rm \Phi}}T X$ is spin; indeed $\mathring{X}$ and $X$ are homotopically equivalent.
 A spin structure for  $(T^{{\rm \Phi}} X, g_{{\rm \Phi}})$ gives, by restriction, a spin structure for $({\rm reg} ({}^{\textrm{S}} X), g_{{\rm \Phi}})$;
 conversely, we can always extend a spin structure $({\rm reg} ({}^{\textrm{S}} X), g_{{\rm \Phi}})$ to a spin structure on $(T^{{\rm \Phi}} X, g_{{\rm \Phi}})$.
Summarzing:  we do spin geometry on $ {}^{\textrm{S}} X$ by passing to the algebroid $T^{{\rm \Phi}} X$, see below for the details.\\

 This brings us to the following general treatment.
  
 \bigskip
We define generalized Dirac operators on a Lie groupoid $G\rightrightarrows X$ with Lie algebroid $\mathfrak{A}(G)$
as follows:  let $g$ be a metric on $\mathfrak{A}(G)$, by pull-back it defines a $G$-invariant metric on $\ker{ds}$ along the $s$-fibers of $G$. Let $\nabla$ be the fiber-wise Levi-Civita connection associated to this metric.
\begin{definition}\label{diracgr}
	Let $\mathrm{Cliff}\left(\mathfrak{A}(G)\right)$ be the Clifford algebra bundle over $X$ associated to the metric $g$.
	Let $S$ be a bundle of Clifford modules over  $\mathrm{Cliff}\left(\mathfrak{A}(G)\right)$ and let $c(X)$ denotes the Clifford multiplication by $X\in\mathrm{Cliff}\left(\mathfrak{A}(G)\right)$.
	Assume that $S$ is equipped with a metric $g_S$ and a compatible connection $\nabla^S$ such that:
	\begin{itemize}
		\item Clifford multiplication is skew-symmetric, that is
		\[
		\langle c(X)s_1,s_2\rangle+\langle s_1,c(X)s_2\rangle=0
		\]
		for all $X\in C^\infty\left(X,\mathfrak{A}(G)\right)$ and $s_1,s_2\in C^\infty(X,S)$;
		\item $\nabla^S$ is compatible with the Levi-Civita connection $\nabla$, namely
		\[
		\nabla^S_X(c(Y)s)=c(\nabla_XY)s+c(Y)\nabla^S_X(s)
		\]
		for all $X,Y\in C^\infty\left(X,\mathfrak{A}(G)\right)$ and $s\in  C^\infty(X,S)$.
	\end{itemize}
	The Dirac operator associated to these data is defined as
	\[
	\slashed{D}_S\colon s\mapsto \sum_{\alpha}c(e_\alpha)\nabla^S_\alpha(s)
	\]
	for $s\in C^\infty(X,S)$ and $\{e_\alpha\}_{\alpha\in A}$ a local orthonormal frame.
\end{definition}
Particular cases of this construction are given by
generalized Dirac operators on manifolds with a Lie structure at infinity, which are treated in detail in \cite{ALN} 
\cite{ALN-geometry} \cite{LN-approach} and in the recent preprint \cite{Bohlen-Schrohe}. In particular, one can define in this generality the notion of spin Lie manifold and that of associated  spin
Dirac operator. We  recall briefly the very natural definitions in the particular case treated here.

Let ${}^{\textrm{S}} X$ be  a stratified space  of dimension $n$, with resolution $X$. As already anticipated
we say that ${}^{\textrm{S}} X$ is 
spin if the vector bundle ${}^{{\rm \Phi}} TX$ is  spin \footnote{this is a topological condition and  could be equivalently 
imposed on $TX$, given that ${}^{{\rm \Phi}} TX$ and $TX$ are isomorphic (albeit in a non-natural way).}. If $g_{{\rm \Phi}}$ is a fibered corner metric on ${}^{{\rm \Phi}} TX$
then a ${\rm \Phi}$-spin structure associated to $g_{{\rm \Phi}}$ is a ${\rm Spin(n)}$-principal bundle $P_{{\rm Spin}} ({}^{{\rm \Phi}} TX)$
together with    a 2-fold covering map to the ${\rm SO(n)}$-principal bundle of orthonormal frames of  $({}^{{\rm \Phi}} TX,
g_{{\rm \Phi}})$. We have a natural notion of spinor bundle $\slashed{S}$, obtained from $P_{{\rm Spin}} ({}^{{\rm \Phi}} TX)$ through the complex spinor representation
of ${\rm Spin(n)}$. The spinor bundle is an  example of a bundle of  Clifford modules $W\to X$
for the $\Phi$-cotangent bundle ${}^{{\rm \Phi}} T^*X$ (briefly, a $\Phi$-Clifford module); 
this notion can of course be given in general, without the spin-assumtion.   On a $\Phi$-Clifford
module $W\to X$ we  can talk about a 
Clifford $\Phi$-connection  $\nabla^W\in\Diff^1_{{\rm \Phi}} (X;W,W\otimes {}^{{\rm \Phi}} T^*X)$ where the adjective Clifford refers to the
usual compatibility with the Levi-Civita connection associated to $g_{{\rm \Phi}}$ \footnote{since  $C^\infty(X, {}^{{\rm \Phi}} TX)$ has the structure of a Lie algebra, the usual definition of Levi-Civita connection can be given}.
To the Clifford module structure and the 
given connection there is associated in a natural way a generalized Dirac operator $D\in \Diff^1_{{\rm \Phi}}(X,W)$
\begin{equation}\label{Dirac}
D:= c\circ \nabla^W
\end{equation}
where we have denote by $c$ the map $C^\infty (X,W\otimes  {}^{{\rm \Phi}} T^*X)\to C^\infty (X,W)$ induced
by the Clifford action on $W$. We can write as usual  the action of $D$ as
\begin{equation}\label{Dirac-local}
D (s) =\sum_{\alpha}c(e^\alpha)\nabla^{W}_\alpha(s)
	\end{equation}
for $s\in C^\infty(X,W)$ and $\{e_\alpha\}$ a local orthonormal frame of ${}^{{\rm \Phi}} TX$
with dual basis $\{e^\alpha\}$.

If $W$ is the spinor bundle and the Clifford connection is the one induced by the Levi-Civita
connection on ${}^{{\rm \Phi}} TX$ then we obtain the spin Dirac operator $\slashed{D} \in \Diff^1_{{\rm \Phi}}(X,\slashed{S})$.\\

Through the bijection between $\Diff^1_{{\rm \Phi}}(X;\slashed{S})$ and $\Diff^1 (G_{{\rm \Phi}}; r^*\slashed{S})$ we obtain also
the Dirac operator on the groupoid $G_{{\rm \Phi}}$, denoted $\slashed{\mathcal{D}}\in \Diff^1 (G_{{\rm \Phi}}; r^*\slashed{S})$. In fact, as explained in detail in \cite[Proposition 6.1]{LN-approach}
we can define directly the Dirac operator on $G_{{\rm \Phi}}$  by lifting the Clifford module structure of $\slashed{S}$ 
to $r^* \slashed{S}$ and by using the Levi-Civita connection 
on $G_{{\rm \Phi}}$ associated to the pull-back of $g_{{\rm \Phi}}$ to  a $G_{{\rm \Phi}}$-invariant metric on $\ker{ds}$.
Through the usual formula \eqref{Dirac} we define in this way an element $\slashed{\mathcal{D}}\in \Diff^1 (G_{{\rm \Phi}}; \slashed{S})$
which corresponds to  $\slashed{D} \in \Diff^1_{{\rm \Phi}}(X,\slashed{S})$
through the bijection between $\Diff^1_{{\rm \Phi}}(X;\slashed{S})$ and $\Diff^1 (G_{{\rm \Phi}}; r^*\slashed{S})$.

\begin{remark}\label{dirac-on-groupoids}
The above procedure can be applied in general, to any Lie groupoid with a spin Lie algebroid; in particular we also 
get in this way a Dirac operator on the adiabatic deformation $(G_{{\rm \Phi}})^{[0,1]}_{ad}$. 
\end{remark}

The operator $\slashed{D} \in  \Diff^1 (G_{{\rm \Phi}};r^* \slashed{S})$ defines an unbounded regular operator 
$\overline{\slashed{D}}$ on the Hilbert
$C^*_r(G_{{\rm \Phi}})$-Hilbert module $\mathcal{E}$ given, by definition, by the closure of $C^\infty_c (G_{{\rm \Phi}},r^* \slashed{S}\otimes\Omega^{1/2})$
in the norm associated to the $C^*_r(G_{{\rm \Phi}})$-valued scalar product $\langle\xi,\eta\rangle (\gamma)=\int_{G_{s(\gamma)}} \langle\xi(\gamma),\eta(\gamma^{-1}\gamma^\prime)\rangle_{\slashed{S}}$. In the sequel, adopting a widely used abuse of notation,
we shall not distinguish $\overline{\slashed{D}}$ from $\slashed{D}$.
As usual, $\slashed{D}$ admits a bounded
inverse 
if there exists a bounded
$C^*_r (G_{{\rm \Phi}})$-Hilbert module operator $Q$ such that $\slashed{D} \circ Q=\Id= Q\circ \slashed{D}$.

%
%
%
%

\subsection{Fully elliptic spin Dirac operators}
We want to describe, first of all,  the normal families of $\slashed{D}$ explicitly.
Let $H_i\xrightarrow{\phi_i} S_i$ be one of the fibered boundary hypersurfaces of $X$. The restriction of $G_{{\rm \Phi}}$ to $G_i:=H_i\setminus\bigcup_{i<j}H_j $ is equal to $(H_i\underset{S_i}{\times} {}^{\phi}TS_i\underset{S_i}{\times} H_i)_{|G_i}\times \RR$ 
and the normal family $\sigma_{\partial_i}(\slashed{D})$ relative to $\phi_i$ is just the Dirac operator of this groupoid,
 endowed with the Clifford structure it inherits from the Clifford structure of $G_{{\rm \Phi}}$. 

Bearing in mind the $G_{{\rm \Phi}}$-invariance this means that the following formula holds
\begin{equation}\label{normal-dirac}
\sigma_{\partial_i}(\slashed{D})(h) =(D_{\phi_i})_{s}+ \slashed{D}_{\RR^n,g_{s}}+\slashed{D}_{\RR}\,, \quad h\in H, \quad \phi_i (h)=s
\end{equation}
where $D_{\phi_i}$ is a vertical family of generalized Dirac operator on $\phi_i:H_i\to S_i$
and $(D_{\phi_i})_s$ is the operator of this family on the fiber $Z_s:=\phi_i^{-1} (s)$ ; we are identifying ${}^{{\rm \Phi}} T_{s}S_i$ with $\RR^n$ and denoting by $g_{s}$ the metric  $g_{S_i}$ at $s\equiv \phi(h)$; $\slashed{D}_{\RR^n,g_s}$ is an euclidian Dirac operator;
$\slashed{D}_{\RR}$ denotes the canonical Dirac operator on the flat real line.

%
%

\begin{lemma}\label{spin-strata}
If the base of the fibration $H_i\xrightarrow{\phi_i} S_i$ is assumed to be spin then 
the vertical tangent bundle ${}^{{\rm \Phi}} T(H_i / S_i)$ is spin and the vertical family 
$D_{\phi_i}$ is a family of spin Dirac operators: $D_{{\rm \Phi}}\equiv \slashed{D}_{{\rm \Phi}}$.
\end{lemma}

\begin{proof}
As we are assuming that $M$ is spin, we also obtain that $H_i$ is spin.
Since, by assumption, $S_i$ is spin, then \[Z_i\rightarrow H_i\xrightarrow{\phi_i} S_i\] is a fibration of spin manifolds.
See \cite{lawson89:_spin}
\end{proof}

%
%
%
%
%
%
%
%
%
%
%

\begin{proposition}\label{fullellipticity}
If for each $i\in\{1,\dots,k\}$ $D_{\phi_i}$ is a family of invertible Dirac operators, then $\slashed{D}$ is fully elliptic.
\end{proposition}
\begin{proof}
	Recall that $\slashed{D}$  is fully elliptic if $\sigma_{\partial_i}(\slashed{D})$ is a family of invertible operators for all $i\in I$.
We shall prove this assertion by  showing that $\sigma_{\partial_i}(\slashed{D})^2$ is a family of strictly positive operators for all $i\in I$.
Let $h\in H_i$ and let $s:=\phi_i (h)\in S_i$.
First recall that $\sigma_{\partial_i}(\slashed{D})(h)$ is an operator on $Z_{s}\times \RR^n\times \RR$ and that the metric on this manifold is a product-type metric $g_{Z_{s}}\oplus g_{\phi(h)}\oplus d\sigma^2$. Because of that
$(D_{\phi_i})_{\phi_i(h)}$ and $ \slashed{D}_{\RR^n,g_{\phi(h)}}+\slashed{D}_{\RR}$ anticommute.
Thus  we have the following identity
\[
\sigma_{\partial_i}(\slashed{D})(h)^2= (D_{\phi_i})_{s}^2+ (\slashed{D}_{\RR^n,g_{s}}+\slashed{D}_{\RR})^2
\]
which, by the strict positivity of  $(D_{\phi_i})_{s}^2$, implies the strict positivity of $\sigma_{\partial_i}(\slashed{D})(h)^2$.
\end{proof}

\begin{proposition}
If	the strata are spin and  the metrics on the links have positive scalar curvature, then $\slashed{D}$  is fully elliptic.
\end{proposition}
\begin{proof}
	By Proposition \ref{fullellipticity} we just have to prove that $D_{\phi_i}$ is invertible for all $i\in \{1,\dots,k\}$.
	Let us fix such an  $i\in\{1,\dots,k\} $.
	From the assumption and Lemma \ref{spin-strata} we know that the fibers of $H_i\xrightarrow{\phi_i} S_i$
	 are spin and that $D_{\phi_i}\equiv \slashed{D}_{\phi_i}$, a family of spin Dirac operators. Thus we can apply the 
	  Schr\"odinger-Lichnerowicz-Weizenbock formula and obtain that $(\slashed{D}_{\phi_i})_{s}$ is invertible for each $s\in S_i$.
	 The proposition is proved.
\end{proof}

\section{Primary and secondary K-theory classes associated to a spin Dirac operator}\label{sect:dirac-k}
\subsection{The fundamental class associated to a fully elliptic Dirac operator}
In this subsection we want to define the K-homology class associated to a fully elliptic spin Dirac operator $\slashed{D}$; 
we do this by producing a class in $K_* (C^*_r(T^{NC} X))$. To this end we shall consider the Dirac operator $\slashed{D}_{ad}$
associated to the adiabatic deformation groupoid $(G_{{\rm \Phi}})^{[0,1]}_{ad}$, see Remark \ref{dirac-on-groupoids}.
In fact, we shall work directly in the $\Gamma$-equivariant case and denote the resulting $\Gamma$-equivariant Dirac operator
by $\slashed{D}^\Gamma$. As usual we employ the simple notation $G\rightrightarrows X$ for the groupoid
$G^\Gamma_{{\rm \Phi}}$. The adiabatic deformation is therefore denoted, as usual, as $G^{[0,1]}_{ad}$.

We shall take the bounded transform of our operator, $\psi (\slashed{D}^\Gamma)$, with $\psi(x)=x/\sqrt{1+x^2}$.
This is not an element in $\Psi^0_c (G)$ but rather in an extended version of it, according to Vassout
\cite{Vassout-jfa}: one adds to $\Psi^0_c (G)$ an algebra of "smoothing operators" which is holomorphically 
closed and obtains in this way an algebra which is contained densely in $\overline{\Psi^0_c (G)}$. We refer to
\cite[Section 4]{Vassout-jfa} for the details.
According to one of the main results in \cite{Vassout-jfa}, this algebra contains the bounded transform $\psi (\slashed{D}\Gamma)$. In particular,
 we see that  $\psi (\slashed{D}^\Gamma)\in 
\overline{\Psi^0_c (G)}$. Similarly, we have that  $\psi (\slashed{D}^\Gamma_{ad})\in 
\overline{\Psi^0_c (G^{[0,1]}_{ad})}$. 

In order to define a class in $K_* (C^*_r (T^{NC} X))$ we would like to proceed as in 
Subsection \ref{subsect:nctb-ncs}. Thus we would like to produce a class in the K-theory of $C^*_r (G_{ad}^{[0,1],F})$;
notice however that the natural candidate, namely   $\psi (\slashed{D}^\Gamma_{ad})$, does not produce a KK-cycle
which is {\em degenerate} on  $\partial X\times \{1\}$, see   Subsection \ref{subsect:k-classes-general}; indeed the restriction
of $\psi (\slashed{D}^\Gamma_{ad})^2 -\Id$ to $\partial X\times \{1\}$, i.e. $\psi (N (\slashed{D}^\Gamma) )^2 - \Id$, is not equal to 0. 
To fix this problem we shall now homotope $\psi (\slashed{D}^\Gamma)$ to an operator that has this property.
Recall the fully elliptic symbol $\sigma_{f.e.}$ and the short exact sequence \eqref{f.e.extention}
\begin{equation*}
\xymatrix{
	0\ar[r]& C^*_r(\mathring{X}_\Gamma\times_\Gamma\mathring{X}_\Gamma)\ar[r] & \overline{\Psi^0_c(G)}\ar[r]^(.6){\sigma_{f.e.}}&  \Sigma\ar[r]& 0
}.
\end{equation*}
where we recall that  $\Sigma:=\overline{\Psi^0_c(\partial G)}\underset{\partial X}{\times} C(\mathfrak{S}^* G)$.
Now observe that $\sigma_{f.e.}(\psi (\slashed{D}^\Gamma))=\psi (\sigma_{f.e.}(\slashed{D}^\Gamma))$. Moreover consider the family of functions $\psi_s(t)=\psi(t/(1-s))$ for $s\in[0,1]$ and observe that $\psi_1(t)=\rm{sign}(t)$.
Since $\slashed{D}^\Gamma$ is fully elliptic, $\psi (\sigma_{f.e.}(\slashed{D}^\Gamma))$ is invertible and 
$\psi_s (\sigma_{f.e.}(\slashed{D}^\Gamma))$ is a homotopy through invertible elements  with the property that $\psi_1 (\sigma_{f.e.}(\slashed{D}^\Gamma))$ is well defined and equal to the sign of the full symbol of $\slashed{D}^\Gamma$.
By the surjecivity of ${\sigma_{f.e.}}$ we can lift $\psi_s (\sigma_{f.e.}(\slashed{D}^\Gamma))$ to a path $P_s$ of elliptic operators in $\overline{\Psi^0_c(G)}$ such that $P_0= \psi (\slashed{D}^\Gamma)$ and $P_1$ is a fully elliptic operator with normal operator $N(P_1)= \rm{sign}(N(\slashed{D}^\Gamma))$. Observe in particular that the restriction to $\pa X$ of $P_1^2 - \Id$ is equal to 0.

Now, up to the reparametrization of the adiabatic deformation from $[0,1]$ to $[0,1/2]$
 and of the homotopy $P_s$ from $[0,1]$ to $[1/2,1]$, the concatenation of these two elements gives an operator on the adiabatic deformation such that the restriction to $\partial X\times \{1\}$ defines a degenarate cycle. As already explained in Subsection \ref{subsect:nctb-ncs}, this provides us with a class $[\slashed{D}_{ad}^F]$ in the K-theory of $C^*_r (G_{ad}^{[0,1],F})$.
  
 By restricting this operator to $\mathring{X}\times {0}\cup \partial X\times[0,1)$ we obtain the desired class $[\sigma_{nc}(\slashed{D}^\Gamma)]$ in  $K_* (C^*_r (T^{NC} X))$, $*=\dim X$.
In the even dimensional case we are of course using the well-known fact that the spinor bundle is $\bbZ_2$-graded and 
the Dirac operator is odd with respect
to this grading. Since $\psi(x)=x/\sqrt{1+x^2}$ is an odd function, we also have that $\psi(\slashed{D}^\Gamma)$ is odd.
Hence, it is easy to see that the Kasparov definition of {\em even} bimodules applies directly, which is why 
$[\sigma_{nc}(\slashed{D}^\Gamma)]\in K_0 (C^*_r (T^{NC} X))$ if $X$ is even dimensional.

 \begin{remark}
Following closely  \cite{DLR} one can check that the Poincar\'e dual of $[\sigma_{nc}(\slashed{D}^\Gamma)]\in K_* (C^*_r (T^{NC} X))$ in $K^\Gamma_* ({}^{\textrm{S}}X_\Gamma)$ is equal to the class defined,  as in Subsection  \ref{subsect:fclass}, by $P_1$.
\footnote{here we
are implicitly extending the reasoning explained in  Subsection \ref{subsect:fclass} to the
closure of $\Psi^0_c (G)$; this extension does not pose any problem.}

 \end{remark}


\subsection{The index class of a fully elliptic Dirac operator as an obstruction}

Recall Definition \ref{index-homo-nc}; following the notations of the previous subsection,
the index class $\operatorname{Ind}^{ad}(\slashed{D})$ is given by the evaluation at 1 of the class
$[\slashed{D}_{ad}^F]$. This class in $KK^*(\bbC,C^*_r(\mathring{X}_\Gamma\times_\Gamma \mathring{X}_\Gamma))$ is represented by the operator $P_1$, also defined in the previous subsection.

\begin{proposition}
If the $\Gamma$-equivariant metric $g_{{\rm \Phi}}$ has positive scalar curvature evereywhere, then $\operatorname{Ind}^{ad}(\slashed{D})$ vanishes.
\end{proposition}
\begin{proof}
By the Schr\"odinger-Lichnerowicz formula, we know that $\slashed{D}$ is invertible.
Following the notations of the previous section, we can define $P_s$ as the path $\psi_s(\slashed{D})$, for $s\in[0,1]$;
this is well defined up to $s=1$ since
$\slashed{D}$ is invertible. Moreover we have that $P_1={\rm sign}(\slashed{D})$, which induces a degenerated cycle over $C^*_r(\mathring{X}_\Gamma\times_\Gamma \mathring{X}_\Gamma)$. This completes the proof.
\end{proof}
     We see the above result as an obstruction result: if the vertical metrics on the fibrations, which are assumed to be
     of positive scalar curvature, can be extended to a fibered corners metric of positive scalar curvature, then 
      $\operatorname{Ind}^{ad}(\slashed{D})$ vanishes. 
      
      \subsection{The rho-class of an invertible Dirac operator and its properties.}$\;$\\
We assume that the $\Gamma$-equivariant fibered corners metric $g_{{\rm \Phi}}$ has positive scalar curvature everywhere; as already
remarked we then have that the operator $\slashed{D}$ is invertible.

Consider the path $P_s:=\psi_s(\slashed{D})$ defined in the previous subsection and observe that, after the usual reparametrization, the concatenation of 
$\psi(\slashed{D}_{ad})$  and $P_s$ defines a $(\C, C^*_r (G^{[0,1]}_{ad}))$-bimodule that, when we restrict it to the adiabatic deformation parameter $t=1$, is degenerate on the whole manifold $X$.

\begin{definition}
The rho class associated to the metric $g_{{\rm \Phi}}$ is the class $\rho(g_{{\rm \Phi}})\in KK(\C, C^*_r (G^{[0,1)}_{ad}))$
defined by  concatenation of 
$\psi(\slashed{D}_{ad})$  and $P_s$.
\end{definition}

We shall now study the 
stability properties of this class. Let 
$\mathcal{R}_{{\rm \Phi}}^+ (X)$ be the set of ${\rm \Phi}$-metrics on $X$ that are of positive scalar curvature.
	Recall that  two ${\rm \Phi}$-metrics   $g_0$ and $g_1$ are concordant if there exists a
	positive scalar curvature ${\rm \Phi}$-metric  $g$  on $X\times[0,1]$, of product type near the boundary, 
	such that the restriction of $g$ to $X\times\{i\}$ is equal to $g_i$ (for $i=0,1$).
	We denote by $\widetilde{\pi}_0 (\mathcal{R}_{{\rm \Phi}}^+ (X))$ the set of concordance classes of 
	psc ${\rm \Phi}$-metrics. We denote by $\pi_0 (\mathcal{R}_{{\rm \Phi}}^+ (X))$ the connected components
	of $\mathcal{R}_{{\rm \Phi}}^+ (X)$.
	
	\begin{proposition}
		The application $\rho\colon\widetilde{\pi}_0 (\mathcal{R}_{{\rm \Phi}}^+ (X)) \to K_* ( C^*_r ( G_{ad}^{[0,1)}))$
		given by
		\[
		[g]\mapsto \rho(g)
		\]
		 is well defined.\\ Similarly, the rho class gives a well defined 
		 application $\rho\colon \pi_0 (\mathcal{R}_{{\rm \Phi}}^+ (X)) \to K_*( C^*_r ( G_{ad}^{[0,1)}) )$.
	\end{proposition}

	\begin{proof}
	This is a direct application of the delocalized APS index theorem stated in Theorem \ref{deloc-aps}.
			\end{proof}



\section{Singular foliations}\label{sing-fol}

\subsection{Blow-up constructions for Lie groupoids}

Following the recent work of Debord and Skandalis, see \cite{DSk}, we shall now reobtain the groupoid associated 
to a manifold with fibered boundary and to the ${\rm \Phi}$-tangent bundle, as a blow-up construction in the
groupoid context.

First, we define the deformation to the normal cone.
Let $Y$ be a smooth compact manifold and let $X$ be a submanifold of $Y$.
The deformation to the normal cone $DNC(Y,X)$ is obtained by gluing $N_X^Y \times \{0\}$ with $Y\times \bbR^*$,
where $N_X^Y$ denoted the normal bundle of $X$ in $Y$.
The smooth structure of $DNC(Y,X)$ is described by use of any exponential map $\theta : U^\prime \to U$ which is a diffeomorphism from an open neighborhood  $U^\prime$ of the 0-section in $N_X^Y$ to an open neighborhood $U$ of $X$. 
See \cite[Section 4.1]{} for the details. 
\\
The group $\bbR^*$ acts on $DNC(Y,X)$ by 
$$\lambda\cdot ((x,\xi),0)=((x,\lambda^{-1}\xi),0) \,,\quad 
\lambda\cdot (y,t)=(y,\lambda t) \;\;\text{with}\;\;(x,\xi)\in N_X^Y\,,(y,t)\in Y\times \bbR^*\,.
$$
Given a commutative diagram of smooth maps 
\[
\xymatrix{X\ar@{^{(}->}[r]\ar[d]^f& Y\ar[d]^f\\
	X^\prime\ar@{^{(}->}[r]& Y^\prime}
\]
where the horizontal arrows are inclusions of submanifolds, we naturally obtain a smooth map
$DNC(f)\colon DNC(Y,X) \to DNC(Y',X')$. This map is defined by $DNC(f)(y,\lambda) = (f (y),\lambda)$
for $y\in Y$ and $\lambda\in\bbR$ and $DNC(f)(x,\xi,0) = (f(x),f_N(\xi),0)$ for $(x,\xi) \in N_X^Y$, where $f_N \colon N_X^Y \to N_{X'}^{Y'}$ is the linear map induced by the differential $df$. This map is  equivariant with respect to the action of $\bbR^*$.

The action of $\bbR^*$ is free and locally proper on $DNC(Y,X)\setminus X\times\bbR$ and we define $Blup(Y,X)$ as the quotient space of this action.

If $H\rightrightarrows H^{(0)}$ is a closed subgroupoid of a Lie groupoid $G\rightrightarrows G^{(0)}$, then 
$DNC(G,H)$ is a Lie groupoid over $DNC(G^{(0)},H^{(0)})$ where the source and target maps are simply given by
$DNC (s)$ and $DNC(r)$ as defined above.

\begin{example}\label{DNC}
	Let $H=G^{(0)}\rightrightarrows G^{(0)}$ be the trivial groupoid, then $N_H^G$ is just $\mathfrak{A}G$, the Lie algebroid of $G$, and  
	$DNC^+(G,H)$ is the adiabatic deformation groupoid $G^{[0,+\infty)}_{ad}$. 
\end{example}

On the other hand,  $Blup(G,H)$ is not a Lie groupoid over $Blup(G^{(0)},H^{(0)})$, since the $Blup$ construction is not functorial.
But $Blup(G,H)$ contains the dense open subset $$Blup_{r,s}(G,H):=\left(DNC(G,H)\setminus (H\times\bbR\cup DNC(s)^{-1}(H^{(0)}\times\bbR)\cup DNC(r)^{-1}(H^{(0)}\times\bbR))\right)/ \bbR^*$$ that is a Lie groupoid over $Blup(G^{(0)},H^{(0)})$.

We shall be also interested in a variant of this construction: we consider $DNC(G,H)\rightrightarrows  DNC(G^{(0)},H^{(0)})$ 
and define $DNC^+ (G,H)$ as its restriction to $(N^{G^{(0)}}_{H^{(0)}})^+ \times \{0\} \cup G^{(0)}\times \bbR^*_+$
with $(N^{G^{(0)}}_{H^{(0)}})^+$ denoting the positive normal bundle \footnote{where, for $h\in H^{(0)}$,
	$(N^{G^{(0)}}_{H^{(0)}})^+_h$ is defined by $(\bbR^n)_+:= \bbR^n_+$ once we fix a linear isomorphism 
	$(N^{G^{(0)}}_{H^{(0)}})_h$ with  $\bbR^n$}.
We also define $Blup^+(G,H)$ as the quotient of $DNC^+ (G,H)\setminus H\times \bbR_+$ by the action 
of $\bbR^*_+$. We obtain in this way the groupoid $$Blup^+_{r,s}(G,H) \rightrightarrows 
Blup^+ (G^{(0)},H^{(0)}).$$

\subsection{Revisiting the groupoid associated to a manifold with  fibered boundary}$\;$\\
Let now $M$ be a smooth manifold with fibred boundary  $\partial M\xrightarrow{\phi} B$.  For example, $M$ is the
resolution of a depth-1 stratified pseudomanifold ${}^{\textrm{S}} M$ with singular stratum $B$.\\
For simplicity we assume that
$\partial M$ is connected.\\ Our first goal is to recover in this groupoid context the $b$-stretched product appearing in
the definition of the $b$-calculus, see \cite{Melrose}.
Consider  $G:=M\times M\rightrightarrows M$ and $H:=\partial M\times \partial M \rightrightarrows \partial M$.
Our present goal is to explicitly describe the groupoid
$$Blup^+_{r,s}(G,H) \rightrightarrows Blup^+ (G^{(0)},H^{(0)}).$$
We  have, by definition,
$$DNC (G,H)= \partial M \times \partial M\times \bm{\mathrm{R}}^2 \times \{0\}\cup M\times M\times \bbR^* \,,$$ 
and
$$DNC  (G^{(0)},H^{(0)}) = \partial M \times \bm{\mathrm{R}}\times \{0\} \cup M\times \bbR^*$$ 
where we use the bold face to distinguish the fibers of the normal bundle.
The groupoid structure of 
$DNC (G,H)$ 
is explicitly given by the following:
\begin{itemize}
	\item $s ((x,y),(n_1,n_2),0)= (y,n_2,0)$, $r ((x,y),(n_1,n_2),0)=(x,n_1,0)$; \\$s((x,y),t)=(y,t)$,  $r((x,y),t)=(x,t)$;
	\item $m (((x,y),(n_1,n_2),0), ((y,z) (n_2,n_3),0))= ((x,z),(n_1,n_3),0)$\\
	$m(((x,y),t),((y,z),t))= ((x,z),t)$.
\end{itemize}
This gives also the groupoid structure for $DNC^+ (G,H)$ where, by definition,
$DNC^+ (G,H)$ is the restriction of $DNC (G,H)$ to $DNC^+  (G^{(0)},H^{(0)})$,
which is  $\partial M \times \bm{\mathrm{R}}^+\times \{0\} \cup M\times \bbR^*_+$;
thus 
$$DNC^+ (G,H)= \partial M \times \partial M\times \bm{\mathrm{R}}^2_+ \times \{0\}\cup M\times M\times \bbR^*_+ \,,$$ 
Now, 
$$Blup^+(G,H):= \left( DNC^+ (G,H)\setminus \left(\partial M \times \partial M\times \{\underline{0}\}\times\{0\} 
\cup  \partial M \times \partial M\times\bbR^*_+ \right)\right)/\bbR^*_+
$$
and hence we see  that
$$Blup^+(G,H)= \partial M\times \partial M \times \bm{\mathrm{S}}^1_{+}\cup (M\times M\setminus \partial M\times \partial M)$$
which is the b-stretched product of Melrose, denoted $[M\times M;\pa M\times \pa M]$ \footnote{Following the definitions one can prove that this identification
	gives a  diffeomorphism between $Blup^+(G,H)$ and  $[M\times M;\pa M\times \pa M]$}.

Finally, observing that $\partial M \times \partial M\times \{\underline{0}\}\times\{0\} 
\cup  \partial M \times \partial M\times\bbR^*_+$ is nothing but $H\times \bbR_+$, we obtain
$Blup^+_{r,s}(G,H)$ by removing from $Blup^+(G,H)$ the image through the quotient map of the subspaces
$$\partial M \times \partial M\times (\{\underline{0}\}\times\bm{\mathrm{R}}_+)\times\{0\}
\cup H\times M\times \bbR_+^* \subset DNC^+(G,H)\setminus H\times \bbR_+ $$
$$
\partial M \times \partial M\times (\bm{\mathrm{R}}_+\times \{\underline{0}\})\times\{0\} \cup
M\times H\times \bbR_+^*\subset DNC^+(G,H)\setminus H\times \bbR_+ $$ (these are the points where 
$$DNC(s): DNC^+(G,H)\setminus H\times \bbR_+ \to DNC^+ (G^{(0)},H^{(0)})$$ and 
$$DNC(r): DNC^+(G,H)\setminus H\times \bbR_+ \to DNC^+ (G^{(0)},H^{(0)})$$  are not defined). 
We see therefore that
$$Blup_{r,s}^+(G,H)= \partial M\times \partial M \times \mathring{\bm{\mathrm{S}}}^1_{+}\cup \mathring{M}\times \mathring{M}
$$
This gives an identification, in the category of smooth manifold with corners, of 
$Blup_{r,s}^+(G,H)$ with the b-stretched product of Melrose with the left and right boundaries removed.

\begin{remark}
	Observe that
	the closed subgroupoid $\partial M\times \partial M \times \mathring{\bm{\mathrm{S}}}^1_{+}$ of $Blup_{r,s}^+(G,H)$ is isomorphic to 
	$\partial M\times \partial M \times \bbR$, that is just the product of the pair groupoid with the additive group $\RR$, through the isomorphism of Lie groupoids   given by
	$$(x,y,[t,s])\mapsto(x,y,\log(t)-\log(s)).$$
	Observe also that the product of $(x,y,[t,s])$ and $(y,z,[u,v])$, that is $(x,z,[tu,vs])$, is sent to
	$$(x,z, \log(tu)-\log(sv))=(x,z, \log(t)+\log(u)-\log(s)-\log(v))=(x,y,\log(t)-\log(s))\cdot(y,z,\log(u)-\log(v)).$$
	
	In this way we recover the Lie groupoid structure that corresponds to the translation invariance of the indicial operator of a b-operator, which is the restriction of the Schwartz kernel of the operator to the front face of the b-streched product of Melrose.\\
\end{remark}

Next, we want to recover through this construction  the ${\rm \Phi}$-double space and the associated groupoid, as explained
in Section \ref{phi-grpd}.
To this end we consider what we have got so far, namely 
$$Blup^+(M\times M,\partial M\times \partial M)\,,\quad Blup_{r,s}^+(M\times M,\partial M\times \partial M)\rightrightarrows M\,,$$
Notice that $\partial M\times_B \partial M \rightrightarrows M$ is a Lie subgroupoid of $Blup_{r,s}^+(M\times M,\partial M\times \partial M)$.  We can first consider 
$$Blup^+ (Blup^+ (M\times M,\partial M\times \partial M), \partial M\times_B \partial M);$$
with arguments similar to those given above, one can see without diffuculty that this space is the ${\rm \Phi}$-double space
considered in Subsection \ref{phi-calculus}.
Let us investigate what happens on the boundary: the inward normal space of $\partial M\times_B \partial M$ in the b-calculus groupoid is 
$$\partial M\times_B \partial M\times_B TB\times\bm{R}\times \bm{R}_+\rightrightarrows B\times \bm{R}_+ $$
where the structure morphisms are given by:
\begin{itemize}
	\item the source map  $s_N(x,y,\xi, u,v)=(y,v)$ and the range map $r_N(x,y,\xi, u,v)=(x,v)$;
	\item  the multiplication $m((x,y,\xi, u,v),(y,z,\eta, w,v))=(x,y,\xi+\eta,u+w,v)$.
\end{itemize} 

Now, as before, we remove the preimage by $s_N$ and $r_N$ of $M\times_B M\times\{0\}$ and we quotient by the action of $\RR^*_+$. 
If we consider the representative elements of the classes in this quotient with $v=1$, then this groupoid is clearly given by $\partial M\times_B TB\times_B\partial M\times \RR$. This leads to the isomorphism  $$Blup_{r,s}^+ (Blup_{r,s}^+ (M\times M,\partial M\times \partial M), \partial M\times_B \partial M)\cong G_{{\rm \Phi}}.$$ 


\begin{remark}
	If $M_\Gamma$ is $\Gamma$ covering of $M$ which is also a manifold with fibered boundary
	$\partial M_\Gamma\to B_\Gamma$, with the $\Gamma$-action satisfying \eqref{compatible-bfs}, then
	we can re-obtain the groupoid  $G^\Gamma_{{\rm \Phi}}$: this is the groupoid
	$$Blup^+_{r,s} \left(Blup^+_{r,s} (M_\Gamma\times M_\Gamma,\partial M_\Gamma\times \partial M_\Gamma), \partial M_\Gamma\times_{B_\Gamma} \partial M_\Gamma\right)/\Gamma.$$
\end{remark}

\subsection{The groupoid associated to a manifold with foliated boundary.}$\;$\\
Keeping with the general philosophy that treating a problem with groupoids solves,
with minor work, a number
of specific geometric problems, we now move to  more singular situations. 
As already observed, when $M$ is a manifold with fibered boundary $\pa M\xrightarrow{\phi} B$, the groupoid 
$Blup^+_{r,s} (Blup^+_{r,s} (M\times M,\partial M\times \partial M), \partial M\times_B \partial M)$, {\it id est}, up to isomorphism,
the groupoid 
$$G_{{\rm \Phi}} = \mathring{M}\times\mathring{M}\cup \partial M\times_B TB\times_B\partial M\times \RR\,,$$
integrates the algebroid ${}^{{\rm \Phi}} TM\to M$ defined by the Lie algebra of fibered boundary vector fields on $M$:
\begin{equation*}
\mathcal{V}_{{\rm \Phi}} (M)=\{\xi\in \mathcal{V}_b (M)\;\;\xi |_{\partial M}\,,\;\;\text{is tangent to the fibers of }\;
\phi:  \pa M\xrightarrow{\phi} B  \text{ and } \xi x\in x^2 C^\infty (M)\}
\end{equation*}
Assume now that  $\pa M$ is foliated 
by $\mathcal{F}$ and 
consider the $\mathcal{F}$-tangent bundle, ${}^\mathcal{F} TM$, defined through the Serre-Swan
theorem starting from the Lie algebra of vector fields
\begin{equation*}
\mathcal{V}_{\mathcal{F}} (M)=\{\xi\in \mathcal{V}_b (M)\;\;\xi |_{\partial M} \in C^\infty (\partial M,T\mathcal{F})\;\; \text{ and } \xi x\in x^2 C^\infty (M)\}
\end{equation*}
See \cite{rochon-foliated}. This defines, an in the previous case, an integrable algebroid and we shall now present a groupoid that integrates it.\\
To this end we first observe that the blow-up construction that we have briefly explained in the previous section
works equally well if $H\rightrightarrows H^{(0)}$ is a Lie groupoid with a (possibly non-injective) immersion $\iota$ into 
$G\rightrightarrows G^{(0)}$; we require the additional property that $\iota$ induces an embedding $H^{(0)}\rightarrow G^{(0)}$.
See \cite[Remark 3.19]{hs-morphismes} \footnote{For example, if $\iota:X\to Y$ is an immersion,
	then we consider 
	the deformation to the normal cone defined via the bundle $N_{X\xrightarrow{\iota} Y}$
	whose fiber at $x\in X$ is $T_{\iota(x)}Y/d\iota (T_x X)$.}. Notice that in this generality already the deformation
to the normal cone, denoted here $DNC^+ (H\xrightarrow{\iota} G)$, could be non-Hausdorff, even if $H$ and $G$ are Hausdorff; 
similarly, the groupoid $Blup_{r,s}^+(H\xrightarrow{\iota} G)\rightrightarrows Blup^+ (G^{(0)}, H^{(0)}) $
is such that $Blup_{r,s}^+(H\xrightarrow{\iota} G)$ is in general non-Hausdorff (whereas, from our additional hypothesis,
we know that $Blup^+ (G^{(0)}, H^{(0)})$ is Hausdorff, as it is required from the definition of  Lie groupoid).
Thanks to the work of Connes, see \cite{connes-survey},  one is nevertheless able to:
\begin{itemize}
	\item define the $C^*$-algebra of such a groupoid 
	\item  consider the associated pseudodifferential
	algebra.
\end{itemize}
We shall come back to these points momentarily.\\
Consider now the holonomy groupoid ${\rm Hol}(\partial M, \mathcal{F})\rightrightarrows \partial M$ of the foliated
manifold $(\partial M, \mathcal{F})$.

\begin{proposition}\label{prop:immersion}
	Let $X$ be a smooth manifold and 
	let $\mathcal{F}$ and $\mathcal{F}'$ be two foliations such that $\mathcal{F}\subset\mathcal{F}'$, then
	there exists an immersion of Lie groupoids
	\[
	\iota\colon{\rm Hol}(X, \mathcal{F})\to {\rm Hol}(X, \mathcal{F}').
	\]
\end{proposition}
\begin{proof}
	Let $x\in X$  and let $L_x$ and $L'_x$ be the leaves of $\mathcal{F}$ and $\mathcal{F}'$ respectively that pass through $x$.
	Let $p_x\colon\widetilde{L}_x\to L_x$ be the ${\rm Hol}(L_x,x)$-covering of $L_x$ and let
	$p'_x\colon\widetilde{L}'_x\to L'_x$ be the ${\rm Hol}(L'_x,x)$-covering of $L'_x$.
	We can lift the map 
	$j_x\circ p_x\colon\widetilde{L}_x\to L'_x$, where $j_x$ is the obvious inclusion,  to a map
	$\iota_x\colon\widetilde{L}_x\to\widetilde{L}'_x$ if 
	$$(j_x\circ p_x)_\sharp(\pi_1(\widetilde{L}_x,x))\subset (p'_x)_\sharp(\pi_1(\widetilde{L}'_x,x))$$
	inside $\pi_1(L'_x)$.
	But this is clear from the commutativity of the following diagram
	\[
	\xymatrix{
		1\ar[r]& \pi_1(\widetilde{L}_x,x)\ar[r]^{(p_x)_\sharp}\ar[d]&\pi_1(L_x,x)\ar[r]\ar[d]^{(j_x)_\sharp}& {\rm Hol}(L_x,x)\ar[r] \ar[d]&1\\
		1\ar[r]& \pi_1(\widetilde{L}'_x,x)\ar[r]^{(p'_x)_\sharp}&\pi_1(L'_x,x)\ar[r]& {\rm Hol}(L'_x,x)\ar[r] &1
	}
	\]
	Observe that the last vertical line is well defined since the group $$ {\rm Hol}(L_x,x)\simeq  {\rm Hol}(L'_x,x)\times  {\rm Hol}_{L'_x}(L_x,x),$$ where ${\rm Hol}_{L'_x}(L_x,x)$ is the holonomy group of the leaf $L_x$ inside $L'_x$; so the last vertical arrow is just the projection to the first factor.
	Now it is easy to check that the collection of $\iota_x\colon{\rm Hol}(X, \mathcal{F})_x\to {\rm Hol}(X, \mathcal{F}')_x $ gives a smooth map $\iota \colon{\rm Hol}(X, \mathcal{F})\to {\rm Hol}(X, \mathcal{F}')$
	and that it is an immersion.
\end{proof}

\begin{corollary}
	Let $(X,\mathcal{F})$ a smooth foliated manifold with $X$ compact.
	Then there exists an immersion of Lie groupoids
	\begin{equation*}
	\iota : {\rm Hol}(X, \mathcal{F}) \rightarrow X\times X
	\end{equation*}
	
\end{corollary}

Given a manifold $M$ with foliated boundary $(\pa M,\mathcal{F})$ as above, we can first take
$Blup^+_{r,s} (M\times M,\partial M\times \partial M)\rightrightarrows M$ and then consider 
\begin{equation}\label{rs-blupphi} Blup^+_{r,s} \left({\rm Hol}(\pa M, \mathcal{F})\xrightarrow{\iota} Blup^+_{r,s} (M\times M,\partial M\times \partial M)\right)\end{equation}
explicitly described, up to a groupoid isomorphism, as:
\begin{equation}\label{rs-blupphi-bis}
\mathring{M}\times\mathring{M}\cup r^*N\mathcal{F}\times \bbR\,,\end{equation}
with $N\mathcal{F}:= T(\pa M)/T\mathcal{F}$ and  $r$ denoting the range map for the holonomy groupoid ${\rm Hol}(\pa M, \mathcal{F})$.
Here the groupoid structure is the pairs groupoid structure on $\mathring{M}\times\mathring{M}$, as usual, and on $ r^*N\mathcal{F}\times \bbR$ is given by 
\begin{equation}\label{groupoid-structure}
(\gamma,\xi, t)\cdot(\gamma', \eta, s)= (\gamma\cdot\gamma',\xi+\gamma(\eta),t+s).
\end{equation}
for $\gamma,\gamma'\in {\rm Hol}(\pa M, \mathcal{F})$,  $\xi\in N_{r(\gamma)}\mathcal{F}$, $\eta\in N_{r(\gamma')}\mathcal{F}$ and $t,s\in\RR$.
Notice that, since a foliation is locally given by a fibration and all the construction we used so far in this section are local, it is easy to prove that the normal bundle of the immersion ${\rm Hol}(\pa M, \mathcal{F})\xrightarrow{\iota} Blup^+_{r,s} (M\times M,\partial M\times \partial M)$ is $ r^*N\mathcal{F}\times \bbR\times \bbR$ and, exactly as for the ${\rm \Phi}$-calculus groupoid, that \eqref{rs-blupphi} is isomorphic to \eqref{rs-blupphi-bis}.


\subsection{The groupoid associated to a foliation degenerating on the boundary}$\;$\\
We can more generally consider a foliated manifold $(M,\mathcal{H})$ with non-empty boundary and with
$\mathcal{H}$ transverse to the boundary. We make the assumption that $\partial M$ is foliated by a foliation
$\mathcal{F}$ such that $\mathcal{F}\subset \mathcal{H}_{| \pa M}$. 
We assume that $M$ has dimension $n$, $\mathcal{H}$ has dimension $p$ and $\mathcal{F}$ has dimension $q$ with
$q\leq p-1$. In the previous subsection we have treated the case $p=n$.

Let $T\mathcal{F}$ be the tangent bundle of $\mathcal{F}$, a subbundle of $T(\pa M)$; let $T\mathcal{H}$ be the 
tangent bundle of $\mathcal{H}$ and consider the Lie algebra of vector fields
\begin{equation*}
\mathcal{V}_{\mathcal{F}} (M,\mathcal{H})=\{\xi\in C^\infty(M,T\mathcal{H})\cap \mathcal{V}_b (M),
\;\;\xi |_{\partial M} \in C^\infty (\partial M,T\mathcal{F})\;\; \text{ and } \xi x\in x^2 C^\infty (M)\}
\end{equation*}
This is a finitely generated projective $C^\infty (M)$-module and thus, according to Serre-Swan theorem, there
exists a vector bundle ${}^\mathcal{F} T\mathcal{H}$ whose sections are precisely given by 
$\mathcal{V}_{\mathcal{F}} (M,\mathcal{H})$.
Recall, see  \cite[Section 5]{hilsum-boundary}, that ${\rm Hol} (\partial M,\mathcal{H}_{| \partial M})$
is a closed and open subgroupoid of ${\rm Hol} ( M,\mathcal{H})_{| \partial M}$ and, by Proposition
\ref{prop:immersion} above, that there exists an immersion
$${\rm Hol} (\partial M,\mathcal{F})\xrightarrow{j} {\rm Hol} (\partial M,\mathcal{H}_{| \partial M})$$
that, together, give an immersion
$$\iota: {\rm Hol} (\partial M,\mathcal{F})\rightarrow {\rm Hol} ( M,\mathcal{H})_{| \partial M}.$$
We can thus consider 
$Blup^+_{r,s} \left( {\rm Hol} (M,\mathcal{H}), {\rm Hol} ( M,\mathcal{H})_{| \partial M} \right)\rightrightarrows M$ and then 
$$ Blup^+_{r,s} \left({\rm Hol}(\pa M, \mathcal{F})\xrightarrow{\iota} Blup^+_{r,s} \left( {\rm Hol} (M,\mathcal{H}), {\rm Hol} ( M,\mathcal{H})_{| \partial M} \right)\right)$$
which is  described as:
$$ {\rm Hol} (M,\mathcal{H})_{| \mathring{M}}\cup r^*N_\mathcal{F}^\mathcal{H}\times \bbR $$ with $N_\mathcal{F}^\mathcal{H}:= T\mathcal{H}_{|\pa M}/ T\mathcal{F}$  and  $r$ denoting the range map for the holonomy groupoid ${\rm Hol}(\pa M, \mathcal{F})$.
Analogously to the previous cases, the groupoid structure is given by that one of the holonomy groupoid in the interior and by
\eqref{groupoid-structure} on the boundary.

\begin{remark}
	We can equally make these constructions with the monodromy groupoid of our foliations, i.e. with
	${\rm Mon} (\pa M,\mathcal{F})$, ${\rm Mon} (M,\mathcal{H})$. In this case the analogue of Proposition 
	\ref{prop:immersion} already appears in \cite[Proposition 6.2]{MM}. 
\end{remark}

\subsection{(Pseudo)-differential operators, K-theory classes and positive scalar curvature}
Let $(M,\mathcal{H})$ be a foliated manifold with boundary with the foliation $\mathcal{H}$ degenerating 
into $(\pa M,\mathcal{F})$ at the boundary. 
The vector fields in $\mathcal{V}_{\mathcal{F} }(M,\mathcal{H})$, together with $C^\infty (M)$, generate an algebra of differential operators
on $M$ that we denote $\Diff^*_\mathcal{F} (M,\mathcal{H})$. In the case of $M$ being a manifold with foliated boundary
we use the notation $\Diff^*_\mathcal{F} (M)$ as in \cite{rochon-foliated}. If $D\in \Diff^*_\mathcal{F} (M,\mathcal{H})$ then $D$ lifts uniquely to a differential operator on the groupoid 
\begin{equation}\label{group-fol}
G_{\mathcal{F}} (M,\mathcal{H}):={\rm Hol} (M,\mathcal{H})_{| \mathring{M}}\cup r^*N_\mathcal{F}^\mathcal{H}\times \bbR 
\end{equation}
with
$N_\mathcal{F}^\mathcal{H}:= T\mathcal{H}_{|\pa M}/ T\mathcal{F}$  and  $r$ denoting the range map for the holonomy groupoid ${\rm Hol}(\pa M, \mathcal{F})$. With a customary abuse of notation we keep the same notation for this lift.
The normal operator $N (D)$ is, by definition, the restriction of (the lift) of the operator to $r^*N_\mathcal{F}^\mathcal{H}\times \bbR $.
This algebra of differential operators is of course a subalgebra of $\Psi^* (G_{\mathcal{F}} (M,\mathcal{H}))$ where,
as in Section \ref{sect:dirac-k}, we have enlarged the algebra of compactly supported operators pseudodifferential operators, $\Psi^*_c (G_{\mathcal{F}} (M,\mathcal{H}))$,
by the addition of the smoothing operators as in \cite{Vassout-jfa}.

We now proceed to define all the objects that have been defined previously in the case of a manifold with fibered boundary
and see how to get interesting K-theory classes. We shall be brief.

We consider the algebra $\Psi^0 (G_{\mathcal{F}} (M,\mathcal{H}))$.
Tangential ellipticity for elements in this algebra  is defined as usual, in terms of ${}^\mathcal{F} T\mathcal{H}$, whereas  full ellipticity is defined by requiring the normal operator to be invertible where invertibility is meant in the algebra $\Psi^0 (G_{\mathcal{F}} (M,\mathcal{H})_{|_{\pa M}})\equiv \Psi^0 (r^*N_\mathcal{F}^\mathcal{H}\times \bbR )$. Similar definitions can be given for
an operator of positive order. 
In order to define the fundamental class and the index class of a fully elliptic operator as well as the rho class 
of an invertible operator,
we 
introduce the relevant groupoids and associated $C^*$-algebras:

\begin{itemize} 
	\item  the adiabatic deformation $G_{\mathcal{F}} (M,\mathcal{H})_{ad}^{[0,1]}={}^\mathcal{F} T\mathcal{H}\times \{0\}\cup G_{\mathcal{F}} (M,\mathcal{H})\times [0,1] $ and its restriction to $M\times[0,1)$, denoted $G_{\mathcal{F}} (M,\mathcal{H})_{ad}^{[0,1)}$;
	\item the noncommutative tangent bundle ${}^\mathcal{F} T\mathcal{H}^{NC}$ defined as  the restriction of $G_{\mathcal{F}} (M,\mathcal{H})_{ad}^{[0,1]} $ to $\mathring{M}\times\{0\}\cup \pa M\times[0,1)$; it is equal to ${}^\mathcal{F} T\mathcal{H}\cup r^*\mathcal{N}\times\RR\times(0,1)$;
	\item the exact sequence of $C^*$-algebras:
	\begin{equation}\label{exact-foliated}
	\xymatrix{0\ar[r]& C^*(G_{\mathcal{F}} (M,\mathcal{H})_{|\mathring{M}}\otimes C_0(0,1))\ar[r]& C^*( G_{\mathcal{F}} (M,\mathcal{H})_{ad}^{[0,1)})\ar[r] &
		C^*({}^\mathcal{F} T\mathcal{H}^{NC})\ar[r]&0}.
	\end{equation}
\end{itemize}

Let $g$ be a foliated metric on $(\mathring{M},\mathcal{H})$, i.e. a metric on $T\mathcal{H}_{| \mathring{M}}$. We shall say that
$g$ is an {\it admissible} metric 
if $g$ extends to a metric 
on  the Lie algebroid ${}^\mathcal{F} T\mathcal{H}$. 

We shall 
work with special admissible metrics, that we call, as usual, {\it rigid} and that we proceed to define. Recall that the dimension of the leaves of $\mathcal{H}$ is $p$ and that
the dimension of the leaves of $\mathcal{F}$ is $q$, with $q\leq p-1$.
Let $U$ be a distinguished neighbourhood  for $(M,\mathcal{H})$ with $U\cap \pa M\not= \emptyset$.
We can assume that $U$ is homeomorphic to  $[0,1)_x\times F^q \times T^{p-q-1} \times S^{n-p}$ where all these
sets are open sets in euclidean spaces of the right dimension.
We shall also briefly write $U\simeq [0,1)_x\times F \times T \times S$.
Notice that in this notation
$T^{p-q-1}$ is a local transversal for $\mathcal{F}$ inside $\mathcal{H}_{| \pa M}$ and $S^{n-p}$ is a local transversal
for $\mathcal{H}$.\\
We shall say that the admissible metric $g$ is rigid if for each distinguished chart as above
$g_{| \mathring{U} } $ is given by a family of metrics $g(s)$ that can be written as 
$$g(s)=\frac{dx^2}{x^4}  + \frac{g_T (s)}{x^2}+ g_F (t,s) \,.$$
Notice that $g$ defines in a natural way a foliated metric $g_{\mathcal{F}}$ on $(\pa M,\mathcal{F})$,
locally given by the $g_F (t,s)$.

When necessary, we denote by $g_{\mathcal{F}}^{\mathcal{H}}$
an admissible rigid metric. If we are considering the foliated boundary case, then we employ the 
notation $g_{\mathcal{F}}$ for an admissible rigid metric in this context. The simple notation $g$ will also be used.

\begin{remark}
Whereas  a stratified space always carries a rigid  fibered corner metric $g_{\Phi}$, see \cite[Proposition 3.1]{ALMP:Witt},
 it is not true in general that the class of admissible rigid metrics $g_{\mathcal{F}}^{\mathcal{H}}$ is non-empty.
Indeed let $L$ be a leaf of the restriction of $\mathcal{H}$ to the boundary and notice that $L$ is foliated by $\mathcal{F}$, then the restriction of $g$ to $L$ would be a bundle-like metric with respect to $\mathcal{F}$ and there are examples of foliated manifolds that do not admit bundle-like metrics.
\end{remark}

Let us now assume that $ T\mathcal{H}$ and $ T\mathcal{F}$ have a spin structure. Then we can consider the spin Dirac $G_{\mathcal{F}} (M,\mathcal{H})$-operator $\slashed{D}$ associated to $g$, as in Remark \ref{dirac-on-groupoids}.

If the induced metric  $g_{\mathcal{F}}$  has positive scalar curvature along the leaves of $\mathcal{F}$, then as in 
Proposition \ref{fullellipticity}, we can prove that $N(\slashed{D})$ is invertible and so that $\slashed{D}$ is fully elliptic.
Following the general procedure used in the previous sections we can then define the following classes and prove the 
properties of these listed below

\begin{itemize}
	\item the noncommutative symbol $$\sigma_{nc} (\slashed{D})\in K_* (C^*({}^\mathcal{F} T\mathcal{H}^{NC}))$$
	defined as in Definition \ref{sigmanc};
	\item the (adiabatic) index class $$\Ind^{ad} (\slashed{D})\in K_* (C^*(G_{\mathcal{F}} (M,\mathcal{H})_{|\mathring{M}}))$$
	obtained as the image of $\sigma_{nc} (\slashed{D})$ through the connecting homomorphism associated to \eqref{exact-foliated};
	observe that this homomorphism has an explicit form as in \eqref{index-homo-nc};
	\item the index class given by the connecting homomorphism associated to the analogue of \eqref{f.e.extention};
	\item the equality of these two index classes, proved as in Theorem \ref{equality-index-classes};
	\item if $g$ has positive scalar curvature everywhere along the leaves of $(\mathring{M},\mathcal{H})$ then we can define a rho class $$\rho (g)\in K_* (C^*(G_{\mathcal{F}} (M,\mathcal{H})_{ad}^{[0,1)}))$$
	defined as in Definition \ref{rho-definition};
	\item if  $\mathcal{R}^+_{\mathcal{F}}(M,\mathcal{H})$ denotes the set of 
	 admissible rigid metrics of positive scalar curvature and $\tilde{\pi}_0 (\mathcal{R}^+_{\mathcal{F}}(M,\mathcal{H}))$
	 the associated set of concordance classes.  then the rho class gives a well defined map
	$$\rho: \tilde{\pi}_0 (\mathcal{R}^+_{\mathcal{F}}(M,\mathcal{H}))\to K_*( C^*(G_{\mathcal{F}} (M,\mathcal{H})_{ad}^{[0,1)})).$$ 
	This is proved as in Theorem \ref{deloc-aps} following the delocalized APS index theorem for 
	general Lie groupoids established in \cite[Theorem 3.6]{Zenobi:Ad}

\end{itemize}

\section{Appendix: non-Hausdorff groupoids}
For the benefit of the reader, let us discuss in some detail the non Hausdorff case. At first glance it would seem
rather involved, but in fact it is not.
To get a clear intuition of what happens when we blow-up a non injective immersion, it is enough consider the local situation, in particular we are going to dissect the following  very simple situation. 
Let us consider the immersion of Lie groupoids
$$\iota\colon\bbZ_2\to \bbR\times\bbR$$
which sends every element of the group $\bbZ_2$ to $(0,0)$.
The deformation to the normal cone of this immersion is
$$\iota^*\mathcal{N}_{\iota(\bbZ_2)}^{\bbR\times\bbR}\cup \bbR\times\bbR\times\bbR^*$$
that is noting else than the quotient of
$  \bbZ_2\times DNC(\bbR\times\bbR,\{ (0,0)\})$
by the equivalence relation such that $([0],(x,y),t)\sim ([1],(x,y),t)$ for $t\neq 0$.
As a set it is equal to $$ \bbZ_2\times \mathbf{R^2}\times\{0\}\cup \bbR\times\bbR\times \bbR^*. $$
Notice that the non-Hausdorff smooth structure is quotient structure induced by that of the two copies of $DNC(\bbR\times\bbR, \{ (0,0)\})$ and this one is given by imposing that
the map $\psi\colon  \bbR\times\bbR\times\bbR\to DNC(\bbR\times\bbR, \{ (0,0)\})$
$$
\psi(x,y,t)=\begin{cases}(x,y,0)\,\quad\mathrm{if}\, t=0\\
(x,ty,t) \quad\mathrm{if}\, t\neq0 \end{cases}
$$
is smooth.
Moreover it is easy to prove that the Lie algebroid of $DNC(\bbZ_2\xrightarrow{\iota}\bbR\times\bbR)$ is isomorphic to 
$\bbR\times\mathbf{R}$ and that the module of its sections is generated by the vector field $t\frac{\partial}{\partial t}$ over $\bbR$.

Now if we proceed to the $Blup$ construction, we have to remove suitable subsets and then quotient by the action of $\bbR^*$, as explained in \ref{sing-fol}. Actually we are going to restrict our attention to the $Blup^+_{r,s}$ construction and we obtain the following non-Hausdorff Lie groupoid
$$\bbZ_2\times\mathring{\mathbf{S}}^1\cup\bbR^*_+\times\bbR^*_+\equiv\bbZ_2\times\bbR\cup\bbR^*_+\times\bbR^*_+$$
that is nothing but the quotient of $\bbZ_2\times Blup^+_{r,s}(\bbR\times\bbR, \{ (0,0)\})$ by the equivalence relation such that
$([0], (x,y))\sim([1], (x,y))$ for all $(x,y)\in \bbR^*_+\times\bbR^*_+$, with the induced smooth structure just as for the deformation to the normal cone.

Notice that, following \cite[Section 2.2.2]{these-month}, since the  pseudodifferential operators are defined, up to elements in the C*-algebra, as conormal distribution  supported near the diagonal, it turns out that 
$\Psi_c(\bbZ_2\times\bbR\cup\bbR^*_+\times\bbR^*_+)$ is equal to $\Psi_c(\{[0]\}\times\bbR\cup\bbR^*_+\times\bbR^*_+)\equiv\Psi_c(Blup^+_{r,s}(\bbR\times\bbR, \{ (0,0)\})$ up to  elements in $C^*(\bbZ_2\times\bbR\cup\bbR^*_+\times\bbR^*_+)$.

In other words, since the objects and the $s$-fibers are Hausdorff manifolds, if we see a pseudodifferential operator on the groupoid as a family of equivariant pseudodifferential operators on the $s$-fibers, parametrized by the manifold of the objects, the property of being not Hausdorff of the whole groupoid is not involved in this definition.


\bibliography{gsssido}
\bibliographystyle{amsalpha}

\end{document}